\documentclass[a4paper,twoside,11pt]{article}

\usepackage{BeaCed_article}

\title{The $Z$-invariant Ising model via dimers}
\author{C\'edric Boutillier\thanks{{\small
Laboratoire de Probabilit\'es Statistiques et Mod\'elisation, Sorbonne Universit\'e, 4 place Jussieu, 
F-75005 Paris.} {\small Email:
  \href{mailto:cedric.boutillier@sorbonne-universite.fr}{\texttt{cedric.boutillier@sorbonne-universite.fr}}}
}\and B\'eatrice de Tili\`ere\thanks{{\small
Laboratoire d'Analyse et de Math\'ematiques Appliqu\'ees, Universit\'e Paris-Est Cr\'eteil, 61 avenue du G\'en\'eral de Gaulle,
F-94010 Cr\'eteil.}
{\small Email:
  \href{mailto:beatrice.taupinart-de-tiliere@u-pec.fr}{\texttt{beatrice.taupinart-de-tiliere@u-pec.fr}}}
} \and Kilian Raschel\thanks{{\small
CNRS, Institut Denis Poisson, Universit\'e de Tours, Parc de Grandmont, F-37200 Tours.}
{\small Email: \href{mailto:raschel@math.cnrs.fr}{\texttt{raschel@math.cnrs.fr}}}}
\thanks{This project has received funding from the European Research Council (ERC) under the EuropeanUnion's Horizon 2020 research and innovation programme under the Grant Agreement No.\ 759702.}
}

\date{\today}

\begin{document}

\maketitle

\begin{abstract}
The $Z$-invariant Ising model~\cite{Baxter:8V} is defined on an \emph{isoradial} graph and has coupling constants 
depending on an elliptic parameter $k$. When $k=0$ the model is critical, and as $k$ varies
the whole range of temperatures is covered. In this paper we study the corresponding dimer model on the Fisher graph, thus
extending our papers~\cite{BoutillierdeTiliere:iso_perio,BoutillierdeTiliere:iso_gen} to the \emph{full} $Z$-invariant case.
One of our main results is an explicit, \emph{local} formula for the inverse of the Kasteleyn operator. 
Its most remarkable feature is that it is an elliptic generalization of~\cite{BoutillierdeTiliere:iso_gen}:
it involves a local function and the massive discrete exponential function introduced in~\cite{BdTR1}. This shows in particular
that $Z$-invariance, and not criticality, is at the heart of obtaining local expressions. We then compute asymptotics and deduce an explicit, local expression 
for a natural Gibbs measure. We prove a local formula for the Ising model free
energy.
We also prove that this free energy is equal, up to constants, to that of the $Z$-invariant spanning forests of~\cite{BdTR1}, and deduce
that the two models have the same second order phase transition in $k$. Next, we prove a self-duality relation for this model, extending 
a result of Baxter to all isoradial graphs. In the last part we prove explicit, local expressions for the dimer model on a bipartite graph 
corresponding to the XOR version of this $Z$-invariant Ising model. 
\end{abstract}

\section{Introduction}
\label{sec:intro}

The $Z$-invariant Ising model, fully developed by Baxter~\cite{Baxter:8V,Baxter:Zinv,Baxter:exactly}, takes its roots in the work of Onsager~\cite{Onsager,Wannier},
see also~\cite{Perk2,Perk:YB,Martinez1,Martinez2,CostaSantos} for further developments in the physics community. It is defined on a planar, embedded graph $\Gs=(\Vs,\Es)$ satisfying a geometric constraint known as \emph{isoradiality}, 
imposing that all faces are inscribable in a circle of radius $1$. In this introduction, the graph $\Gs$ is assumed to be infinite and locally finite.
The \emph{star-triangle move} (see Figure \ref{fig:star_triangle}) preserves isoradiality; it transforms a three-legged star of the graph into
a triangle face. The Ising model is said to be \emph{$Z$-invariant} if, when decomposing the partition function according to the possible spins at vertices bounding the 
triangle/star, the contributions only change by an overall constant. This constraint imposes that the coupling constants 
$\Js=(\Js_e)_{e\in\Es}$ satisfy the \emph{Ising model Yang-Baxter equations}. The solution to these equations is parametrized by angles naturally assigned 
to edges in the isoradial embedding of the graph $\Gs$, and an \emph{elliptic parameter} $k$, with $k^2\in(-\infty,1)$:
\begin{equation*}
\forall\,e\in\Es,\quad
\Js_e=\Js(\overline{\theta}_e|k)=
\frac{1}{2}\log\Bigl(
  \frac{1+\sn(\theta_e|k)}{\cn(\theta_e|k)}
\Bigr),
\end{equation*}
where $\sn$ and $\cn$ are two of the twelve \emph{Jacobi trigonometric elliptic functions}. More details and precise references are to be found in 
Section~\ref{sec:Z_inv_versions}. When $k=0$, the elliptic functions $\sn,\cn$ degenerate to the usual trigonometric functions $\sin,\cos$
and one recovers the \emph{critical} $Z$-invariant Ising model, whose criticality is proved in~\cite{Li:critical,CimasoniDuminil,Lis}. Note that the coupling 
constants range from $0$ to $\infty$ as $k$ varies, thus covering the whole range of temperatures, see Lemma~\ref{lem:poids_croissants}.

A fruitful approach for studying the planar Ising model is to use Fisher's correspondence~\cite{Fisher} relating it to the dimer model on a decorated version $\GF$
of the graph $\Gs$, see for example the book~\cite{McCoyWu}. The dimer model on the Fisher graph arising from the \emph{critical} $Z$-invariant Ising model 
was studied by two of the present authors in~\cite{BoutillierdeTiliere:iso_perio,BoutillierdeTiliere:iso_gen}. One of the main goals of this
paper is to prove a generalization to the \emph{full} $Z$-invariant Ising model of the latter results.
Furthermore, we answer questions arising when the parameter $k$ varies. 
In the same spirit, we also solve the bipartite dimer model on the graph $\GQ$ associated to two independent $Z$-invariant 
Ising models~\cite{Dubedat,BoutillierdeTiliere:XORloops} and related to the XOR-Ising model~\cite{KadanoffBrown,WilsonXOR}. In order to explain the main features of our results, 
we now describe them in more details.

The \emph{Kasteleyn matrix/operator} \cite{Kasteleyn1,TF} is the key object used to obtain explicit expressions 
for quantities of interest in the dimer model, as 
the partition function, the Boltzmann/Gibbs measures and the free energy. It is a weighted, oriented, adjacency matrix of the dimer graph.
Our first main result proves an explicit, local expression for an inverse $\KF^{-1}$ of the Kasteleyn operator $\KF$ of the dimer model on 
the Fisher graph $\GF$ arising from the $Z$-invariant Ising model; it can loosely be stated as follows, see 
Theorem~\ref{thm:KFmoins_un} for a more precise statement. 
\begin{thm}\label{thm:main1_intro}
Define the operator ${\KF}^{-1}$ by its coefficients:
\begin{equation*}
\forall\,\xs,\ys\in\VF,\quad     \Ks^{-1}_{\xs,\ys}=\frac{ik'}{8\pi}\int_{\Gamma_{\xs,\ys}} \fs_\xs(u+2K)\fs_{\ys}(u)\expo_{(\xb,\yb)}(u)\ud u+C_{\xs,\ys},
\end{equation*}
where $\fs$ and $\expo$, see \eqref{equ:rewriting_fs} and \eqref{eq:recursive_def_expo}, respectively, are elliptic functions defined on the torus $\TT(k)$, whose aspect ratio depends on $k$.
The contour of integration $\Gamma_{\xs,\ys}$ is a simple closed curve winding
once vertically around $\TT(k)$, which intersects the horizontal axis away from the poles of the integrand; the constant $C_{\xs,\ys}$ is equal to  $\pm1/4$ when 
$\xs$ and $\ys$ are close, and $0$ otherwise, see \eqref{eq:expression_C_x_y}. 

Then ${\KF}^{-1}$ is an inverse of the Kasteleyn operator $\KF$ on $\GF$. When $k\neq 0$, it is the unique inverse with bounded coefficients.
\end{thm}

\begin{rem}$\,$
\begin{itemize}
 \item The expression for $\KF^{-1}_{\xs,\ys}$ has the remarkable feature of being \emph{local}. This property is inherited from 
 the fact that 
 the integrand, consisting of the function $\fs$ and the massive discrete exponential function, is itself \emph{local}: it is defined through a path joining two
 vertices corresponding to $\xs$ and $\ys$ in the isoradial graph $\Gs$. This locality property is unexpected when computing inverse operators in
 general.
 \item As for the other local expressions proved for inverse operators~\cite{Kenyon3,BoutillierdeTiliere:iso_gen,BdTR1},
 Theorem~\ref{thm:KFmoins_un} has the following interesting features: there is no periodicity assumption on the isoradial graph $\Gs$, the integrand has identified poles 
 implying that explicit computations can be performed using the residue theorem (see Appendix~\ref{app:explicit_integrals}), 
 asymptotics can be obtained via a saddle-point analysis (see Theorem~\ref{thm:asymptotics_inverse_Kasteleyn}). 
 
 \item The most notable feature is that Theorem~\ref{thm:KFmoins_un} is a generalization to the elliptic case of Theorem 1 
 of~\cite{BoutillierdeTiliere:iso_gen}. Let us explain why it is not evident that such a generalization should exist.
 Thinking of $Z$-invariance from a probabilist's point of view suggests
 that there should exist local expressions for probabilities. The latter are computed using the Kasteleyn operator $\KF$ and its 
 inverse, suggesting that there should exist a local expression for the inverse operator $\KF^{-1}$, but giving no tools for finding it. Until our recent paper~\cite{BdTR1}, local expressions for inverse operators were only proved for 
 critical models~\cite{Kenyon3,BoutillierdeTiliere:iso_gen}, leading to the belief that not only $Z$-invariance but also
 \emph{criticality} played a role in the existence of the latter. Another difficulty was that some key tools were missing. 
 We believed that if a local expression existed in the non-critical case, it should be an elliptic version of the one of the critical case, thus requiring
 an elliptic version of the discrete exponential function of~\cite{Mercat:exp}, which was unavailable. 
 This was our original motivation for the paper~\cite{BdTR1} introducing
 the \emph{massive discrete exponential function} and the \emph{$Z$-invariant massive Laplacian}. 
 The question of solving the dimer representation of the full $Z$-invariant Ising model turned out to be more intricate than expected, 
 but our original intuition of proving an elliptic version of the critical results turns out to be correct.
 \item On the topic of locality of observables for critical $Z$-invariant models,
 let us also mention the paper~\cite{GrimmettManolescu2} by Manolescu and Grimmett, recently extended to the random cluster 
 model~\cite{Manolescu}.
 Amongst other results, the authors prove the universality of typical critical
 exponents and Russo-Seymour-Welsh type estimates. 
 The core of the proof consists in iterating star-triangle moves in order to
 relate different lattices. This is also the intuition
 behind locality in $Z$-invariant models: if these critical exponents were
 somehow related to inverse operators (which could maybe be true for the $q=2$
 case), then one would expect local expressions for these inverses.
\end{itemize}
\end{rem}

In Theorem~\ref{thm:Gibbs_measure}, using the approach of~\cite{deTiliere:quadri}, see also~\cite{BoutillierdeTiliere:iso_gen}, we prove an explicit, local expression for 
a Gibbs measure on dimer configurations of the Fisher graph, involving the operator $\KF$ and the inverse $\KF^{-1}$ of 
Theorem~\ref{thm:main1_intro}. This allows to explicitly compute probability of edges in polygon configurations of the low or high 
temperature expansion of the Ising model, see Equation~\eqref{ex:proba_comput}.

Suppose now that the isoradial graph $\Gs$ is $\ZZ^2$-periodic, and let $\Gs_1=\Gs/\ZZ^2$ be the fundamental domain. 
Following an idea of~\cite{Kenyon3} and using the explicit expression of Theorem~\ref{thm:main1_intro}, we prove an explicit formula for the 
free energy of the $Z$-invariant Ising model, see also Corollary~\ref{cor:free_energy_Ising}.
 This expression is also \emph{local} in the sense
that it decomposes as a sum over edges of the fundamental domain $\Gs_1$. 
A similar expression is obtained by Baxter \cite{Baxter:8V,Baxter:exactly}, see Remark \ref{rem:feb} for a comparison between the two.
\begin{thm}
The free energy $\Fising^k$ of the $Z$-invariant Ising model is equal to:
\begin{multline*}
\Fising^k=-\vert\Vs_{1}\vert\frac{\log 2}{2}-\vert\Vs_{1}\vert\int_{0}^{K} 2{\Hh}'(2\theta\vert k)\log\sc(\theta\vert k)\,\ud\theta
\\+\sum_{e\in\Es_1}\left(
-\Hh(2\theta_e\vert k)\log\sc(\theta_e\vert k) +\int_{0}^{\theta_e} 2{\Hh}'(2\theta\vert k)\log(\sc\theta\vert k)\,\ud\theta
\right),
\end{multline*}
where $\sc=\frac{\sn}{\cn}$ and the function $H$ is defined in~\eqref{eq:definition_Hh_Hv_k2>0}--\eqref{eq:definition_Hh_Hv_k2<0}.
\end{thm}

It turns out that the free energy of the Ising model is closely related to that of the $Z$-invariant spanning forests of~\cite{BdTR1},
see also Corollary~\ref{cor:link_free_energies}.
\begin{cor}\label{cor:free_energy_intro}
One has
\begin{equation*}
\Fising^k=-|\Vs_1|\frac{\log 2}{2}+\frac{1}{2}\Fforest^k. 
\end{equation*}
\end{cor}
This extends to the full $Z$-invariant Ising model the relation proved in the critical case~\cite{BoutillierdeTiliere:iso_gen} between the Ising model 
free energy and that of critical spanning trees of~\cite{Kenyon3}. Moreover,
in~\cite{BdTR1} we prove a continuous (i.e., \emph{second order}) phase transition at $k=0$ for 
$Z$-invariant spanning forests, by performing an expansion of the free energy
around $k=0$: at $k=0$, the free energy is continuous, but its derivative has a
logarithmic singularity.
As a consequence of Corollary~\ref{cor:free_energy_intro} we deduce that the $Z$-invariant Ising model has a 
second order phase transition at $k=0$ as well. This result in itself is not
surprising and other techniques, such as those
  of~\cite{DuminilHonglerNolin} and the fermionic
  observable~\cite{ChelkakSmirnov:ising}
could certainly be used in our setting too to derive this kind of result; but what is remarkable is that this phase transition is (up to a factor $\frac{1}{2}$) exactly 
the same as that of
$Z$-invariant spanning forests. More details are to be found in Section~\ref{sec:phase_transition}. 

It is interesting to note that the $Z$-invariant Ising model satisfies a duality relation in the sense of Kramers and 
Wannier~\cite{KramersWannier1,KramersWannier2}: the high temperature expansion of a $Z$-invariant Ising model with elliptic parameter $k$ 
on an isoradial graph $\Gs$, and the low temperature expansion of a $Z$-invariant Ising model with \emph{dual} elliptic parameter 
$k^*=i\frac{k}{\sqrt{1-k^2}}$ on the dual isoradial graph $\Gs^*$ yield the same probability measure on polygon configurations of the graph $\Gs$.
The elliptic parameters $k$ and $k^*$ can be interpreted as parametrizing dual temperatures, see Section~\ref{sec:duality} and also~\cite{CGNP,McCM}.

The next result proves a self-duality property for the Ising model free energy, see also Corollary~\ref{cor:free_energy_self_dual}. 
This is a consequence of Corollary~\ref{cor:free_energy_intro} and of Lemma~\ref{lem:Laplacian_proportional}, proving a self-duality property 
for the $Z$-invariant massive Laplacian. 
\begin{cor}
\label{cor:free_energy_self_dual_intro}
The free energy of the $Z$-invariant Ising model on the graph $\Gs$ satisfies the self-duality relation
\begin{equation*}
     \Fising^k+\frac{\vert\Vs_1\vert}{2}\log k'=\Fising^{k^*}+\frac{\vert\Vs_1\vert}{2}\log {k^{*}}',
\end{equation*}
where $k'=\sqrt{1-k^2}$ is the complementary elliptic modulus, and ${k^{*}}'=1/k'$.
\end{cor}

The above result extends to all isoradial graphs a self-duality relation proved by Baxter~\cite{Baxter:exactly} in the case of the triangular and honeycomb lattices.
Note that this relation and the assumption of uniqueness of the critical point was the argument originally used to derive the critical temperature of the 
Ising model on the triangular and honeycomb lattices, see also Section~\ref{subsec:self_duality_Ising}.

In Section~\ref{sec:double_Ising} we consider the dimer model on the graph $\GQ$ associated to two independent $Z$-invariant Ising models. This dimer model 
is directly related to the XOR-Ising model~\cite{Dubedat,BoutillierdeTiliere:XORloops}. Our main result is to prove an explicit, local expression for the 
inverse $\KQ^{-1}$ of the Kasteleyn operator associated to this dimer model. This is a generalization, in the specific case of the bipartite graph $\GQ$,
of the local expression obtained by Kenyon~\cite{Kenyon3} for all ``critical'' bipartite dimer models. 
\begin{thm}
Define the operator ${\KQ}^{-1}$ by its coefficients:
\begin{equation*}
\forall\ (b,w)\in\BQ\times\WQ,\quad
    {\KQ}^{-1}_{b,w}=\frac{1}{4i\pi} \int_{\Gamma_{b,w}} f_{(b,w)}(u) \ud u,
  \end{equation*}
  where $f_{(b,w)}$ is an elliptic function defined on the torus $\TT(k)$,
  defined in Section~\ref{sec:functions_kernel}.
  The contour $\Gamma_{b,w}$ is a simple closed curve winding once vertically
  around $\TT(k)$, which intersects the horizontal axis away from the poles of
  the integrand.

  Then ${\KQ}^{-1}$ is an inverse operator of $\KQ$. For $k\neq 0$, it is the
  only inverse with bounded coefficients.
\end{thm}

We also derive asymptotics and deduce an explicit, local expression for a Gibbs
measure on dimer configurations of $\GQ$, allowing to do explicit probability 
computations.

\paragraph{Outline of the paper.}

\begin{itemize}
 \item \textbf{Section~\ref{sec:themodelsinquestion}.} Definition of the Ising model, of the two corresponding dimer models 
 and of their $Z$-invariant versions. Definition of the $Z$-invariant massive Laplacian of~\cite{BdTR1}.
 \item \textbf{Section~\ref{sec:Ising_dimers}.} Study of the $Z$-invariant Ising model on $\Gs$ via the dimer model on the Fisher graph
 $\GF$ and the corresponding Kasteleyn operator $\KF$: definition of a one-parameter family of functions in the kernel of $\KF$, statement 
 and proof of a local formula for an inverse $\KF^{-1}$, explicit computation of asymptotics, specificities when the graph $\Gs$ is periodic 
 (connection with the massive Laplacian), and consequences for the dimer model on $\GF$. 
 \item \textbf{Section~\ref{sec:duality_and_phase_transition}.} Behavior of the model as the parameter $k$ varies: duality in the sense of Kramers
 and Wannier~\cite{KramersWannier1,KramersWannier2}, phase transition in $k$, self-duality property, connection with the modular group.
 \item \textbf{Section~\ref{sec:double_Ising}.} Study of the double $Z$-invariant Ising model on $\Gs$ via the dimer model 
 on the bipartite graph $\GQ$ and the Kasteleyn operator $\KQ$: one-parameter family of functions in the kernel of $\KQ$, statement and proof 
 of a local formula for an inverse $\KQ^{-1}$, explicit computation of asymptotics and consequences for the dimer model on $\GQ$.
\end{itemize}

\medskip

\emph{Acknowledgments:} 
We acknowledge support from the Agence Nationale de la Recherche (projet  MAC2:\ ANR-10-BLAN-0123) and from the R\'egion Centre-Val de Loire 
(projet MADACA). We are grateful to the referee for his/her many insightful comments. 

\section{The models in question}\label{sec:themodelsinquestion}
\subsection{The Ising model via dimers}
In this section we define the Ising model and two of its dimer
representations. The first is Fisher's correspondence~\cite{Fisher} providing a mapping
between the high or low temperature expansion of the Ising model on a graph $\Gs$
and the dimer model on a non-bipartite graph $\GF$. The second is a mapping between
two independent Ising models on $\Gs$ and the dimer model on a bipartite
graph $\GQ$~\cite{Dubedat,BoutillierdeTiliere:XORloops}.

\subsubsection{The Ising model}

Consider a finite, planar graph $\Gs=(\Vs,\Es)$ together with positive edge-weights
$\Js=(\Js_e)_{e\in\Es}$. The \emph{Ising model on $\Gs$ with coupling constants 
$\Js$} is defined as follows. A \emph{spin configuration} $\sigma$ of $\Gs$ is a function on vertices
of $\Gs$ with values in $\{-1,1\}$. The probability of occurrence of a spin configuration $\sigma$
is given by the \emph{Ising Boltzmann measure}, denoted $\PPising$:
\begin{equation*}
\PPising(\sigma)=\frac{1}{\Zising(\Gs,\Js)}\exp 
\Biggl(\sum_{e=\xb\yb\in\Es}\Js_e\sigma_\xb\sigma_\yb\Biggr),
\end{equation*}
where $\Zising(\Gs,\Js)$ is the normalizing constant known as the \emph{Ising partition function}.

A \emph{polygon configuration} of $\Gs$ is a subset of edges such that every vertex has even degree; let $\P(\Gs)$ denote 
the set of polygon configurations of $\Gs$. Then, the \emph{high temperature expansion}~\cite{KramersWannier1,KramersWannier2} of the Ising model 
partition function 
gives the following identity:
\begin{equation*}
\Zising(\Gs,\Js)=2^{\vert \Vs\vert }\prod_{e\in\Es} \cosh\Js_e\sum_{\Ps\in\P(\Gs)}\prod_{e\in\Ps}\tanh\Js_e.
\end{equation*}

\subsubsection{The dimer model}

Consider a finite, planar graph $G=(V,E)$ together with positive edge-weights $\nu=(\nu_e)_{e\in E}$.
A \emph{dimer configuration} $\Ms$ of $G$, also known as a \emph{perfect matching}, is a subset of edges of $G$ such that every
vertex is incident to exactly one edge of $\Ms$. Let $\M(G)$ denote the set of dimer configurations of the graph $G$. The probability
of occurrence of a dimer configuration $\Ms$ is given by the \emph{dimer Boltzmann measure}, denoted $\PPdimer$:
\begin{equation*}
\PPdimer(\Ms)=\frac{\prod_{e\in\Es} \nu_e}{\Zdimer(G,\nu)},
\end{equation*}
where $\Zdimer(G,\nu)$ is the normalizing constant, known as the \emph{dimer partition function}. 

\subsubsection{Dimer representation of a single Ising model: Fisher's correspondence}\label{sec:Fisher_correspondence}

Fisher's correspondence~\cite{Fisher,Dubedat}
gives a mapping
between polygon configurations of a graph $\Gs$ and dimer configurations of a
decorated version of the graph, denoted $\GF$ and called the \emph{Fisher graph}. For the purpose of this paper it suffices to consider
graphs with no boundary.
The decorated graph $\GF=(\VF,\EF)$ is constructed from $\Gs$ as follows. Every
vertex of $\Gs$ of degree $d$ is replaced by a decoration containing $2d$
vertices: a triangle is attached to every edge incident to this vertex and these
triangles are glued together in a circular way, see Figure~\ref{Fig:Fisher_graph}. 

The correspondence goes as follows. To a polygon configuration $\Ps$ of $\Gs$ one assigns $2^{\vert \Vs\vert }$ dimer configurations of $\GF$: edges
present (resp.\ absent) in $\Ps$ are present (resp.\ absent) in $\GF$; then there are exactly two ways to fill each decoration of $\GF$ so as 
to have a dimer configuration, see Figure~\ref{Fig:Fisher_graph}.

\begin{figure}[ht]
\centering
\begin{overpic}[width=12cm]{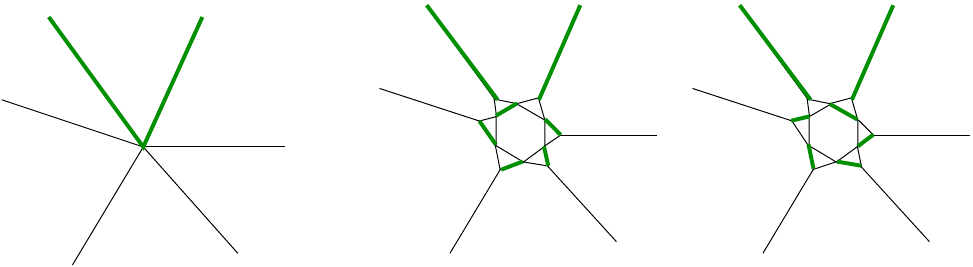}
  \put(3,0){\scriptsize $\Gs$}
  \put(70,0){\scriptsize $\GF$}
\end{overpic}
\caption{Left: a piece of a planar graph $\Gs$ and of a polygon configuration. Center and right: the corresponding Fisher graph $\GF$ 
and the two associated dimer configurations.}
\label{Fig:Fisher_graph}
\end{figure}

Let $\nu=(\nu_\es)_{\es\in\EF}$ be the dimer weight function corresponding to 
the high temperature expansion of the Ising model. Then $\nu$ is equal to
\begin{equation*}
\nu_\es
=  \begin{cases}
    1 & \text{if the edge $\es$ belongs to a decoration,} \\ 
    \tanh \Js_e & \text{if the edge $\es$ arises from an edge $e$ of $\Gs$,}\\
    0&\text{otherwise}.
  \end{cases}
\end{equation*}
From the correspondence, we know that:
\begin{equation}\label{equ:HTE}
\Zising(\Gs,\Js)=\left(\prod_{e \in\Es}\cosh\Js_e\right)\Zdimer(\GF,\nu).
\end{equation}
Note that the above is Dub\'edat's version of Fisher's correspondence~\cite{Dubedat}. It is more
convenient than the one used
in~\cite{BoutillierdeTiliere:iso_perio,BoutillierdeTiliere:iso_gen} because it allows to consider polygon configurations rather than
complementary ones, and the Fisher graph has less vertices, thus reducing the number of cases to handle.

\subsubsection{Dimer representation of the double Ising model}\label{sec:double_Ising_correspondence}

Based on results of physicists~\cite{KadanoffWegner,Wu71,FanWu,WuLin}, Dub\'edat~\cite{Dubedat}
provides a mapping between two independent Ising models, one living on the primal graph $\Gs$, the other on the dual graph $\Gs^*$,
to the dimer model on a bipartite graph $\GQ$. Based on results of~\cite{Nienhuis,Wu71}, two of the authors of the present paper exhibit 
an alternative mapping between two independent Ising models living on the \emph{same} graph $\Gs$ (embedded on a surface of genus $g$) to 
the bipartite dimer model on $\GQ$ \cite{BoutillierdeTiliere:XORloops}. 

Since the above mentioned mappings cannot be described
shortly, we refer to the original papers and only define the bipartite graph
$\GQ$ and the corresponding dimer weights. Note that dimer probabilities on the graph
$\GQ$ can be interpreted as probabilities of the low temperature expansion of
the \emph{XOR-Ising model}~\cite{BoutillierdeTiliere:XORloops}, also known as the polarization of the Ising model 
\cite{KadanoffBrown,WilsonXOR} obtained by taking the product of the spins of the two independent Ising models.

We only consider the case where the graph $\Gs$ is planar and infinite. 
The bipartite graph~$\GQ=(\VQ,\EQ)$ is obtained from 
$\Gs$ as follows. Every edge $e$ of $\Gs$ is replaced by a ``rectangle'', and the
``rectangles'' are joined in a circular way. The additional edges
of the cycles are referred to as \emph{external edges}. Note that in each
``rectangle'', two edges are ``parallel'' to an edge of the graph $\Gs$ and
two are ``parallel'' to the dual edge of $\Gs^*$, see Figure~\ref{fig:GQ}.

\begin{figure}[ht]
  \begin{center}
       \includegraphics[width=3.5cm]{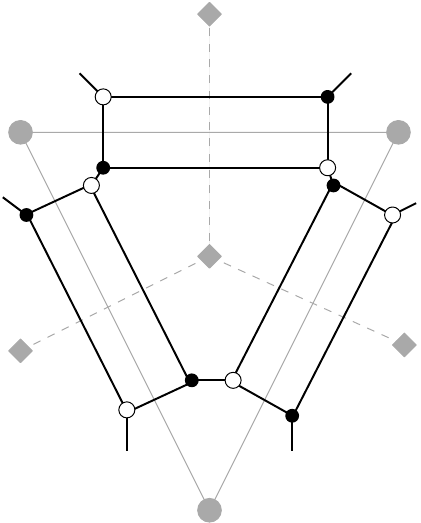}
    \caption{%
      A piece of a graph $\Gs$ (plain grey lines) and its dual graph $\Gs^*$ (dotted grey lines),
      and the corresponding bipartite graph $\GQ$ (plain black lines).}
    \label{fig:GQ}
  \end{center}
\end{figure}

Let $\overline{\nu}=(\overline{\nu}_\es)_{\es\in\EQ}$ be the dimer weight
function corresponding to two independent Ising models with coupling constants
$\Js$. Then $\overline{\nu}$ is equal to~\cite{Dubedat,BoutillierdeTiliere:XORloops}
\begin{equation*}
\overline{\nu}_\es
=  \begin{cases}
    \tanh(2\Js_e)& \text{if $\es$ belongs to a ``rectangle'' and is parallel to an edge $e$ of $\Gs$,} \\ 
    \cosh(2\Js_e)^{-1} & \text{if $\es$ belongs to a ``rectangle'' and is parallel to the dual of an edge $e$ of $\Gs$,}\\
     1 & \text{if $\es$ is an external edge,} \\ 
    0&\text{otherwise}.\\
  \end{cases}
\end{equation*}

\subsection{$Z$-invariant Ising model, dimer models and massive Laplacian}\label{sec:Z_inv_versions}

Although already present in the work of
Kenelly~\cite{Kennelly}, Onsager~\cite{Onsager} and Wannier~\cite{Wannier}, the notion of $Z$-invariance has been fully developed by
Baxter in the context of the integrable 8-vertex model~\cite{Baxter:8V}, and then applied to the Ising
model and self-dual Potts model~\cite{Baxter:Zinv}; see also~\cite{Perk:YB,Perk2,Kenyon6}. 
$Z$-invariance imposes a strong locality constraint which leads to the parameters of the model 
satisfying a set of equations known as the \emph{Yang-Baxter equations}. From the point of view of physicists it implies that transfer matrices
commute, and from the point of view of probabilists it suggests that there should exist local expressions for probabilities, but it provides no tool for 
finding such expressions if they exist. 

In Section~\ref{sec:iso_star_triangle} we define isoradial graphs, the associated diamond graph and star-triangle moves, all being key elements of $Z$-invariance.
Then in Section~\ref{sec:ZinvIsing} we introduce the 
$Z$-invariant Ising model~\cite{Baxter:8V,Baxter:Zinv,Baxter:exactly}, followed by the corresponding versions for the dimer models on $\GF$ and $\GQ$. 
Finally in Section~\ref{sec:massive_Lap} we define the $Z$-invariant massive Laplacian and the corresponding model of spanning
forests~\cite{BdTR1}.

\subsubsection{Isoradial graphs, diamond graphs and star-triangle moves}\label{sec:iso_star_triangle}

Isoradial graphs, whose name comes from the paper~\cite{Kenyon3}, see
also~\cite{Duffin,Mercat:ising}, are defined as follows. An infinite planar graph
$\Gs=(\Vs,\Es)$ is \emph{isoradial}, if it can be embedded in the plane in such
a way that all internal faces are inscribable in a circle, with all circles
having the same radius, and such that all circumcenters are in the interior of
the faces, see Figure~\ref{fig:Iso1} (left). This definition is easily adapted when $\Gs$ is finite or embedded
in the torus.

From now on, we fix an embedding of the graph, take the
common radius to be $1$, and also denote by $\Gs$ the embedded graph. An
isoradial embedding of the dual graph $\Gs^*$, with radius $1$, is obtained by
taking as dual vertices the circumcenters of the corresponding faces.

The \emph{diamond graph}, denoted $\GR$, is constructed from
an isoradial graph $\Gs$ and its dual $\Gs^*$.
Vertices of $\GR$ are those of~$\Gs$ and those of $\Gs^*$. A dual vertex of
$\Gs^*$ is joined to all primal
vertices on the boundary of the corresponding face, see Figure~\ref{fig:Iso1}
(right). Since edges of the diamond graph $\GR$ are radii of circles, 
they all have length $1$, and can be assigned a direction $\pm
e^{i\overline{\alpha}}$. Note that faces of $\GR$ are 
side-length $1$ rhombi.

\begin{figure}[ht]
  \begin{center}
    \begin{tabular}{cc}
      \includegraphics[width=7cm]{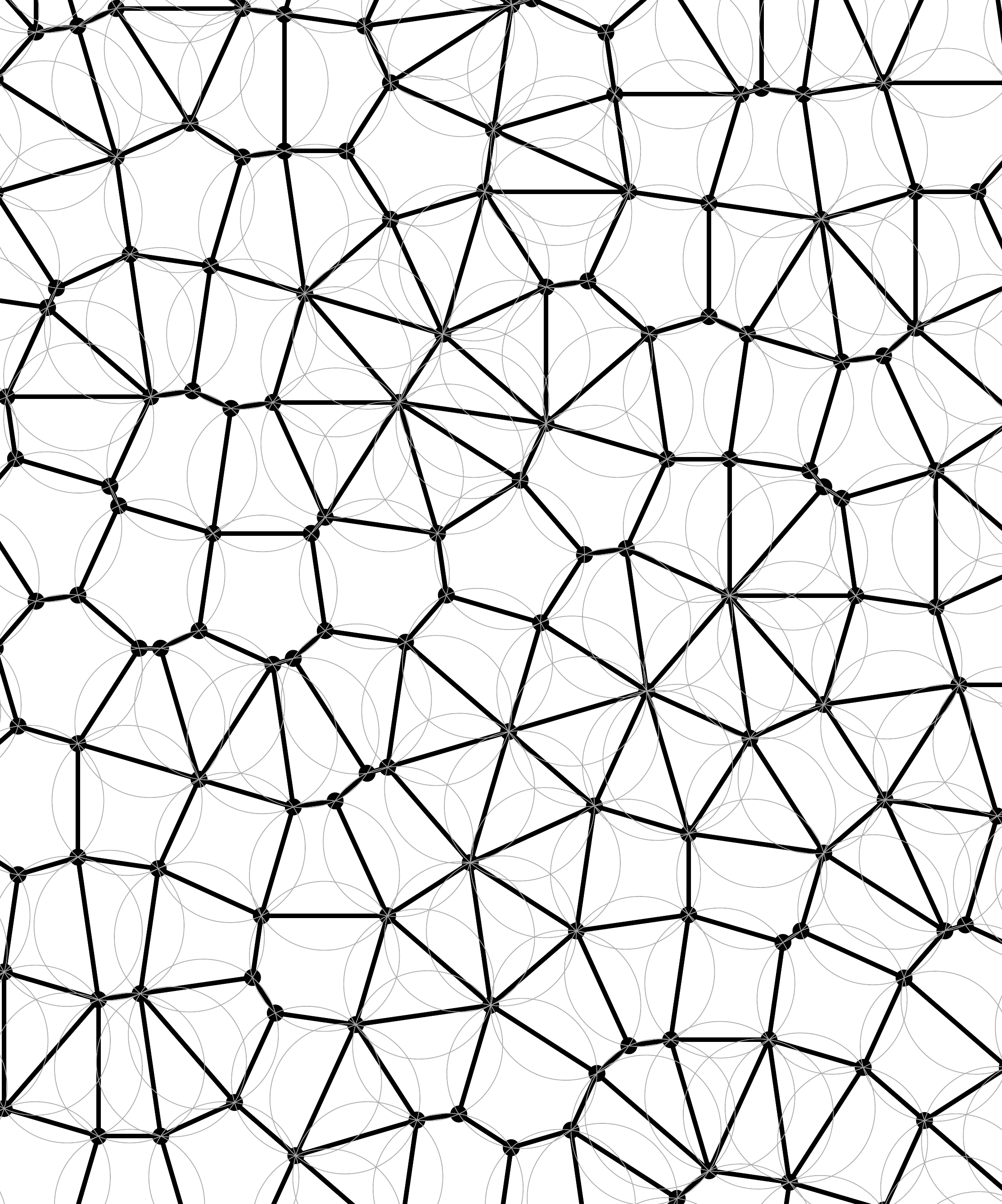} &
      \includegraphics[width=7cm]{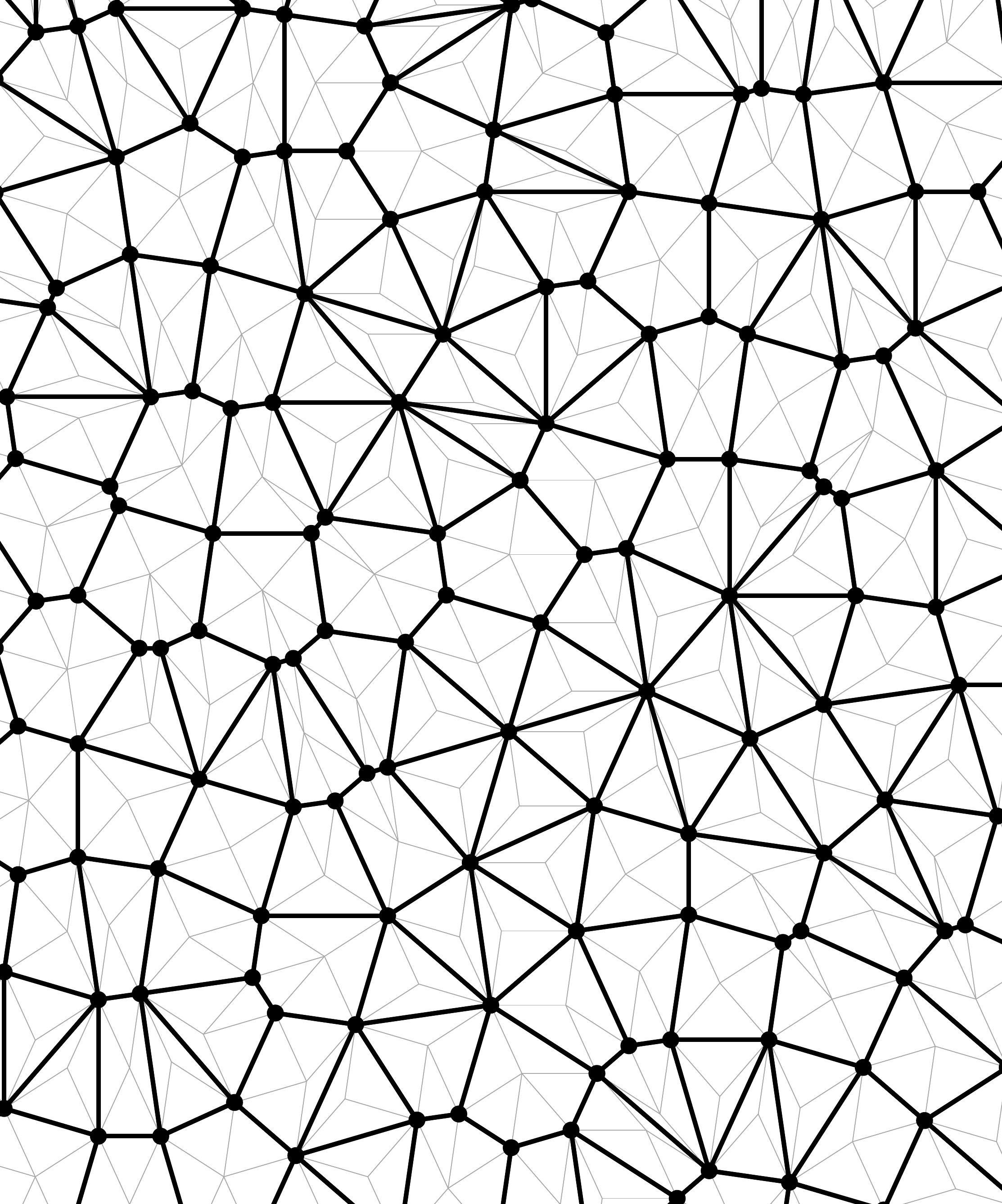} 
    \end{tabular}
    \caption{%
      Left: a piece of an infinite isoradial graph $\Gs$ (bold)
      with its circumcircles. Right: the diamond
      graph $\GR$.}
    \label{fig:Iso1}
  \end{center}
\end{figure}

Using the diamond graph, angles can naturally be assigned to edges of the
graph~$\Gs$ as follows.
Every edge $e$ of $\Gs$ is the diagonal of exactly one
rhombus of $\GR$, and we let $\overline{\theta}_e$ be the half-angle at
the vertex it has in common with $e$, see Figure~\ref{fig:rhombus_angle}.
We have $\overline{\theta}_e\in(0,\frac{\pi}{2})$, because circumcircles are
assumed to be in the interior of the faces.
From now on, we ask more and suppose that there exists $\eps>0$ such that
$\overline{\theta}_e\in(\eps,\frac{\pi}{2}-\eps)$.
We further assign two rhombus vectors to the edge $e$, denoted by 
$e^{i\overline{\alpha}_e}$ and $e^{i\overline{\beta}_e}$, see
Figure~\ref{fig:rhombus_angle}.

\begin{figure}[ht]
  \centering
\begin{overpic}[height=2cm]{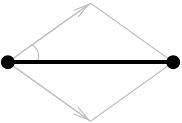}
\put(47,25){\scriptsize{$e$}}
\put(8,10){\scriptsize{$e^{i\overline{\alpha}_e}$}}
\put(8,49){\scriptsize{$e^{i\overline{\beta}_e}$}}
\put(25,38){\scriptsize{$\overline{\theta}_e$}}
\end{overpic}
\caption{An edge $e$ of $\Gs$, the corresponding rhombus half-angle $\overline{\theta}_e$ and rhombus vectors 
$e^{i\overline{\alpha}_e}$, $e^{i\overline{\beta}_e}$.
\label{fig:rhombus_angle}}
\end{figure}
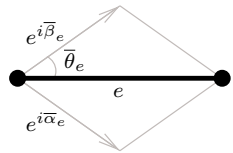

A \emph{train-track} of $\Gs$ is a bi-infinite chain of edge-adjacent rhombi of $\GR$ which does not turn: on entering a face, it exits
along the opposite edge~\cite{KeSchlenk}. Each rhombus in a train-track $T$ has an edge parallel to a fixed unit vector $\pm e^{i\overline{\alpha}_T}$,
known as the \emph{direction of the train-track}. Train-tracks are also known as \emph{rapidity lines} or simply \emph{lines} 
in the field of integrable systems, see for example~\cite{Baxter:8V}.

The \emph{star-triangle move}, also known as the
\emph{$Y$-$\,\Delta$ transformation}, underlies $Z$-invariance~\cite{Baxter:8V,Baxter:Zinv}.
It is defined as follows: if $\Gs$ has a vertex of degree 3, that is a \emph{star} $Y$, it can be replaced by a
\emph{triangle} $\Delta$ by removing the vertex and connecting its three
neighbors. The
graph obtained in this way is still isoradial: its diamond graph is obtained by
performing a \emph{cubic flip} in $\GR$, that is by flipping the three rhombi of
the corresponding hexagon, see Figure~\ref{fig:star_triangle}. This operation is involutive.

\begin{figure}[ht]
\begin{center}
\resizebox{0.6\textwidth}{!}{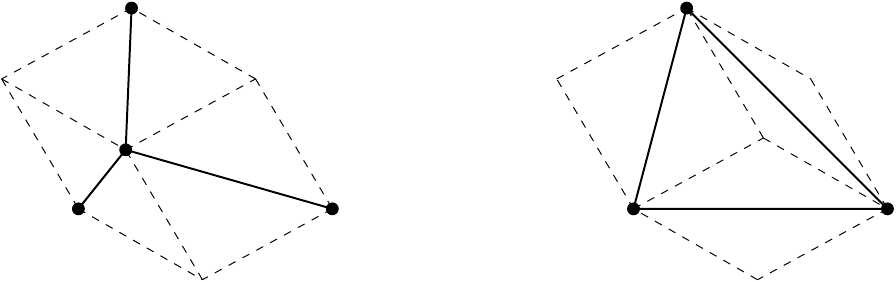}
\end{center}
\caption{Star-triangle move on an isoradial graph $\Gs$ (plain lines) and cubic flip on the underlying diamond graph $\GR$ (dotted lines).}
\label{fig:star_triangle}
\end{figure}

\subsubsection{$Z$-invariant Ising model}\label{sec:ZinvIsing}

The Ising model defined on a graph $\Gs$ is said to be
\emph{$Z$-invariant}, if when decomposing the partition function according
to the possible spin configurations at the three vertices of a star/triangle, it
only changes by a constant when performing the 
$Y$-$\,\Delta$ move, this constant being independent of the choice of
spins at the three vertices. 

This strong constraint yields
a set of equations known as the \emph{Ising model Yang-Baxter equations}, see (6.4.8) of~\cite{Baxter:exactly} and also
\cite{Onsager,Wannier}.
The solution to these equations can be 
parametrized by the \emph{elliptic modulus} $k$, where $k$ is a complex number such that $k^2\in(-\infty,1)$,
and the \emph{rapidity parameters}, see Equation~(7.8.4) and page 478 of~\cite{Baxter:exactly}.
In this context it is thus natural to suppose that the graph $\Gs$ is isoradial. Extending the 
form of the coupling constants to the whole of $\Gs$ we obtain that they
are given by, for every edge $e$ of $\Gs$, 
\begin{equation}
\label{eq:def_weight}
\Js_e=\Js(\overline{\theta}_e\vert k)=\frac{1}{2}
\log\left(\frac{1+\sn(\theta_e\vert k)}{\cn(\theta_e\vert k)}\right),
\text{ or equivalently }
\sinh(2\Js(\overline{\theta}_e\vert k))= \sc(\theta_e\vert k)
\footnote{In Equation (7.8.4), Baxter actually uses the complementary parameter $k'=\sqrt{1-k^2}$ and the 
parametrization, $\sinh(2\Js_e)=-i\sn(i\theta_e\vert k')$. The latter is equal to $\sc(\theta_e\vert k)$ by \cite[(2.6.12)]{La89}.},
\end{equation}
where
$k$ is the \emph{elliptic modulus}, $\theta_e=\overline{\theta}_e \frac{2K}{\pi}$, 
$K=K(k)=\int_{0}^{\frac{\pi}{2}} \frac{1}{\sqrt{1-k^2\sin^2\tau}}\ud\tau$ is the 
\emph{complete elliptic integral of the first kind}, $\cn(\cdot\vert k)$, $\sn(\cdot\vert k)$ and $\sc(\cdot\vert k) =
\frac{\sn(\cdot\vert k)}{\cn(\cdot\vert k)}$ are three of the twelve \emph{Jacobi trigonometric
elliptic functions}. More on their definition can be found in the books~\cite[Chapter~16]{AS} and
\cite{La89}; a short introduction is also given in the paper~\cite[Section 2.2]{BdTR1}. 
Identities that are useful for this paper can be found in Appendix~\ref{app:elliptic}.

For a given isoradial graph $\Gs$, we thus have a one-parameter
family of coupling constants $(\Js)_{k}$, indexed by the
elliptic modulus $k$, with $k^2\in(-\infty,1)$. For every edge $e$, the coupling constant $\Js(\overline{\theta}_e\vert k)$ is analytic in $k^2$
and increases from $0$ to $\infty$ as $k^2$ increases from $-\infty$ to $1$, see Lemma~\ref{lem:poids_croissants}; 
the elliptic modulus $k$ thus parametrizes the whole range of temperatures. When $k=0$, elliptic functions degenerate to trigonometric functions,
and we have:
\begin{equation*}
\Js(\overline{\theta}_e\vert 0)=\frac{1}{2}
\log\left(\frac{1+\sin\theta_e}{\cos\theta_e}\right).
\end{equation*}
The Ising model is critical at $k=0$, see~\cite{Li:critical,CimasoniDuminil,Lis}. More on this subject is to be found in 
Section~\ref{sec:duality_and_phase_transition}.

\subsubsection{Corresponding dimer model on the Fisher graph $\GF$}
Let us compute the dimer weight function $\nu$ on $\GF$ corresponding to the
$Z$-invariant Ising model on $\Gs$ with coupling constants $\Js$
given by \eqref{eq:def_weight}. For every edge $e$ of $\Gs$, we have
\begin{align*}
\tanh(\Js_e)
=\frac{e^{2\Js_e}-1}{e^{2\Js_e}+1}
=\frac{%
  \frac{1+\sn\theta_e}{\cn\theta_e}-1
}{%
  \frac{1+\sn\theta_e}{\cn\theta_e}+1
}
=\frac{1+\sn\theta_e-\cn\theta_e}{1+\sn\theta_e+\cn\theta_e}
=\frac{\sn\theta_e}{1+\cn\theta_e} =
\sc\frac{\theta_e}{2}
\dn\frac{\theta_e}{2},
\end{align*}
see~\cite[(2.4.4)--(2.4.5)]{La89} for the last identity.

As a consequence of Section~\ref{sec:Fisher_correspondence},
the dimer weight function $\nu$ on the Fisher graph $\GF$ is
\begin{equation}\label{eq:dimer_weights_GF}
\nu_\es=
\begin{cases}
1 &
    \text{if $\es$ belongs to a decoration,}\\
\frac{\sn\theta_e}{1+\cn\theta_e}=\sc\frac{\theta_e}{2}
\dn\frac{\theta_e}{2} &
    \text{if $\es$ corresponds to an edge $e$ of $\Gs$,}\\
0&\text{otherwise}.
\end{cases}
\end{equation}

When $k=0$ we have $\dn=1$ and $\sc=\tan$, which corresponds to the critical case.

\subsubsection{Corresponding dimer model on the bipartite graph $\GQ$}
In a similar way, we compute the dimer weight function $\overline{\nu}$ of the
graph $\GQ$ corresponding to two independent $Z$-invariant Ising model.
We have
\begin{align}
\cosh(2\Js_e)&=\frac{1}{2}
\left(
\frac{1+\sn\theta_e}{\cn\theta_e}+\frac{\cn\theta_e}{1+\sn\theta_e}\right)={\nc\theta_e},\label{equ:dheart}\\
\tanh(2\Js_e) &=\frac{\sinh(2\Js_e)}{\cosh(2\Js_e)} = \sn\theta_e.\nonumber
\end{align}
As a consequence of Section~\ref{sec:double_Ising_correspondence},
the dimer weight function $\overline{\nu}$ on the bipartite graph $\GQ$ is
\begin{equation}\label{eq:dimer_weights_GQ}
  \overline{\nu}_\es=
\begin{cases}
\sn\theta_e & \text{if $\es$ is parallel to an edge $e$ of $\Gs$,}\\ 
\cn\theta_e & \text{if $\es$ is parallel to the dual edge of an edge $e$ of $\Gs$,}\\ 
1 &\text{if $\es$ is an external edge,}\\
0&\text{otherwise}.
\end{cases}
\end{equation}

\subsubsection{The $Z$-invariant massive Laplacian}\label{sec:massive_Lap}

We will be using results on the $Z$-invariant massive Laplacian introduced in~\cite{BdTR1}. 
Let us recall its definition and the key facts required for this paper.

Following \cite[Equation~(1)]{BdTR1}, the \emph{massive Laplacian operator} $\Delta^{m}:\CC^\Vs\rightarrow\CC^\Vs$ 
is defined
as follows. Let $\xb$ be a vertex of $\Gs$ of degree $n$; denote by $e_1,\dots,e_n$ edges incident to $\xb$ and by
$\overline{\theta}_1,\dots,\overline{\theta}_n$ the corresponding rhombus half-angles, then
\begin{equation}
\label{eq:Laplacian_operator}
     (\Delta^{m} f)(\xb)=\sum_{j=1}^n \rho(\overline{\theta}_i\vert k)[f(\xb)-f(\yb)]+m^{2}(\xb\vert k)f(\xb),
\end{equation}
where the \emph{conductances} $\rho$ and \emph{(squared) masses} $(m^{2})$ are defined by
\begin{align}\label{eq:conductances}
     \rho_e=\rho(\overline{\theta}_e\vert k)&=\sc(\theta_e\vert k),\\
     (m^{2})(\xb)=m^2(\xb\vert k)&=\sum_{j=1}^{n}(\Arm (\theta_j\vert k)-\sc(\theta_{j}\vert k)),\label{eq:mass}
\end{align}
with
\begin{equation*}
     \Arm (u\vert k)=
     \frac{1}{k'}\Bigl(
     \Dc(u\vert k)+\frac{E-K}{K}u
     \Bigr),
\end{equation*}
where $\Dc(u\vert k)=\int_{0}^u \dc^2(v\vert k)\ud v$, and $E=E(k)$ is the \emph{complete elliptic integral of the second kind}.

We also need the definition of the \emph{discrete $k$-massive exponential function}
or simply \emph{massive exponential function}, 
denoted $\expo_{(\cdot,\cdot)}(\cdot)$, of \cite[Section 3.3]{BdTR1}. It 
is a function from $\Vs\times\Vs\times\CC$ to $\CC$. Consider a pair of vertices
$\xb,\yb$ of $\Gs$ and an edge-path
$\xb=\xb_1,\dotsc,\xb_n=\yb$ of the diamond graph $\GR$ from $\xb$ to $\yb$; let
$e^{i\overline{\alpha}_j}$ be the vector corresponding to the edge
$\xb_j\xb_{j+1}$. Then $\expo_{(\xb,\yb)}(\cdot)$ is defined inductively 
along the edges of the path:
\begin{align}
\label{eq:recursive_def_expo}
\forall\,u\in\CC,\quad
\expo_{(\xb_j,\xb_{j+1})}(u) &= i \sqrt{k'}\,\sc\Bigl(\frac{u-{\alpha_j}}{2}\Bigr),\nonumber\\
\expo_{(\xb,\yb)}(u)         &= \prod_{j=1}^{n-1} \expo_{(\xb_j,\xb_{j+1})}(u),
\end{align}
where $\alpha_j=\overline{\alpha}_j\frac{2K}{\pi}$.
These functions are in the kernel of the massive Laplacian
\eqref{eq:Laplacian_operator}, see \cite[Proposition~11]{BdTR1}.

The \emph{massive Green function}, denoted $G^{m}$, is the inverse of the massive Laplacian
operator~\eqref{eq:Laplacian_operator}. The following local formula is proved in
\cite[Theorem~12]{BdTR1}:
\begin{equation}
\label{eq:def_Green_function}
     G^{m}(\xb,\yb) =\frac{k'}{4i\pi} \int_{\Gamma_{\xb,\yb}} \expo_{(\xb,\yb)}(u) \ud u,
\end{equation}
where $k'=\sqrt{1-k^2}$ is the complementary elliptic modulus, $\Gamma_{\xb,\yb}$ is a vertical contour on the torus 
$\TT(k):=\CC/(4K\ZZ+4iK'\ZZ)$, whose direction is given by the angle of the ray $\RR\overrightarrow{\xb\yb}$.

The massive Laplacian is the operator underlying the model of spanning forests, the latter being defined as follows. 
A \emph{spanning forest} of $\Gs$ is a subgraph spanning all vertices of the graph, such that every connected component is a rooted tree. 
Denote by $\F(\Gs)$ the set of spanning forests of $\Gs$ and for a rooted tree $\Ts$, denote its root by $\xb_\Ts$.
The \emph{spanning forest Boltzmann measure}, denoted $\PPforest$, is defined by:
\begin{equation*}
\forall\,\Fs\in\F(\Gs),\quad \PPforest(\Fs)=
\frac{\prod_{\Ts\in\Fs}\left(m^2(\xb_\Ts\vert k)\prod_{e\in\Ts}\rho(\overline{\theta}_e\vert k)\right)}{\Zforest(\Gs,\rho,m)},
\end{equation*}
where $\Zforest(\Gs,\rho,m)$ is the spanning forest partition function. In~\cite[Theorem 41]{BdTR1} we prove that 
this model is $Z$-invariant (thus explaining the name $Z$-invariant massive Laplacian). 
By Kirchhoff's matrix-tree theorem we have $\Zforest(\Gs,\rho,m)=\det(\Delta^{m})$.

\section{$Z$-invariant Ising model via dimers on the Fisher graph $\GF$}\label{sec:Ising_dimers}

From now on, we consider a fixed elliptic modulus
$k^2\in(-\infty,1)$,
so that we will remove the dependence in $k$ from the notation.

In the whole of this section, we let $\Gs$ be an infinite isoradial
graph and $\GF$ be the corresponding Fisher graph. We suppose that
edges of $\GF$ are assigned the weight function $\nu$
of~\eqref{eq:dimer_weights_GF} arising from the $Z$-invariant Ising model.

We give a full description of the dimer model on the Fisher graph $\GF$
with explicit expressions having the remarkable property of being \emph{local}. This extends to 
the $Z$-invariant \emph{non-critical} case the results of~\cite{BoutillierdeTiliere:iso_perio,BoutillierdeTiliere:iso_gen} 
obtained in the $Z$-invariant \emph{critical} case,
corresponding to $k=0$.
One should keep in mind that when $k=0$, the ``torus'' $\TT(0)$
  is in fact an infinite cylinder with two points at infinity, and that
 ``elliptic'' functions are trigonometric series.

Prior to giving a more detailed
outline, we introduce 
the main object involved in explicit expressions for the dimer model, namely, the 
\emph{Kasteleyn matrix/operator}~\cite{Kasteleyn1,TF}.

\subsection{Kasteleyn operator on the Fisher graph}

An orientation of the edges of $\GF$ is said to be \emph{admissible} if all
cycles bounding faces of the graph are \emph{clockwise odd},
meaning that, when following such a cycle clockwise, there is an odd number of
co-oriented edges. By Kasteleyn~\cite{Kasteleyn2}, such an orientation always exists.

Suppose that edges of $\GF$ are assigned an admissible orientation, then the
\emph{Kasteleyn matrix} $\KF$ is the 
corresponding weighted, oriented, adjacency matrix of $\GF$. It has rows and
columns indexed by vertices of $\GF$ and coefficients 
given by, for every $\xs,\ys\in\VF$,
\begin{equation*}
\KF_{\xs,\ys}=\sgn(\xs,\ys)\nu_{\xs\ys},
\end{equation*}
where $\nu$ is the dimer weight function \eqref{eq:dimer_weights_GF} and 
\begin{equation*}
\sgn(\xs,\ys)=
\begin{cases}
1&\text{if $\xs\sim\ys$ and $\xs\rightarrow\ys$,}\\ 
-1&\text{if $\xs\sim\ys$ and $\ys\rightarrow\xs$.}
\end{cases}
\end{equation*}
Note that $\KF$ can be seen as an operator acting on $\CC^{\VF}$:
\begin{equation*}
     \forall\,f\in\CC^{\VF},\,\forall\,\xs\in\VF,\quad (\KF f)_\xs=\sum_{\ys\in\VF}\KF_{\xs,\ys}f_\ys.
\end{equation*}

\paragraph{Outline.}
Section~\ref{sec:Ising_dimers} is structured as follows. 
In Section~\ref{sec:function_kernel} we introduce a 
one-parameter family of functions in the kernel of the Kasteleyn operator $\KF$; this key result allows us to prove 
one of the main theorems of this paper:
a \emph{local} formula for an inverse ${\KF}^{-1}$ of the 
operator $\KF$, see Theorem~\ref{thm:KFmoins_un} of Section~\ref{sec:localKFinv}. Then in Section~\ref{sec:asymptKF} we derive asymptotics
of this inverse. In Section~\ref{sec:KFperio} we handle the case where the graph $\Gs$ is periodic. Finally in 
Section~\ref{subsec:dimer_model_GF} we derive results for the dimer model on $\GF$: we prove a local expression for the dimer Gibbs measure, see
Theorem~\ref{thm:Gibbs_measure}, and a local formula for the dimer and Ising free energies, see Theorem~\ref{thm:free_energy_dimer}  and 
Corollary~\ref{cor:free_energy_Ising}; we then show that up to an additive constant the Ising model free energy is equal to $\frac{1}{2}$
of the spanning forest free energy, see Corollary~\ref{cor:link_free_energies}.

\paragraph{Notation.} 
Throughout this section, we use the following notation. A vertex 
$\xs$ of $\GF$ belongs to a decoration corresponding to a unique vertex $\xb$ of $\Gs$. 
Vertices of $\GF$ corresponding to a vertex $\xb$ of $\Gs$ are labeled as
follows. Let $d(\xb)$ be the degree of the vertex 
$\xb$ in $\Gs$, then the decoration consists of $d(\xb)$ triangles, labeled from
$1$ to $d(\xb)$ in counterclockwise order. 
For the $j$-th triangle, we let $\vs_j(\xb)$ be the vertex incident to an edge
of $\Gs$, and $\ws_j(\xb),\ws_{j+1}(\xb)$
be the two adjacent vertices in counterclockwise order, see Figure~\ref{Fig:GF_notation}.

There is a natural way of assigning rhombus unit-vectors of $\GR$ to vertices of $\GF$: 
for every vertex $\xb$ of $\Gs$ and every $k\in\{1,\dots,d(\xb)\}$, let us
associate the rhombus vector $e^{i\overline{\alpha}_j(\xb)}$ to $\ws_j(\xb)$,
and the rhombus vectors 
$e^{i\overline{\alpha}_j(\xb)}, e^{i\overline{\alpha}_{j+1}(\xb)}$ to $\vs_j(\xb)$,
see Figure~\ref{Fig:GF_notation}; we let $\overline{\theta}_j(\xb)$ be the half-angle
at the vertex $\xb$ of the rhombus defined by $e^{i\overline{\alpha}_j(\xb)}$
and $e^{i\overline{\alpha}_{j+1}(\xb)}$, with 
$\overline{\theta}_j(\xb)\in(0,\frac{\pi}{2})$.

\begin{figure}[ht]
  \centering
\begin{overpic}[width=13cm]{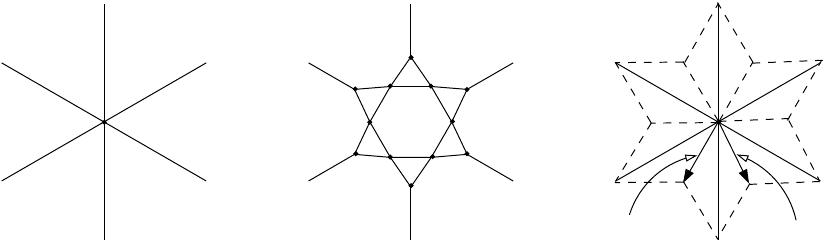}
\put(10,-2){$\Gs$}
\put(45,-2){$\GF$}
\put(82,-2){$\GR$}
\put(13,12){\scriptsize $\xb$}
\put(45,8){\scriptsize $\ws_j$}
\put(47.5,5.5){\scriptsize $\vs_j$}
\put(52,8){\scriptsize $\ws_{j+1}$}
\put(57,11){\scriptsize $\vs_{j+1}$}
\put(49,14){\scriptsize $\xb$}
\put(87,11){\scriptsize $\xb$}
\put(73,0){\scriptsize $e^{i\overline{\alpha}_j}$}
\put(92,0){\scriptsize $e^{i\overline{\alpha}_{j+1}}$}
\end{overpic}
\caption{Notation for vertices of decorations, and rhombus vectors assigned to vertices. Since no confusion occurs, the argument $\xb$ is 
omitted.}
\label{Fig:GF_notation}
\end{figure}
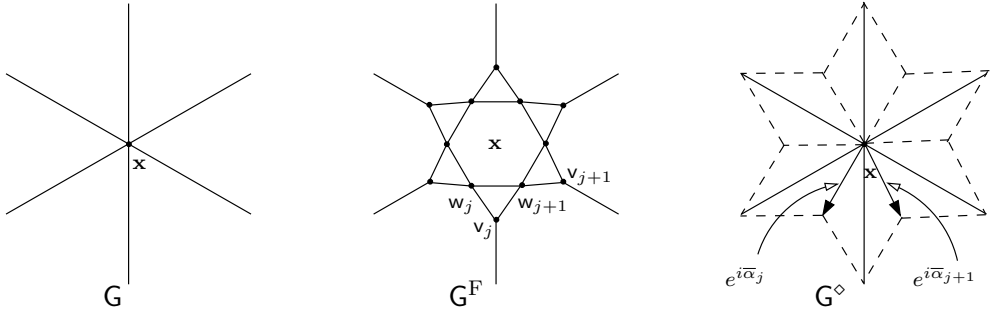

Recall the notation $\theta_e=\overline{\theta}_e \frac{2K}{\pi}$ and $\alpha=\overline{\alpha} \frac{2K}{\pi}$ for the 
elliptic versions of $\overline{\theta}_e$ (rhombus half-angle) and  $\overline{\alpha}$ 
(angle of the rhombus vector $e^{i\overline{\alpha}}$ of $\GR$).

\subsection{Functions in the kernel of the Kasteleyn operator $\KF$}\label{sec:function_kernel}

The definition of the one-parameter family of functions in the kernel of the Kasteleyn operator $\KF$ requires two 
ingredients: the function $\fs$ of Definition~\ref{def:function_f}
and the massive discrete exponential 
function of~\cite{BdTR1}.

\label{subsubsec:fonction_f}

The function $\fs$ uses the angles $(\overline{\alpha}_j(\xb))$ assigned to vertices of $\GF$, the latter being 
\textit{a priori} defined in $\RR/2\pi\ZZ$. For the function $\fs$ to be well defined, we
actually need them to be defined in $\RR/4\pi\ZZ$, 
which is equivalent to a coherent choice for the determination of the square
root of $e^{i\overline{\alpha}_j(\xb)}$. 
This construction is done iteratively, relying on our choice of Kasteleyn
orientation.

Fix a vertex $\xb_0$ of 
$\Gs$ and set the value of $\overline{\alpha}_1(\xb_0)$ to some value, say $0$. 
In the following, we use the index $j$ (resp.\ $\ell$)
to refer to vertices of $\GF$ belonging to a decoration $\xb$ (resp.\ $\yb$) of
$\Gs$; with this convention, we omit the arguments $\xb$ and $\yb$
from the notation. For vertices in a decoration
of a vertex $\xb$ of $\Gs$, define
\begin{equation}
\label{equ:def_angle}
\overline{\alpha}_{j+1}= 
\begin{cases}
\overline{\alpha}_j+2\overline{\theta}_{j}&\text{if $\ws_j \to \ws_{j+1}$,}\\
\overline{\alpha}_j+2\overline{\theta}_{j}+2\pi&\text{if $\ws_{j+1} \to \ws_j$.}
\end{cases}
\end{equation}
Given a directed path $\gamma$, let 
$\co(\gamma)$ be the number of co-oriented edges.
Here is the rule defining angles in the decoration corresponding to a vertex
$\yb$ of $\Gs$, neighbor of the vertex~$\xb$.
Let $j$ and $\ell$ be such that $\vs_j$ is incident to $\vs_{\ell}$, as in
Figure~\ref{Fig:fig_GF_1}. 
Consider the length-three directed path
$\ws_j,\vs_j,\vs_\ell,\ws_\ell$ from $\ws_j$ to $\ws_\ell$. Then
\begin{equation}
\label{eq:relative_definition_a_a'}
\overline{\alpha}_{\ell}= 
\begin{cases}
\overline{\alpha}_{j}-\pi& \text{if $\co(\ws_j,\vs_j,\vs_\ell,\ws_\ell)$ is odd,}\\
\overline{\alpha}_{j}+\pi& \text{if $\co(\ws_j,\vs_j,\vs_\ell,\ws_\ell)$ is even}.
\end{cases}
\end{equation}

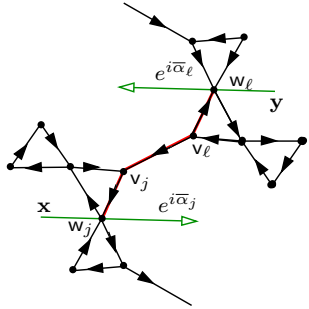
\begin{figure}[ht]
  \centering
\begin{overpic}[width=4cm]{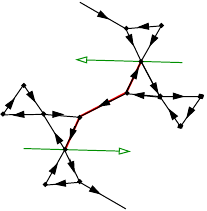}
\put(40,40){\scriptsize $\vs_j$}
\put(20,25){\scriptsize $\ws_j$}
\put(60,50){\scriptsize $\vs_\ell$}
\put(72,72){\scriptsize $\ws_\ell$}
\put(10,31){\scriptsize $\xb$}
\put(85,65){\scriptsize $\yb$}
\put(49,31){\scriptsize $e^{i\overline{\alpha}_j}$}
\put(48,74){\scriptsize $e^{i\overline{\alpha}_{\ell}}$}
\end{overpic}
\caption{Defining angles in neighboring decorations.}
\label{Fig:fig_GF_1}
\end{figure}

\begin{lem}\label{lem:def_angles}
The angles $(\overline{\alpha}_j(\xb))_{\xb\in\Vs,\, j\in\{1,\dots,d(\xb)\}}$ 
are well defined in $\RR/4\pi\ZZ$.
\end{lem}

The proof is postponed to Appendix~\ref{app:proof_angles_4pi}. It is reminiscent of the proof of Lemma 4 of~\cite{BoutillierdeTiliere:iso_gen}
but has to be adapted since we are working with a different version of the Fisher graph.

\begin{defi}\label{def:function_f}
The function $\fs:\VF\times \CC \rightarrow\CC$ is defined by
\begin{equation}
\label{equ:rewriting_fs}
\left\{\begin{array}{lll}
  \fs(\ws_j,u)&\hspace{-2.5mm}:=\fs_{\ws_j}(u) &\hspace{-2.5mm}=
    \nc(\frac{u-\alpha_j}{2}),\vspace{0.1cm} \\
  \fs(\vs_j,u)&\hspace{-2.5mm}:=\fs_{\vs_j}(u) &\hspace{-2.5mm}=
  \displaystyle \KF_{\vs_j,\ws_j} \fs_{\ws_j}(u)+ \KF_{\ws_{j+1},\vs_j} \fs_{\ws_{j+1}}(u)\\
  &&\hspace{-2.5mm}=\KF_{\vs_j,\ws_j} \nc(\frac{u-\alpha_j}{2})+ \KF_{\ws_{j+1},\vs_j} \nc(\frac{u-\alpha_{j+1}}{2}).
  \end{array}\right.
\end{equation}
\end{defi}

\begin{defi}\label{def:fonction_g}
The function $\gs:\VF\times\VF\times\CC\rightarrow\CC$ is defined by
\begin{equation}\label{equ:function_g}
  \gs(\xs,\ys,u):=\gs_{\xs,\ys}(u)= \fs_{\xs}(u+2K) \fs_{\ys}(u) \expo_{(\xb,\yb)}(u),
\end{equation}
where $\expo_{(\cdot,\cdot)}(\cdot)$ is the massive exponential function of~\cite{BdTR1}, whose definition is recalled in Section~\ref{sec:massive_Lap}.
\end{defi}

\begin{rem}
The function $\gs$ is meromorphic and biperiodic:
\begin{equation*}
     \gs_{(\xs,\ys)}(u+4K) = \gs_{(\xs,\ys)}(u+4iK') = \gs_{(\xs,\ys)}(u),
\end{equation*}
so that we restrict the domain of definition to $\TT(k):=\CC /(4K\ZZ + 4iK'\ZZ)$. 
Note however that taken separately, $\fs_{\xs}(\cdot+2K)$ and $\fs_{\ys}(\cdot)$
are not periodic on $\TT(k)$: only their product is.
\end{rem}

The function $\gs$ can also be seen as a one-parameter family of matrices
$(\gs(u))_{u\in\TT(k)}$, where for every $u\in\TT(k)$, $\gs(u)$
has rows and columns indexed by vertices of $\GF$, and
$\gs(u)_{\xs,\ys}:=\gs_{(\xs,\ys)}(u)$. We have the following key proposition.

\begin{prop}
\label{prop:function_f_kernel_Kasteleyn_Fisher}
  For every $u\in\TT(k)$,
  $
    \KF \gs(u) =  \gs(u) \KF= 0$.
\end{prop}

\begin{proof}
  Note that since $\KF$ is skew-symmetric, and that up to a sign,
   the functions $\xs\mapsto \gs_{(\xs,\zs)}(u)$ and $\xs\mapsto \gs_{(\zs,\xs)}(u+2K)$
  are equal:
  \begin{equation*}
    \gs_{(\zs,\xs)}(u+2K) =
    \fs_{\zs}(u+4K)\fs_{\xs}(u+2K)\expo_{(\zb,\xb)}(u+2K)=
    -\fs_{\xs}(u+2K)\fs_{\zs}(u)\expo_{(\xb,\zb)}(u)=
    -\gs_{(\xs,\zs)}(u),
  \end{equation*}
  it is enough to check the first equality, \emph{i.e.}, $\KF \gs(u)=0$.
  
  Let us fix $\zs$. We need to check that for every vertex $\xs$ of $\GF$,
  \begin{equation*}
    \sum_{i=1}^d \KF_{\xs,\xs_i} \gs_{(\zs,\xs_i)}(u) = 0,
  \end{equation*}
  where $\xs_1,\ldots,\xs_d$ are the $d$ (equal to three or four) neighbors of
  $\xs$ in $\GF$.
  We distinguish two cases depending on whether the vertex $\xs$ is of type $\ws$ or $\vs$.

$\bullet$ If $\xs=\ws_j(\xb)$ for some $j$, then $\xs$ has four neighbors:
      $\ws_{j-1}(\xb)=\ws_{j-1}$, $\ws_{j+1}(\xb)=\ws_{j+1}$, $\vs_{j-1}(\xb)=\vs_{j-1}$ and $\vs_j(\xb)=\vs_j$, see
      Figure~\ref{Fig:GF_notation}. Since all these vertices belong to the same
      decoration, the part $\fs_{\zs}(u+2K)\expo_{(\zb,\xb)}(u)$ is common to all the terms
      $\gs_{\zs,\xs_i}(u)$. One is left with proving the following identity:
      \begin{equation*}
        (\KF \fs)_{\ws_j}=
	\KF_{\ws_j,\ws_{j-1}} \fs_{\ws_{j-1}}
	+
	\KF_{\ws_j,\ws_{j+1}} \fs_{\ws_{j+1}}
	+
	\KF_{\ws_j,\vs_{j-1}} \fs_{\vs_{j-1}}
	+
	\KF_{\ws_j,\vs_j} \fs_{\vs_j}
	= 0.
      \end{equation*}
      Using the second line in Equation \eqref{equ:rewriting_fs} to express $\fs_{\vs_j}$ and $\fs_{\vs_{j-1}}$ in terms
      of $\fs_{\ws}$'s, one gets for the left-hand side of the previous equation:
      \begin{multline*}
	(\KF_{\ws_j,\ws_{j-1}} + \KF_{\ws_j,\vs_{j-1}}\KF_{\vs_{j-1},\ws_{j-1}})
	\fs_{\ws_{j-1}}
	+
	(\KF_{\ws_j,\vs_j}\KF_{\vs_j,\ws_j} 
	+ \KF_{\ws_j,\vs_{j-1}}\KF_{\ws_j,\vs_{j-1}}) 
	\fs_{\ws_j} 
	\\+
	(\KF_{\ws_j,\ws_{j+1}} + \KF_{\ws_j,\vs_j} \KF_{\ws_{j+1}, \vs_j})
	\fs_{\ws_{j+1}}.
      \end{multline*}
      The coefficient in front of $\fs_{\ws_j}$,
      \begin{equation*}
	\KF_{\ws_j,\vs_j}\KF_{\vs_j,\ws_j} 
	+ \KF_{\ws_j,\vs_{j-1}}\KF_{\ws_j,\vs_{j-1}} = -(\KF_{\ws_j,\vs_j})^2 +
	(\KF_{\ws_j,\vs_{j-1}})^2 = -1+1,
      \end{equation*}
      is trivially equal to zero.
      Moreover, because of the condition on the orientation of the triangles in the
      Kasteleyn orientation, we have:
      \begin{equation}
      \label{eq:sum_coeff_0}
      	\KF_{\ws_j,\ws_{j-1}} + \KF_{\ws_j,\vs_{j-1}}\KF_{\vs_{j-1},\ws_{j-1}} =
	\KF_{\ws_j,\ws_{j+1}} + \KF_{\ws_j,\vs_j} \KF_{\ws_{j+1}, \vs_j} =0.
      \end{equation}
      Indeed, to check this, it is enough to look at the case where the edges of a
      triangle are all oriented clockwise, and notice that the quantity is
      invariant if we simultaneously change the orientation of any pair of edges
      of the triangle, which is a transitive operation on all the (six)
      clockwise odd orientations of a triangle. So $(\KF \fs)_{\ws_j}$ is identically zero.

    $\bullet$ If $\xs=\vs_j(\xb)=\vs_j$ for some $j$, then $\xs$ has three neighbors:
      $\ws_j(\xb)=\ws_j$, $\ws_{j+1}(\xb)=\ws_{j+1}$ and
      $\vs_\ell(\yb)=\vs_{\ell}$. Factoring out $\fs_{\zs}(u+2K)\expo_{(\zb,\xb)}(u)$, it is sufficient to prove that
      \begin{equation*}
	\KF_{\vs_j,\ws_j} \fs_{\ws_j}(u) +
	\KF_{\vs_j,\ws_{j+1}} \fs_{\ws_{j+1}}(u) +
	\KF_{\vs_j,\vs_\ell} \fs_{\vs_\ell}(u)\expo_{(\xb,\yb)}(u)=0.
      \end{equation*}
      Note that under inversion of the orientation of all edges around any of
      the vertices $\ws_j$, $\ws_{j+1}$ and $\vs_j$, all the signs of three terms
      either stay the same, or change at the same time. To fix ideas, we can
      thus suppose that the edges of the triangles
      $\ws_j,\ws_{j+1},\vs_j$ and $\ws_\ell,\ws_{\ell+1},\vs_\ell$ are all
      oriented clockwise, and that the edge between $\vs_j$ and $\vs_\ell$ is
      oriented from $\vs_\ell$ to $\vs_j$, as in Figure~\ref{Fig:fig_GF_1}.
      Returning to the definition of the angles mod $4\pi$, see \eqref{equ:def_angle} and \eqref{eq:relative_definition_a_a'}, and 
      simplifying notation, we obtain
      \begin{equation*}
      \alpha=\alpha_j(\xb),\ \  \beta=\alpha_{j+1}(\xb)=\alpha+2\theta,
      \ \ \alpha'=\alpha_{\ell}(\yb)=\alpha-2K,
      \ \ \beta'=\alpha_{\ell+1}(\yb)=\alpha'+2\theta=\beta-2K.
      \end{equation*}
      We have:
      \begin{equation*}
	\KF_{\vs_j,\ws_j} \fs_{\ws_j}(u) +
	\KF_{\vs_j,\ws_{j+1}} \fs_{\ws_{j+1}}(u) =
	\fs_{\ws_j}(u) - \fs_{\ws_{j+1}}(u) = 
	\frac{\cn(\frac{u-\beta}{2}) - \cn(\frac{u-\alpha}{2})}{\cn(\frac{u-\alpha}{2})\cn(\frac{u-\beta}{2})}.
      \end{equation*}
  On the other hand, \eqref{equ:rewriting_fs} entails that
      \begin{equation*}
      \textstyle
	\fs_{\vs_\ell}(u)=\fs_{\ws_{\ell}}(u)+\fs_{\ws_{\ell+1}}(u)=
	\nc(\frac{u-\alpha'}{2}) +\nc(\frac{u-\beta'}{2}) =\displaystyle
	-\frac{1}{k'}\frac{\sd(\frac{u-\alpha}{2})+\sd(\frac{u-\beta}{2})}{\sd(\frac{u-\alpha}{2})\sd(\frac{u-\beta}{2})}.
      \end{equation*}
      This has to be multiplied by $\KF_{\vs_j ,\vs_{\ell}} =
      -\frac{\sn\theta}{1+\cn\theta}$ and by the exponential function
      $\expo_{(\xb,\yb)}(u)$, so that:
      \begin{align*}
      \textstyle
	\KF_{\vs_j,\vs_\ell} \expo_{(\xb,\yb)}(u) \fs_{\vs_\ell}(u) &=
	\frac{\sn\theta}{1+\cn\theta}
	\frac{1}{k'}\frac{\sd(\frac{u-\alpha}{2})+\sd(\frac{u-\beta}{2})}{\sd(\frac{u-\alpha}{2})\sd(\frac{u-\beta}{2})}\textstyle
	(-k')\sc(\frac{u-\alpha}{2})\sc(\frac{u-\beta}{2})\\
	&=
	-\frac{\sn\theta}{1+\cn\theta}
      \frac{\sn(\frac{u-\alpha}{2})\dn(\frac{u-\beta}{2})+\sn(\frac{u-\beta}{2})\dn(\frac{u-\alpha}{2})} {\cn(\frac{u-\alpha}{2})\cn(\frac{u-\beta}{2})}.
      \end{align*}
  Proving that 
  \begin{equation}
  \label{eq:harmonicity_g}
  \KF_{\vs_j,\ws_j} \fs_{\ws_j}(u) +
	\KF_{\vs_j,\ws_{j+1}} \fs_{\ws_{j+1}}(u)+\KF_{\vs_j,\vs_\ell}
	\expo_{(\xb,\yb)}(u) \fs_{\vs_\ell}(u)=0
	 \end{equation}
	 amounts to showing that
	\begin{equation*}
	\textstyle
	  (1+\cn\theta)\bigl\{\cn(\frac{u-\beta}{2})-\cn(\frac{u-\alpha}{2})\bigr\}=
	\sn\theta\,\bigl\{\sn(\frac{u-\alpha}{2})\dn(\frac{u-\beta}{2})+
	\sn(\frac{u-\beta}{2})\dn(\frac{u-\alpha}{2})\bigr\}.
	\end{equation*}
	However, the addition formula (see Exercice 32 (v) in \cite[Chapter 2]{La89} and also the similar
	relation~\eqref{eq:(iii)32}) reads:
	\begin{equation*}
	  \cn(u+v)\cn u = \cn v-\sn(u+v)\sn u\dn v.
	\end{equation*}
	Evaluated at $u=\frac{u-\alpha}{2}$, $v=-\frac{u-\beta}{2}$ and $u+v=\theta$ (and
	exchanging the role of $\alpha$ and $\beta$ for the second equation), we obtain
	\begin{equation*}
	  \left\{\begin{array}{l}
	  \textstyle\cn(\frac{u-\beta}{2})\cn\theta = \cn(\frac{u-\alpha}{2}) + \sn(\frac{u-\beta}{2})\dn(\frac{u-\alpha}{2})
	  \sn\theta,\medskip\\
	  \textstyle\cn(\frac{u-\alpha}{2})\cn\theta = \cn(\frac{u-\beta}{2})
	  -\sn(\frac{u-\alpha}{2})\dn(\frac{u-\beta}{2})\sn\theta.
	  \end{array}\right.
	\end{equation*}
	Taking the difference of these two equations yields the result.
\end{proof}

\begin{rem}
  \label{rem:int_g_mesure}
  If $\mu$ is a measure on $\TT(k)$ and if we define $\hs_{\xs,\ys} = \int
  \gs_{\xs,\ys}(u) \ud\mu(u)$, then by linearity of the integral, one also has
  $\KF\hs =\hs\KF =0$. A particular case, which will be important for what
  follows (see also~\cite{Kenyon3,BoutillierdeTiliere:iso_gen}), is the case when
  $\mu$ is the integration along a contour on $\TT(k)$.
\end{rem}

\subsection{Local expression for the inverse of the Kasteleyn operator $\KF$}\label{sec:localKFinv}

We now state Theorem~\ref{thm:KFmoins_un}, proving an explicit, local formula
for coefficients of the inverse ${\KF}^{-1}$ of the Kasteleyn operator
$\KF$. 
 This formula is
constructed from the function $\gs$ of Definition~\ref{def:fonction_g}.

\begin{thm}
\label{thm:KFmoins_un}
Consider the dimer model on the Fisher graph $\GF$ arising from the
$Z$-invariant Ising model on the isoradial graph $\Gs$, and let 
$\KF$ be the corresponding Kasteleyn operator.
Define the operator ${\KF}^{-1}$ by its coefficients:
\begin{align}
\forall\,\xs,\ys\in\VF,\quad     \Ks^{-1}_{\xs,\ys}&=
\frac{ik'}{8\pi}\int_{\Gamma_{\xs,\ys}} \gs_{\xs,\ys}(u)\ud u +C_{\xs,\ys}\\
&=\frac{ik'}{8\pi}\int_{\Gamma_{\xs,\ys}} \fs_\xs(u+2K)\fs_{\ys}(u)\expo_{(\xb,\yb)}(u)\ud u+C_{\xs,\ys},
\label{equ:KF_inverse}
\end{align}
where
the contour of integration $\Gamma_{\xs,\ys}$ is a simple closed curve winding
once vertically around the torus $\TT(k)$ (along which the second coordinate globally
increases), which intersects the horizontal axis in the angular sector
(interval) $s_{x,y}$ of length larger than or equal to $2K$ (see
Section~\ref{subsubsec:sectors}),
and the constant $C_{\xs,\ys}$ is given by
\begin{equation}
\label{eq:expression_C_x_y}
C_{\xs,\ys}=
-\frac{1}{4}\cdot
\begin{cases}
1 &\text{if $\xs=\ys=\ws_j(\xb)$,}\\
(-1)^{n(\xs,\ys)} & \text{if $\xs=\ws_j(\xb)$ and $\ys=\ws_\ell(\xb)$ for some $j\neq \ell$,}\\
0&\text{otherwise,}
\end{cases}
\end{equation}
where $n(\xs,\ys)$ is the number of edges oriented clockwise in the counterclockwise arc from $\xs$
to $\ys$ in the inner decoration.

Then ${\KF}^{-1}$ is an inverse of the Kasteleyn operator $\KF$ on $\GF$.

When $k\neq 0$, it is the unique inverse with bounded coefficients.

Alternatively, the coefficients of the inverse of the Kasteleyn operator admit
the expression
\begin{equation}
\label{equ:KF_inverseH}
     \Ks^{-1}_{\xs,\ys}=
\frac{ik'}{8\pi}\oint_{\C_{\xs,\ys}} \fs_\xs(u+2K)\fs_{\ys}(u)\expo_{(\xb,\yb)}(u)\Hh(u)\ud u+C_{\xs,\ys},
\end{equation}
where the function $\Hh$ is defined in~\eqref{eq:definition_Hh_Hv_k2>0}--\eqref{eq:definition_Hh_Hv_k2<0},
$\C_{\xs,\ys}$ is a trivial contour oriented counterclockwise on the torus, not crossing
$\Gamma_{\xs,\ys}$ and containing in its interior all the poles of
$\gs_{(\xs,\ys)}$ and the pole of $\Hh$, and $C_{\xs,\ys}$ is defined in
\eqref{eq:expression_C_x_y}. 
\end{thm}

Before we go on with the proof of this theorem, let us make a few comments about
the formula of the inverse Kasteleyn matrix:
\begin{itemize}
  \item As soon as $\xs$ and $\ys$ are not in the same decoration, or one of
    them is of type $\vs$, then the constant $C_{\xs,\ys}$ is zero, 
    and the formula for $\KF^{-1}_{\xs,\ys}$ as a contour integral has the same
    flavour as the Green function of the
    $Z$-invariant massive Laplacian introduced
    in~\cite[Theorem~12]{BdTR1}. 
\item The constant $C_{\xs,\xs}$ is here to ensure that $\KF^{-1}_{\xs,\xs}$
    is $0$ if $\xs$ is of type $\ws$ (the integral is $0$ when $\xs$ is of type
    $\vs$ as we shall see later).
  \item As one can expect, the full formula is skew-symmetric in $\xs$ and $\ys$.
  \item To obtain the alternative expression \eqref{equ:KF_inverseH} from
    \eqref{equ:KF_inverse}, one can make use of a meromorphic multivalued
    function with a horizontal period of $1$, like the function $\Hh$ defined in
    \eqref{eq:definition_Hh_Hv_k2>0}--\eqref{eq:definition_Hh_Hv_k2<0},
    originally introduced in~\cite{BdTR1} for $k^2\in(0,1)$. Following this way, one may rewrite
    the integral as an integral over a contour bounding a disk, allowing one to
    perform explicit computation with Cauchy's residue theorem. One can add to
    $\Hh$ any elliptic function on $\TT(k)$ without changing the result of the
    integral, given that $\C_{\xs,\ys}$ encloses all the poles of the new
    integrand.
  \item Adding to the columns of $\KF^{-1}$ functions in the kernel of $\KF$
    yield other inverses, with different behaviour at infinity. Such a function
    in the kernel is obtained by integrating $\gs_{(\xs,\ys)}(\cdot)$ along a
    horizontal contour in $\TT(k)$, see Remark~\ref{rem:int_g_mesure}. As a
    consequence, if we replace in \eqref{equ:KF_inverse} the contour
    $\Gamma_{\xs,\ys}$ by a contour winding $a$ times vertically and $b$ times
    horizontally, with $a$ and $b$ coprimes, and divide the integral by $a$,
    then we get a new inverse for the Kasteleyn operator $\KF$, which has an
    alternative expression as a trivial contour integral involving integer linear combinations
    of functions $\Hh$ and $\Hv$, as defined in Appendix \ref{app:HhHv}.
  \item When $k=0$, the ``torus'' $\TT(k)$ is in fact a cylinder, with two points
    at infinity. The contour $\Gamma_{\xs,\ys}$ has infinite length. The
    function $\gs_{\xs,\ys}(u)$ decays sufficiently fast at infinity to ensure
    convergence of the integral. By performing the change of variable
    $\lambda=-e^{iu}$ in the integrals \eqref{equ:KF_inverse} or
    \eqref{equ:KF_inverseH}, one gets the adaptation to this variant of the
    Fisher graph of the formula for the inverse Kasteleyn operator
    in~\cite{BoutillierdeTiliere:iso_gen}, as an integral along a ray from $0$ to
    $\infty$, or as an integral over a closed contour with a $\log$.
\end{itemize}

We now turn to the proof of~Theorem~\ref{thm:KFmoins_un}. We show that the
operator $\KF^{-1}$ with those coefficients satisfy $\KF \KF^{-1} =
\operatorname{Id}$ and 
$\KF^{-1} \KF = \operatorname{Id}$. These identities, understood as products of
infinite matrices, make sense since $\KF$ has a finite number of non zero
coefficients on each row and column. Moreover, by skew-symmetry, it is enough to
check the first one.
When $k\neq 0$, it turns out that these coefficients for $\KF^{-1}$ go to zero
exponentially
fast, see Theorem~\ref{thm:asymptotics_inverse_Kasteleyn}. This property
together with $\KF^{-1} \KF = \operatorname{Id}$ imply injectivity of $\KF$ on
the space of bounded functions on vertices of $\GF$, which in turn implies
uniquess of an inverse with bounded coefficients.

The general idea for proving $\KF \KF^{-1}=\operatorname{Id}$ follows~\cite{Kenyon3}, but it is complicated by 
the fact that the Fisher graph $\GF$ itself is not isoradial. In this respect, the proof follows more closely that
of Theorem~5 of~\cite{BoutillierdeTiliere:iso_gen}
with two main differences: we work with a different Fisher graph $\GF$ and more importantly we handle the elliptic case, making it a non-trivial extension.
Section~\ref{subsubsec:encoding_poles} corresponds to Sections 6.3.1 and 6.3.2 of~\cite{BoutillierdeTiliere:iso_gen}. It consists  
in the delicate issue of encoding the poles of the integrand $\fs_\xs(u+2K)\fs_{\ys}(u)\expo_{(\xb,\yb)}(u)$; for this question there are no
additional difficulties so that we have made it as short as possible and refer to the paper~\cite{BoutillierdeTiliere:iso_gen} for 
more details and figures. Section~\ref{subsubsec:sectors} consists in obtaining 
a sector $s_{\xs,\ys}$ on the horizontal axis of the torus $\TT(k)$ from the encoding of the poles; this is then used to define the contour
of integration $\Gamma_{\xs,\ys}$. It corresponds to Section 6.3.3 
of~\cite{BoutillierdeTiliere:iso_gen} but requires adaptations to pass to the elliptic case. Section~\ref{subsubsec:proof_main_thm} is a 
non-trivial adaptation of Section 6.4 of~\cite{BoutillierdeTiliere:iso_gen}, handling a different Fisher graph $\GF$ and more importantly 
handling the elliptic case.

\subsubsection{Preliminaries: encoding the poles of the integrand}
\label{subsubsec:encoding_poles}

Let $\Gs$ be an infinite isoradial graph and let $\GR$ be the corresponding diamond graph. 
In order to encode poles of the integrand of $\Ks^{-1}_{\xs,\ys}$, we need the notion of \emph{minimal path}
which relies on the notion of train-tracks, see Section~\ref{sec:iso_star_triangle} for definition.
A train-track is said to \emph{separate} two vertices $\xb,\yb$ of $\GR$ if every path connecting $\xb$ and $\yb$ crosses this train-track.
A path from $\xb$ to $\yb$ in $\GR$ is said to be \emph{minimal} if all its edges cross train-tracks that separate $\xb$ from $\yb$, 
and each such train-track is crossed exactly once.
A minimal path from $\xb$ to $\yb$ is in fact a geodesic for the graph metric on
$\GR$. Since $\GR$ is connected, it always exists. In general, there are several
minimal paths between two vertices, but they all consist of the same steps taken
in a different order.

For every pair of vertices $\xs,\ys$ of $\GF$, we now define an 
edge-path $\gamma_{\xs,\ys}$ of $\GR$ encoding the poles of the integrand of 
$\KF^{-1}_{\xs,\ys}$.
Consider a minimal path from $\xb$ to $\yb$
and let $e^{i\overline{\alpha}_{\ell}}$ be one of the steps of the path,
then the corresponding pole of the exponential function is $\alpha_\ell+2K$. Since 
$e^{i\overline{\alpha_\ell+2K}}=e^{i\overline{\alpha_\ell}+i\pi}=-e^{i\overline{\alpha}_{\ell}}$, this pole is encoded in the reverse step. As a consequence,
poles of the exponential function are encoded in the steps $\{-e^{i\overline{\alpha}_{\ell}}\}$ of a minimal path from $\yb$ to $\xb$.

We now have to add the poles of the functions $\fs_\xs(u+2K)$ and $\fs_\ys(u)$. The difficulty lies in the fact that some of them might be canceled by 
factors in the numerator of the exponential function. By definition, the function $\fs_\xs(u+2K)$ has either one or two poles $\{\alpha_j\}$, encoded in the edge(s) 
$\{e^{i\overline{\alpha}_{j}}\}$ of the diamond graph $\GR$; let $T_\xs=\{T_\xs^j\}$ be the corresponding 
train-track(s). Similarly, the pole(s) of $\fs_\ys(u)$ are at $\{\alpha_j'+2K\}$ and are encoded in the edge(s) $\{-e^{i\overline{\alpha}_{j}'}\}$ of $\GR$, and 
$T_\ys=\{T_\ys^j\}$ are the corresponding train-track(s). 

Let us start from a minimal path $\gamma_{\xs,\ys}$ from $\yb$ to $\xb$. For every $j$, do the following procedure: if $T_\ys^j$ separates $\yb$ from $\xb$, then
the pole $\alpha_j'+2K$ is canceled by the exponential and we leave $\gamma_{\xs,\ys}$ unchanged. If not, this pole remains, and we extend 
$\gamma_{\xs,\ys}$ by adding the step $-e^{i\overline{\alpha}_j'}$ at the beginning of $\gamma_{\xs,\ys}$. The path $\gamma_{\xs,\ys}$ 
obtained is still a path of $\GR$,
denote by $\widehat{\yb}$ the new starting point, at distance at most $2$ from $\yb$. 

When dealing with a pole of $\fs_{\xs}(u+2K)$, one needs to be careful since, even when the corresponding train-track separates $\yb$ from $\xb$,
the exponential function might not cancel the pole, if it has already canceled the same pole of $\fs_{\ys}(u)$; this happens when $T_\xs$ and 
$T_\ys$ have a common train-track. The procedure to extend $\gamma_{\xs,\ys}$ runs as follows: for each $j$, if $T_\xs^j$ separates $\yb$ and $\xb$
and is not a train track of $T_\ys$, then the pole $\alpha_j$ is canceled by the exponential function, and we leave $\gamma_{\xs,\ys}$ unchanged.
If not, this pole remains, and we extend $\gamma_{\xs,\ys}$ by attaching the step $e^{i\overline{\alpha}_j}$ at the end of $\gamma_{\xs,\ys}$. The path
obtained in this way is still a path of $\GR$, starting from $\widehat{\yb}$. Denote by $\widehat{\xb}$ its ending point, which is at distance at most 
$2$ from $\xb$.

\subsubsection{Obtaining a sector $s_{\xs,\ys}$ from $\gamma_{\xs,\ys}$}
\label{subsubsec:sectors}

Let $\xs,\ys$ be two vertices of $\GF$ and let $\gamma_{\xs,\ys}$ 
be the path encoding poles of the integrand of 
$\Ks^{-1}_{\xs,\ys}$ constructed above. Denote by 
$\{e^{i\overline{\tau}_j}\}$ the steps of the path, seen as vectors
in the unit disk; the corresponding poles of the integrand are $\{\tau_j\}$. 
Using these poles, we now define an interval/sector $s_{\xs,\ys}$ in the horizontal axis $\mathbb{R}/4K\mathbb{Z}$ of the torus $\TT(k)$. 
Given the sector $s_{\xs,\ys}$, the contour 
of integration $\Gamma_{\xs,\ys}$ of $\Ks^{-1}_{\xs,\ys}$ is then defined to be  
a simple closed curve winding once around the torus vertically, \emph{i.e.}, in the direction $i$,
along which the second coordinate globally increases, and which intersects the
horizontal axis in $s_{\xs,\ys}$, see Figure~\ref{Fig:repr_sectors} below.

\emph{General case}. This case contains all but the three mentioned below. We
know by Lemmas~17 and 18 of~\cite{BoutillierdeTiliere:iso_gen} that
there exists a sector in the unit circle, of size greater than or equal to $\pi$, containing none of the steps 
$\{e^{i\overline{\tau}_{j}}\}$. Equivalently, there exists a sector in
the horizontal axis $\mathbb{R}/4K\mathbb{Z}$ of the torus $\TT(k)$, of size larger than or equal to $2K$,
containing none of the poles $\{\tau_j\}$. We let $s_{\xs,\ys}$ be this sector, it is represented in Figure~\ref{Fig:repr_sectors}.

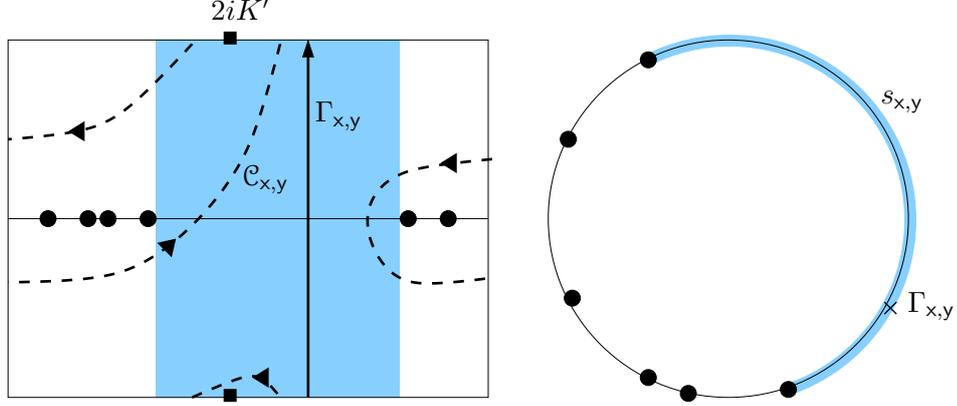
\begin{figure}[htb]
  \centering\smallskip
  \begin{overpic}[width=12cm]{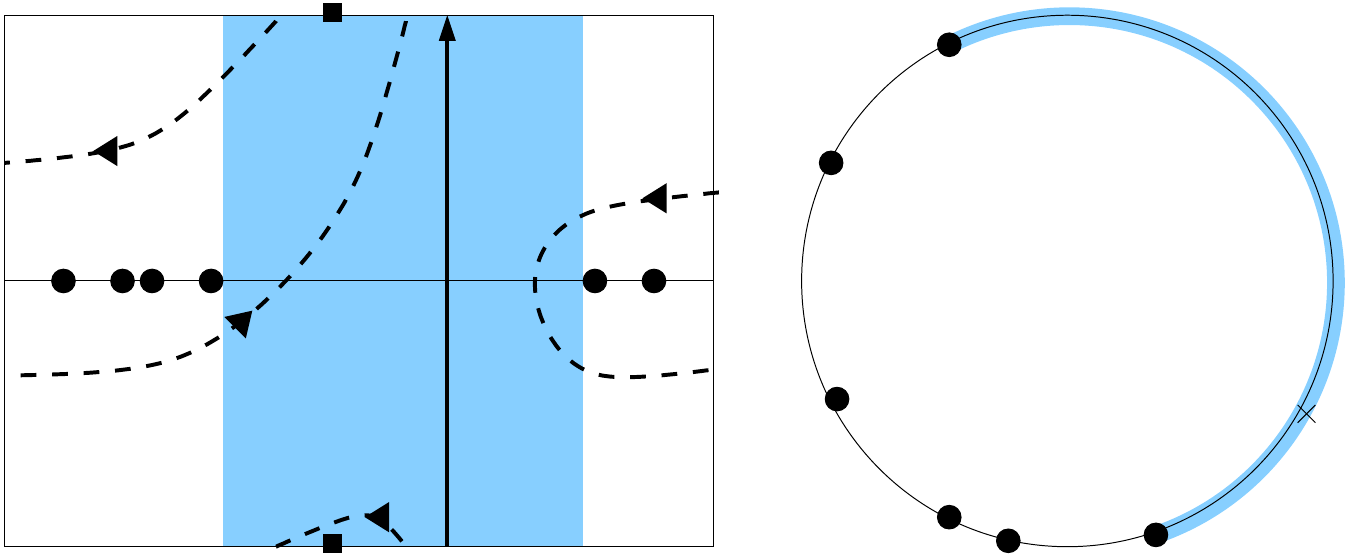}
  \put(22.6,42.3){$2iK'$}
  \put(34,31){$\Gamma_{\xs,\ys}$}
  \put(26,24){$\C_{\xs,\ys}$}
  \put(96,33){$s_{\xs,\ys}$}
  \put(99,10){$\Gamma_{\xs,\ys}$}
  \end{overpic}
  \caption{Left: the torus $\TT(k)$ with the contours of integration $\Gamma_{\xs,\ys}$ and $\C_{\xs,\ys}$;
  the poles of the integrand $\fs_\xs(u+2K)\fs_{\ys}(u)\expo_{(\xb,\yb)}(u)$ are represented by
  black bullets and the pole $2iK'$ of the function $\Hh$ by a black square. Right: in blue the sector $s_{\xs,\ys}$ used
  to define the contour of integration $\Gamma_{\xs,\ys}$.}
  \label{Fig:repr_sectors}
\end{figure}

Here are the three cases which do not fit in the general situation. 

\emph{Case 1}. \label{pageref:case_1_2}The path $\gamma_{\xs,\ys}$ consists of two steps that are opposite.
Then, the two poles separate
the real axis of $\TT(k)$ in two sectors of size exactly $2K$, leaving an ambiguity.
This can only happen when $\xs=\ys=\ws_j(\xb)$ and the two poles
are $\{\alpha_j$, $\alpha_j+2K\}$.
In this case, the \emph{standard convention}
is that\footnote{When indicating sectors on the circle, the convention we adopt is that
  $(\alpha,\beta)$ represents the sector where the horizontal coordinate
  increases from $\alpha$ to $\beta$.
} $s_{\ws_j,\ws_j}=(\alpha_j,\alpha_j+2K)$, 
and by Lemma~\ref{lem:integrale_fisher} we have
\begin{equation}
\label{equ:case_1}
     \frac{ik'}{8\pi}\int_{\Gamma_{\ws_j,\ws_j}}\fs_{\ws_j}(u+2K)\fs_{\ws_j}(u)\ud u =\frac{1}{4}.
\end{equation}
We choose the value of the constant $C_{\ws_j,\ws_j}$ to compensate exactly the value of this integral
so as to have $\Ks^{-1}_{\ws_j,\ws_j}=0$, that is $C_{\ws_j,\ws_j}=-\frac{1}{4}$, 
and we recover the first line of the definition of 
$C_{\xs,\ys}$ of Equation~\eqref{eq:expression_C_x_y}.

\begin{figure}[ht]
  \centering
\begin{overpic}[width=3.5cm]{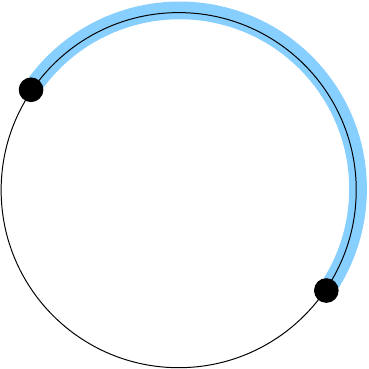}
 \put(77,23){\scriptsize $\alpha_j$}
 \put(13,72){\scriptsize $\alpha_j+2K$}
\end{overpic}
\caption{Standard convention for the definition of the sector $s_{\ws_j,\ws_j}$.\label{Fig:fig_preuve_1}}
\end{figure}

\begin{rem}\label{rem:Case1}
It will be useful for the proof to consider also the \emph{non-standard convention} 
with the complementary sector, defining a contour
$\Gamma_{\ws_j,\ws_j}'$. Returning to the definition of the function 
$\fs_{\ws_j}(u)$, we have that the integral over $\Gamma_{\ws_j,\ws_j}'$
is equal to minus the one on the contour $\Gamma_{\ws_j,\ws_j}$, 
so that in order to have $\Ks^{-1}_{\ws_j,\ws_j}=0$, we set
$C_{\ws_j,\ws_j}'=-C_{\ws_j,\ws_j}$.
\end{rem}

The two other cases correspond to situations when $\vert T_\xs\bigcap T_\ys\vert =2$. The
corresponding path $\gamma_{\xs,\ys}$ does not enter the
framework of Lemmas 17 and 18 of~\cite{BoutillierdeTiliere:iso_gen}.
They occur when $\xs$ and $\ys$ are equal or are neighbors in $\GF$, and both of type
`$\vs$'.
In other words, one has $\xs=\ys=\vs_j(\xb)$, or 
$(\xs,\ys)=(\vs_j(\xb),\vs_{\ell}(\yb))$, with $\xb\sim\yb$ in $\Gs$, $j$ and
$\ell$ being such that $\vs_j(\xb)\sim \vs_{\ell}(\yb)$ in $\GF$.

\emph{Case 2}. Suppose first that $\xs=\ys=\vs_j(\xb)$. Then the 
poles are $\{\alpha_j,\alpha_j+2K,\alpha_{j+1},\alpha_{j+1}+2K\}$ 
(the exponential function is equal to $1$ and
cancels no pole). If we take $s_{\vs_j,\vs_j}=(\alpha_{j},\alpha_{j+1})$, then
  the integral is zero by symmetry. Indeed, the change of variable $u\to \alpha_{j+1}+\alpha_{j}- u$ leaves
  the contour invariant (up to homotopy)
  and $\fs_{\vs_j}(u+2K)$ is changed into
  its opposite, whereas $\fs_{\vs_j}(u)$ is invariant. Note that taking
  $s_{\vs_j,\vs_j}=(\alpha_{j}+2K,\alpha_{j+1} + 2K)$ also gives a zero integral,
  because it is related to the previous one by the change of variable $u\to
  u+2K$. These two choices of sectors will be useful in the proof of Theorem~\ref{thm:KFmoins_un}, see Figure~\ref{Fig:fig_preuve_2} (center).

\emph{Case 3}.
Suppose now that $(\xs,\ys)=(\vs_j(\xb),\vs_{\ell}(\yb))$. Then $\fs_{\vs_j(\xb)}$ and $\fs_{\vs_\ell(\yb)}$ induce twice the same 
poles $\{\alpha_j,\alpha_{j+1}\}$. The exponential adds the poles $\{\alpha_j+2K,\alpha_{j+1}+2K\}$ and the 
numerator cancels one pair of $\{\alpha_j,\alpha_{j+1}\}$, implying that there remains the poles
$\{\alpha_j,\alpha_j+2K,\alpha_{j+1},\alpha_{j+1}+2K\}$. We set the convention given in Figure~\ref{Fig:fig_preuve_2} (right).

\begin{figure}[ht]
  \centering
\begin{overpic}[width=12cm]{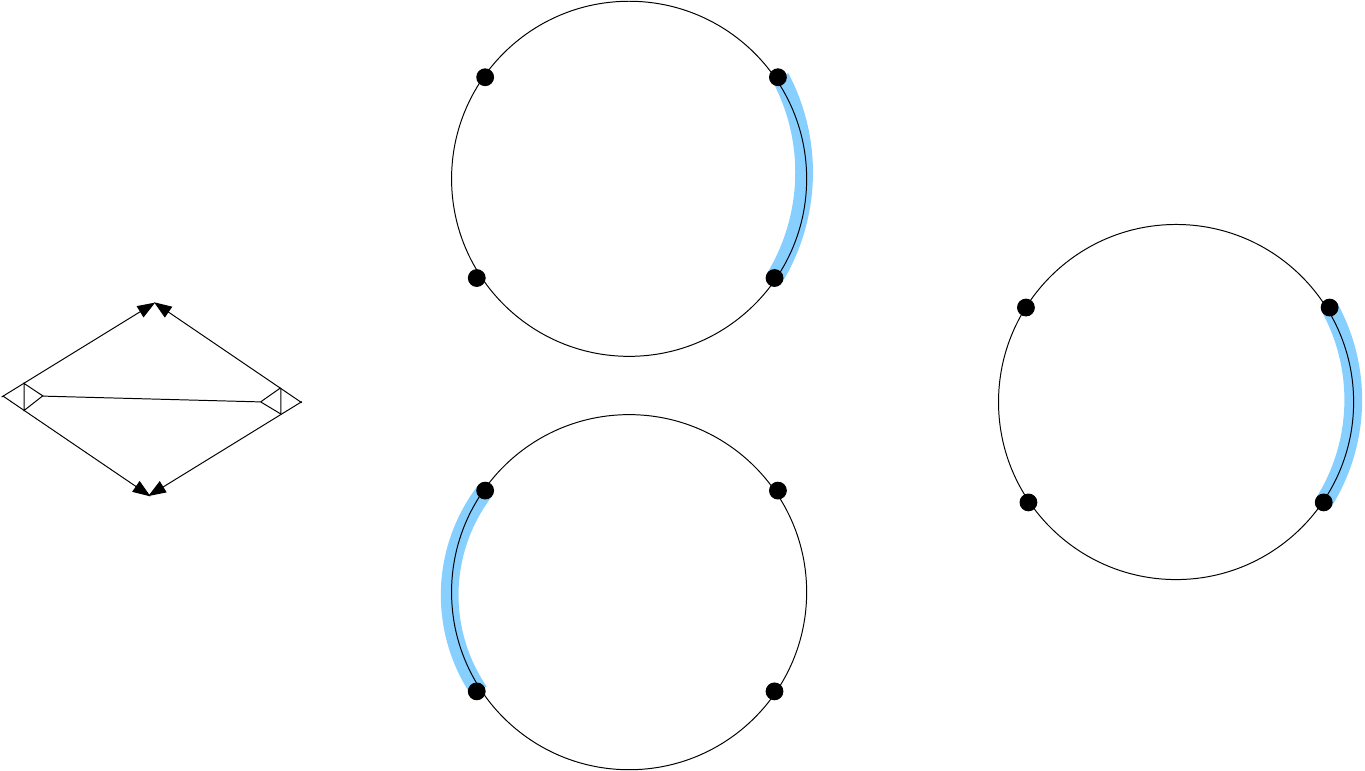}
\put(37,50){\scriptsize $\alpha_j+2K$}
\put(37,36){\scriptsize $\alpha_{j+1}+2K$}
\put(59,50){\scriptsize $\alpha_{j+1}$}
\put(59,36){\scriptsize $\alpha_j$}

\put(37,20){\scriptsize $\alpha_j+2K$}
\put(37,06){\scriptsize $\alpha_{j+1}+2K$}
\put(59,20){\scriptsize $\alpha_{j+1}$}
\put(59,06){\scriptsize $\alpha_j$}

\put(77,34){\scriptsize $\alpha_j+2K$}
\put(77,20){\scriptsize $\alpha_{j+1}+2K$}
\put(99,34){\scriptsize $\alpha_{j+1}$}
\put(99,20){\scriptsize $\alpha_j$}

\put(4,25){\scriptsize $\vs_j$}
\put(3,33){\scriptsize $e^{i\overline{\alpha}_{j+1}}$}
\put(3,20){\scriptsize $e^{i\overline{\alpha}_j}$}
\put(14,20){\scriptsize $-e^{i\overline{\alpha}_{j+1}}$}
\put(14,33){\scriptsize $-e^{i\overline{\alpha}_j}$}
\put(16,25){\scriptsize $\vs_\ell$}
\put(-3,26){\scriptsize $\xb$}
\put(23,26){\scriptsize $\yb$}
\put(43,28){\scriptsize $s_{\vs_j,\vs_j}$}
\put(84,11){\scriptsize $s_{\vs_j,\vs_\ell}$}
\end{overpic}
\caption{Definition of the sectors $s_{\vs_j,\vs_j}$ (center) and
$s_{\vs_j,\vs_\ell}$ (right).\label{Fig:fig_preuve_2}} 
\end{figure}
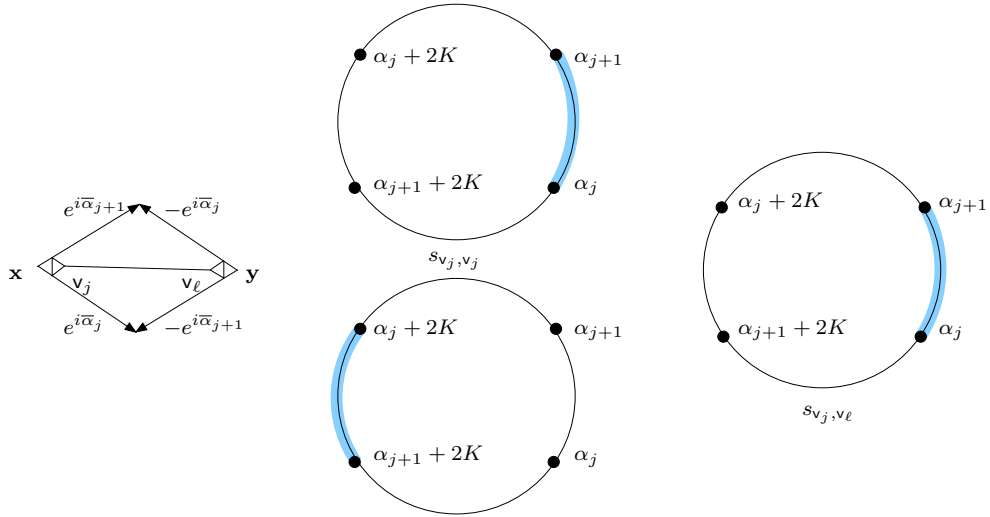

\subsubsection{Proof of the local formula for $\Ks^{-1}$ of
Theorem~\ref{thm:KFmoins_un}}\label{subsubsec:proof_main_thm}

We need to prove that
\begin{equation*}
\forall\,\xs,\ys\in\VF,\quad (\Ks\Ks^{-1})_{\xs,\ys}=\delta_{\xs,\ys}.
\end{equation*}

We use the following notation.
The vertex $\ys$ is $\ys=\ws_j(\yb)$ or $\vs_j(\yb)$, for some vertex $\yb$ of
$\Gs$ and some $j\in\{1,\dots,d(\yb)\}$.
If $\ys=\ws_j(\yb)$, it has four
neighbors $\ws_{j-1}(\yb),\vs_{j-1}(\yb),\vs_j(\yb),\ws_{j+1}(\yb)$;
if $\ys=\vs_j(\yb)$, it has three neighbors
$\ws_{j}(\yb),\ws_{j+1}(\yb),\vs_\ell(\yb')$, see Figure~\ref{fig:fig_preuve_9}.
We denote by $\xs_i$ the neighbors of~$\xs$, with $i$ ranging from $1$ to $3$ or $4$.

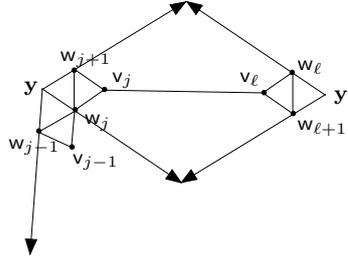
\begin{figure}[ht!]
  \centering
\begin{overpic}[width=4cm]{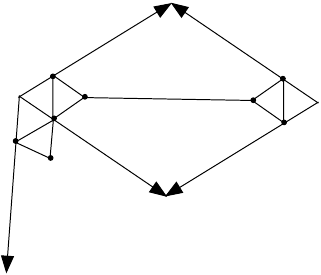}
\put(0,55){\scriptsize{$\yb$}}
\put(102,52){\scriptsize{$\yb'$}}
\put(20,44){\scriptsize{$\ws_j$}}
\put(29,58){\scriptsize{$\vs_j$}}
\put(16,31){\scriptsize{$\vs_{j-1}$}}
\put(-5,36){\scriptsize{$\ws_{j-1}$}}
\put(12,65){\scriptsize{$\ws_{j+1}$}}
\put(71,58){\scriptsize{$\vs_\ell$}}
\put(90,62){\scriptsize{$\ws_\ell$}}
\put(90,42){\scriptsize{$\ws_{\ell+1}$}}
\end{overpic}
\caption{Notation for the cases where the general argument for proving $(\KF\KF^{-1})_{\xs,\ys}=\delta_{\xs,\ys}$ does not work and which have to be treated
separately. \label{fig:fig_preuve_9}}
\end{figure}

As long as the computation of $(\KF \KF^{-1})_{\xs,\ys}=\sum_{i}\KF_{\xs,\xs_i}\KF^{-1}_{\xs_i,\ys}$ 
  only involves
  terms $\KF^{-1}_{\xs_i,\ys}$ for which the constant $C_{\xs_i,\ys}$ is $0$, and
  the sector $s_{\xs_i,\ys}$ defining the contour $\Gamma_{\xs_i,\ys}$ does not use any special convention, that is when
  \begin{itemize}
    \item $\xs$ is not in the same decoration as $\ys$, if $\ys=\ws_j(\yb)$,
    \item $\xs\notin\{\ws_j(\yb), \ws_{j+1}(\yb), \vs_\ell(\yb')\} \bigcup
      \{\ws_\ell(\yb'),\ws_{\ell+1}(\yb'),\vs_j(\yb)\}$, if $\ys=\vs_j(\yb)$,
  \end{itemize}
  then
  by the argument of \cite{Kenyon3,BoutillierdeTiliere:iso_gen}, all contours of
  integration $\Gamma_{\xs_i,\ys}$
  can be deformed into a common contour $\Gamma$, crossing the horizontal axis
  in the nonempty intersection of the sectors $\bigcap_i s_{\xs_i,\ys}$, so
  that by Proposition~\ref{prop:function_f_kernel_Kasteleyn_Fisher} (see also Remark~\ref{rem:int_g_mesure}) we have:
  \begin{equation*}
    \sum_{i} \Ks_{\xs,\xs_i}
    \oint_{\Gamma_{\xs_i,\ys}} \fs_{\xs_i}(u+2K) \fs_{\ys}(u) \expo_{(\xb,\yb)}(u) \ud u 
    = \oint_{\Gamma} \sum_{i} \Ks_{\xs,\xs_i} \fs_{\xs_i}(u+2K)
    \fs_{\ys}(u)\expo_{(\xb,\yb)}(u) \ud u 
    = 0. 
  \end{equation*}
  Let us check the remaining cases separately. 
\paragraph{Suppose that $\ys=\ws_j(\yb)$.} The degree of the vertex $\yb$ is $d(\yb)$ and indices below should be 
thought of as being modulo $d$. We have to handle all cases where the vertex $\xs$ belongs to the decoration
$\yb$, whether it is of type `$\vs$' or `$\ws$'.

$\bullet$ We first compute $(\KF \KF^{-1})_{\xs,\ws_j}$ when $\xs=\vs_{r}(\yb)$ for some $r\in\{1,\dots,d(\yb)\}$. 
The vertex $\xs$ has three neighbors $\ws_r(\yb),\ws_{r+1}(\yb)$
and a vertex of type `$\vs$' in a neighboring decoration. We now omit the argument $\yb$ from the notation.

When $r\in\{j-d+1,\dots,j-2\}$, we are in the general case of the definition of $\Gamma_{\xs_i,\ys}$;
when $r=j-1$, we choose the standard convention of Case 1, that is $\Gamma_{\ws_j,\ws_j}$ and $C_{\ws_j,\ws_j}$; when $r=j$, we choose the equivalent, non-standard convention 
of Case 1, that is $\Gamma_{\ws_j,\ws_j}'$ and $C_{\ws_j,\ws_j}'$.

With these choices, the three sectors
appearing in the expressions of $\KF^{-1}_{\xs_i,\ws_j}$
have non-empty intersection, so that
the contours $\Gamma_{\xs_i,\ys}$ in the three integrals can be 
continuously deformed into the same contour $\Gamma$, 
and thus the combination of the integral parts gives zero. 

Since vertices of type `$\vs$' have no constant contribution $C_{\xs_i,\ys}$, we are left with proving that
\begin{equation}
\label{equ:preuve_vj_wj_0}
     \forall\,r\in\{j-d+1,\dots,j-1\},\quad 
     \KF_{\vs_r,\ws_r}C_{\ws_r,\ws_j}+\KF_{\vs_r,\ws_{r+1}}C_{\ws_{r+1},\ws_j}=0,
\end{equation}
and that
\begin{equation}
\label{equ:preuve_vj_wj_1}    
\KF_{\vs_{j},\ws_j}C_{\ws_j,\ws_j}'+\KF_{\vs_{j},\ws_{j-d+1}}C_{\ws_{j-d+1},\ws_j}=0.
\end{equation}
Multiplying each of the equations of \eqref{equ:preuve_vj_wj_0} by $\KF_{\vs_r,\ws_{r}}$, and using that 
$-\KF_{\vs_r,\ws_r}\KF_{\vs_r,\ws_{r+1}}=\KF_{\ws_{r},\ws_{r+1}}$ by the clockwise odd condition on triangles,
we have that the first set of equations is equivalent to, for all $r\in\{j-d+1,\dots,j-1\}$, $C_{\ws_{r},\ws_j}=\KF_{\ws_{r},\ws_{r+1}}C_{\ws_{r+1},\ws_j}$, which in turn holds if and only if
\begin{equation}
\label{equ:rec_C}
     C_{\ws_{r},\ws_j}=\Biggl( \prod_{m=r}^{j-1} \KF_{\ws_{m},\ws_{m+1}} \Biggr) C_{\ws_j,\ws_j}
=(-1)^{n(\ws_{r},\ws_{j})}C_{\ws_j,\ws_j}.
\end{equation}
Recalling that $C_{\ws_j,\ws_j}=-\frac{1}{4}$ and returning to the second line
of the definition of $C_{\xs,\ys}$, we see that this is indeed the case, whence
\eqref{equ:preuve_vj_wj_0} is proved.

We are left with proving that Equation~\eqref{equ:preuve_vj_wj_1} is satisfied.
Doing the same steps as above, and using that
$C_{\ws_{j-d+1},\ws_j}=(-1)^{n(\ws_{j-d+1},\ws_{j})}C_{\ws_j,\ws_j}$, this is
equivalent to proving that 
\[
C_{\ws_j,\ws_j}'=C_{\ws_j,\ws_j}(-1)^{n(\ws_{j-d+1},\ws_{j})}\KF_{\ws_j,\ws_{j-d+1}}.
\]
Observing that $(-1)^{n(\ws_{j-d+1},\ws_j)} \KF_{\ws_j,\ws_{j-d+1}}=-1$ because
of the clockwise odd condition on the inner circle of decorations, and
recalling that $C_{\ws_j,\ws_j}'=-C_{\ws_j,\ws_j}$ (see Remark~\ref{rem:Case1}),
we deduce that this equation is indeed true, thus ending the proof when
$\xs=\vs_{r}(\yb)$.

$\bullet$ We now compute $(\KF \KF^{-1})_{\xs,\ws_j}$ when $\xs=\ws_{r}(\yb)$ for some $r\in\{1,\dots,d(\yb)\}$. 
The vertex $\xs=\ws_r(\yb)$ has four neighbors: $\ws_{r-1}(\yb)$, $\ws_{r+1}(\yb)$,
$\vs_{r-1}(\yb)$ and $\vs_r(\yb)$, and we now omit the argument $\yb$.

Let us first handle the integral part. When $r\neq j+1$, we are either in the
general case of the definition of $\Gamma_{\xs_i,\ys}$ or in
Case 1, and we choose the standard definition. When $r=j+1$, we choose the
non-standard definition of Case 1, that
is $\Gamma_{\ws_j,\ws_j}'$ and $C_{\ws_j,\ws_j}'$. With these choices, as long
as $r\neq j$, the four sectors have non-empty intersection,
so that the combination of the integral parts is equal to zero. 

When $r=j$, then the four sectors enter the framework of the general case and
are {not} compatible, see
Figure~\ref{fig:fig_preuve_6}.

\begin{figure}[ht!]
  \centering
\begin{overpic}[width=\linewidth]{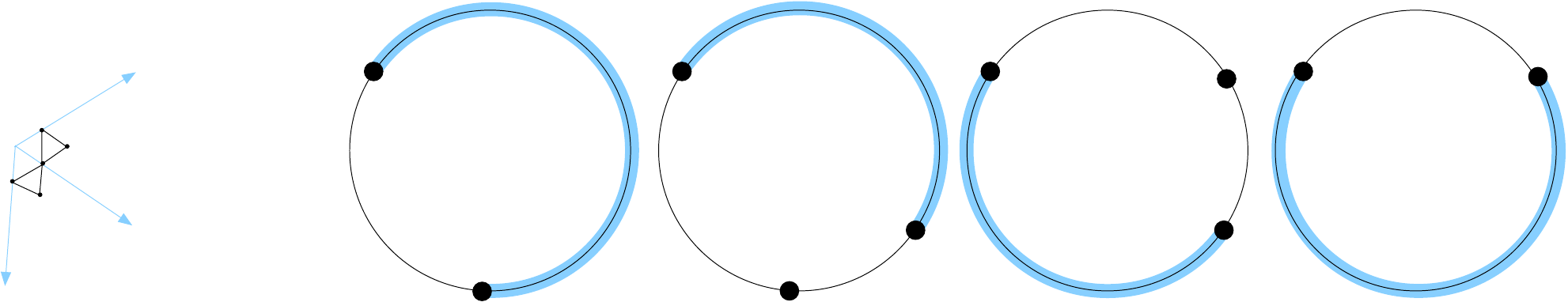}
 \put(24,16){\scriptsize $\alpha_j+2K$}
 \put(43,16){\scriptsize $\alpha_j+2K$}
 \put(62,16){\scriptsize $\alpha_j+2K$}
 \put(82,16){\scriptsize $\alpha_j+2K$}

 \put(56,6){\scriptsize $\alpha_j$}
 \put(76,6){\scriptsize $\alpha_j$}
 \put(31,2){\scriptsize $\alpha_{j-1}$}
 \put(51,2){\scriptsize $\alpha_{j-1}$}

 \put(75,16){\scriptsize $\alpha_{j+1}$}
 \put(94,16){\scriptsize $\alpha_{j+1}$}

 \put(3.5,8.2){\scriptsize $\ws_j$}
 \put(2,11.5){\scriptsize $\ws_{j+1}$}
 \put(-2,8.5){\scriptsize $\ws_{j-1}$}
 \put(5,10){\scriptsize $\vs_j$}
 \put(2,5){\scriptsize $\vs_{j-1}$}
 \put(30,-1){\scriptsize $s_{\ws_{j-1},\ws_j}$}
 \put(48,-1){\scriptsize $s_{\vs_{j-1},\ws_j}$}
 \put(70,-1){\scriptsize $s_{\vs_{j},\ws_j}$}
 \put(87,-1){\scriptsize $s_{\ws_{j+1},\ws_j}$}
\end{overpic}
\caption{Sectors $s_{\xs_i,\ys}$ when $\xs=\ys=\ws_j$.\label{fig:fig_preuve_6}}
\end{figure}

A vertical contour $\Gamma'$ passing between $\alpha_j$ and $\alpha_{j+1}$ is
contained in the three sectors 
$s_{\ws_{j-1},\ws_j}$, $s_{\vs_{j-1},\ws_j}$, $s_{\ws_{j+1},\ws_j}$. If the
fourth integral was taken along this contour, 
the combination of the four would be zero. By adding and subtracting the
integral for the pair $(\vs_j,\ws_j)$ along $\Gamma'$, we
have that the contribution of the integral part of $(\KF\KF^{-1})_{\ws_j,\ws_j}$
is equal to
\begin{equation}
  \frac{ik'}{8\pi}\KF_{\ws_j,\vs_j} \left(
  \oint_{\Gamma_{\vs_j,\ws_j}} - \oint_{\Gamma'} \right)
  \fs_{\vs_j}(u+2K) \fs_{\ws_j}(u) \ud u.
  \label{eq:integrale_cylindre}
\end{equation}

The contour $\Gamma_{\vs_j,\ws_j} -\Gamma'$ is the (negatively oriented)
boundary of a cylinder in the torus, which contains only one pole of the
integrand, at
$u=\alpha_j$. The function $\fs_{\ws_j}(u)$ has no pole in the cylinder, and only the term involving $\fs_{\ws_j}$ of the function 
$\fs_{\vs_j}(u+2K)=\KF_{\vs_j,\ws_j}\fs_{\ws_j}(u+2K)+\KF_{\ws_{j+1},\vs_j}\fs_{\ws_{j+1}}(u+2K)$
has a pole at $u=\alpha_j$. As a consequence, the contribution of the integral
part is equal to
\begin{equation*}
-\frac{ik'}{8\pi}\left(
\oint_{\Gamma_{\vs_j,\ws_j}} - \oint_{\Gamma'} \right)
\fs_{\ws_j}(u+2K) \fs_{\ws_j}(u) \ud u\\=-\frac{ik'}{8\pi}\left(
\oint_{\Gamma_{\ws_j,\ws_j}'} - \oint_{\Gamma_{\ws_j,\ws_j}} \right)
\fs_{\ws_j}(u+2K) \fs_{\ws_j}(u) \ud u=\frac{1}{2},
\end{equation*}
by continuously deforming the contours to those of Case 1, and using
Equation~\eqref{equ:case_1} and Remark~\ref{rem:Case1}.

We now handle the constant part of $(\KF\KF^{-1})_{\ws_r,\ws_j}$, keeping in mind that vertices of type `$\vs$' have no constant contribution.
As long as $r+1\notin\{j+1,j+2\}$, we have by Equation~\eqref{equ:rec_C}
\begin{equation*}
     C_{\ws_{r},\ws_j}=\KF_{\ws_{r-1},\ws_{r}}\KF_{\ws_{r},\ws_{r+1}}C_{\ws_{r+1,j}},
\end{equation*}
so that
\begin{equation*}
     \Ks_{\ws_r,\ws_{r-1}}C_{\ws_{r-1},\ws_j}+\Ks_{\ws_r,\ws_{r+1}}C_{\ws_{r+1},\ws_j}
=(\Ks_{\ws_r,\ws_{r-1}}\KF_{\ws_{r-1},\ws_{r}}\KF_{\ws_{r},\ws_{r+1}}+\Ks_{\ws_r,\ws_{r+1}})C_{\ws_{r+1},\ws_j}=0.
\end{equation*}
When $r=j+1$, recalling that we have chosen the non-standard definition from Case 1, factoring $\Ks_{\ws_{j},\ws_{j+1}}$,
using Equation~\eqref{equ:rec_C} to write $C_{\ws_{j+2},\ws_j}=C_{\ws_{j-d+2},\ws_j}$, and finally remembering 
that $C_{\ws_j,\ws_j}'=-C_{\ws_j,\ws_j}$, we have
\begin{multline*}
     \quad\Ks_{\ws_{j+1},\ws_{j}}C_{\ws_{j},\ws_j}'+\Ks_{\ws_{j+1},\ws_{j+2}}C_{\ws_{j+2},\ws_j}\\=
\Ks_{\ws_{j},\ws_{j+1}}\bigl(1+\Ks_{\ws_{j},\ws_{j+1}}\Ks_{\ws_{j+1},\ws_{j+2}}(-1)^{n(\ws_{j-d+2},\ws_j)}\bigr)C_{\ws_j,\ws_j},\quad
\end{multline*}
which is equal to $0$ by the Kasteleyn orientation condition on inner cycles of decorations. 

When $r=j$, using a similar argument, we obtain
\begin{equation*}
     \Ks_{\ws_{j},\ws_{j-1}}C_{\ws_{j-1},\ws_j}+\Ks_{\ws_{j},\ws_{j+1}}C_{\ws_{j+1},\ws_j}=-2C_{\ws_j,\ws_j}=\frac{1}{2}.
\end{equation*}
Wrapping up, we have proved that $(\KF\KF^{-1})_{\ws_r,\ws_j}$ is equal to $0$ when $r\neq j$, and to 
$\frac{1}{2}+\frac{1}{2}=1$ when $r=j$.

\paragraph{Suppose that $\ys=\vs_j(\yb)$.} Note that since $\yb$ is of type `$\vs$', we always have $C_{\xs_i,\ys}=0$. We have to handle the cases
where $\xs\in\{\ws_j(\yb),\ws_{j+1}(\yb),\vs_\ell(\yb')\}\bigcup \{\ws_\ell(\yb'),\ws_{\ell+1}(\yb'),\vs_j(\yb)\}$,
and need to check whether the sectors defining the contours $\Gamma_{\xs_i,\ys}$ in the integral part of $\KF_{\xs_i,\ys}$ have non-empty intersections.

There are three values of $\xs$ where one of the neighbors of $\xs$ is
$\ys=\vs_j(\yb)$: namely when $\xs\in\{\ws_j(\yb),\ws_{j+1}(\yb),\vs_\ell(\yb')\}$. We now omit the arguments $\yb,\yb'$ from the notation.
In these three cases, the sectors $s_{\ws_j,\vs_j}$ and $s_{\ws_{j+1},\vs_j}$
are compatible and
intersect, either in the arc from $\alpha_j$ to $\alpha_{j+1}$ (2 first cases),
or from $\alpha_{j}+2K$ to $\alpha_{j+1}+2K$ (last case). In all these
situations, using the two possible definitions of Case 2 to write
$\KF^{-1}_{\vs_j,\vs_j} = 0$ as the integral with a
contour in that common sector, then by the general argument, we get that
$(\KF\KF^{-1})_{\xs,\vs_j} = 0$.

We now need to check the remaining three cases where the combination uses
$\KF^{-1}_{\vs_\ell,\vs_j}$, corresponding to the situation where
$\xs\in\{\ws_\ell(\yb'),\ws_{\ell+1}(\yb'),\vs_j(\yb)\}$.

In the two first situations, using the general Case and Case 3, we see that the
sectors are compatible, and we can conclude with
the general argument that $(\KF \KF^{-1})_{\xs,\vs_j}=0$.

Suppose now that $\xs=\ys=\vs_j$. Its three neighbors are
$\ws_j$, $\ws_{j+1}$ and $\vs_\ell$ and the corresponding sectors are {not}
compatible, see Figure~\ref{Fig:fig_preuve_7}.

\begin{figure}[h]
  \centering
\begin{overpic}[width=\linewidth]{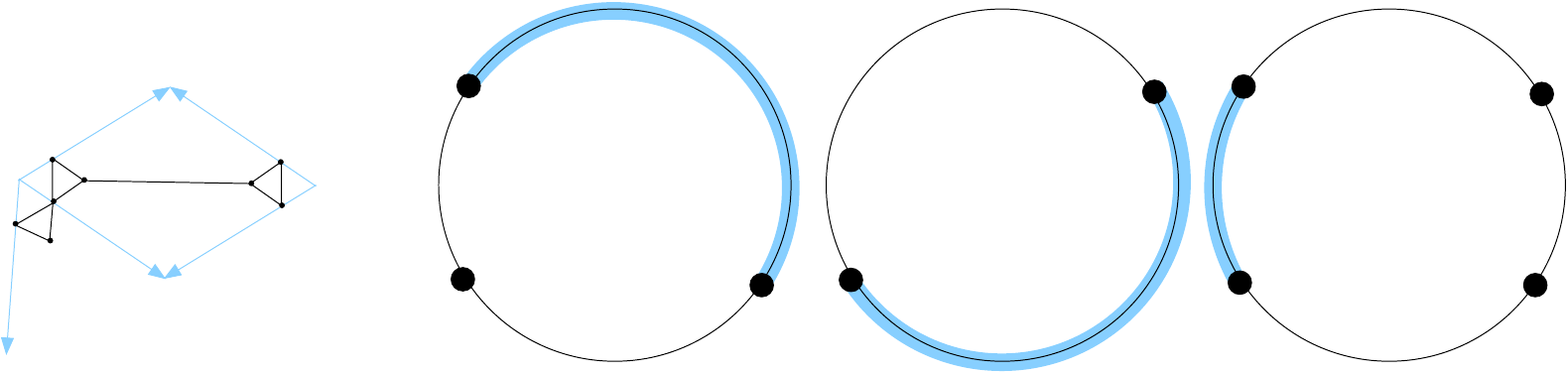}
 \put(31,18){\scriptsize $\alpha_j+2K$}
 \put(81,18){\scriptsize $\alpha_j+2K$}

 \put(46,7){\scriptsize $\alpha_j$}
 \put(95,7){\scriptsize $\alpha_j$}
 
 \put(31,7){\scriptsize $\alpha_{j+1}+2K$}
 \put(56,7){\scriptsize $\alpha_{j+1}+2K$}
 \put(81,7){\scriptsize $\alpha_{j+1}+2K$}
 \put(68,18){\scriptsize $\alpha_{j+1}$}
 \put(93,18){\scriptsize $\alpha_{j+1}$}

 \put(6,11){\scriptsize $\vs_j$}
 \put(5,9){\scriptsize $\ws_j$}
 \put(2,14){\scriptsize $\ws_{j+1}$}
 \put(2,7){\scriptsize $\vs_{j-1}$}
 \put(14,11){\scriptsize $\vs_\ell$}
 \put(38,-1){\scriptsize $s_{\ws_{j},\vs_j}$}
 \put(63,-1){\scriptsize $s_{\ws_{j+1},\vs_j}$}
 \put(88,-1){\scriptsize $s_{\vs_{\ell},\vs_j}$}
 \put(1,12){\scriptsize $\xb$}
 \put(20,12){\scriptsize $\xb'$}
\end{overpic}
\caption{Sectors $s_{\xs_i,\ys}$ when $\xs=\ys=\vs_j$.\label{Fig:fig_preuve_7}}
\end{figure}

The two sectors for $\ws_j$ and
$\ws_{j+1}$ are compatible and intersect in the arc from $\alpha_j$ to
$\alpha_{j+1}$, whereas according to the convention of Case 3, the one for
$\vs_\ell$ is the arc from $\alpha_{j}+2K$ to $\alpha_{j+1}+2K$. A vertical
contour $\Gamma'$ passing between $\alpha_j$ and 
$\alpha_{j+1}$ is contained in the three sectors $s_{\ws_j,\vs_j}$,
$s_{\ws_{j+1},\vs_j}$ and $s_{\vs_\ell,\vs_j}$. If the third integral
was taken along this contour, the combination of the three would be $0$. By
adding and subtracting the integral of the pair 
$(\vs_\ell,\vs_j)$ along $\Gamma'$, and using
Proposition~\ref{prop:function_f_kernel_Kasteleyn_Fisher} to write
\begin{equation*}
\KF_{\vs_j,\vs_\ell}\fs_{\vs_{\ell}}(u+2K)\expo_{\yb',\yb}(u)=
-\left(\KF_{\vs_j,\ws_j}\fs_{\ws_j}(u+2K)
+\KF_{\vs_j,\ws_j}\fs_{\ws_{j+1}}(u+2K)\right),
\end{equation*}
we obtain
\begin{align*}
  (\KF \KF^{-1})_{\vs_j,\vs_j} =-\frac{ik'}{8\pi}\left(
  \oint_{\Gamma_{\vs_\ell,\vs_j}}-\oint_{\Gamma'}
  \right)
  \left(
  \KF_{\vs_j,\ws_j}\fs_{\ws_j}(u+2K)+
  \KF_{\vs_j,\ws_{j+1}} \fs_{\ws_{j+1}}(u+2K)
  \right)
  \fs_{\vs_j}(u) \ud u.
\end{align*}
 By a change of variable $u\to u+2K$, the integral of the first term in the sum is
 \begin{multline*}
   \frac{ik'}{8\pi}\left(
  \oint_{\Gamma'} -\oint_{\Gamma_{\vs_\ell,\vs_j}}
  \right)
  \KF_{\vs_j,\ws_j}\fs_{\ws_j}(u+2K)\fs_{\vs_j}(u) \ud u \\=
  -\frac{ik'}{8\pi}\KF_{\vs_j,\ws_j}
  \left(
  \oint_{\Gamma'+2K} -\oint_{\Gamma'_{\vs_\ell,\vs_j}+2K}
  \right)\fs_{\vs_j}(u+2K)\fs_{\ws_j}(u)\ud u,
 \end{multline*}
 which is exactly the same integral as the one computed
 in~\eqref{eq:integrale_cylindre}. Indeed, $\Gamma'+2K$ 
 (resp.\ $\Gamma_{\vs_\ell,\vs_j}+2K$) is homologous to
 $\Gamma_{\vs_j,\ws_j}$ (resp.\ to $\Gamma'$). Therefore it is equal to
 $\frac{1}{2}$.

 Using the same argument as for the computation of
 \eqref{eq:integrale_cylindre}, we obtain that the integral of the second term
 in the sum
 \begin{multline*}
   \frac{ik'}{8\pi}\left(
  \oint_{\Gamma'} -\oint_{\Gamma_{\vs_\ell,\vs_j}}
  \right)
  \KF_{\vs_j,\ws_{j+1}} \fs_{\ws_{j+1}}(u+2K)
  \fs_{\vs_j}(u) \ud u\\= -\frac{ik'}{8\pi}
 \left(
  \oint_{\Gamma'} -\oint_{\Gamma_{\vs_\ell,\vs_j}}
  \right)
  \fs_{\ws_{j+1}}(u+2K) \fs_{\ws_{j+1}}(u) \ud u =\frac{1}{2}.
\end{multline*}
Therefore $(\KF \KF^{-1})_{\vs_j,\vs_j} = \frac{1}{2}+\frac{1}{2} = 1$, which
completes the proof.
\hfill$\square$

Note that the proof uses essentially the fact that the contour
$\Gamma_{x,y}$ winds once vertically, but makes no use of the horizontal winding
of the contour, which can be arbitrary. However, ``verticality'' of the contour
plays a crucial role for the exponential decay of the coefficients of
$\KF^{-1}$, as stated below in Theorem~\ref{thm:asymptotics_inverse_Kasteleyn}.

\subsection{Asymptotics for the inverse Kasteleyn operator $\KF$}\label{sec:asymptKF}
For any $\xb,\yb\in \Gs$, define
\begin{equation}
\label{eq:def_chi_KF}
     \chi(u) = \frac{1}{\vert \xb-\yb\vert}\log\{ \expo_{(\xb,\yb)}(u+2iK')\},
\end{equation}
with the exponential function introduced in Section~\ref{sec:massive_Lap}. The
main result of this section (Theorem~\ref{thm:asymptotics_inverse_Kasteleyn})
shows the exponential decay of the inverse Kasteleyn operator, with a rate that
can be directly computed in terms of $\chi$.

Since $\vert \xb-\yb\vert$ will be typically large in this section concerned
with asymptotic results, we are in the general case, according to
Section~\ref{subsubsec:sectors}. The poles of the exponential function are
$\{\tau_j\}$;
they belong to a sector of size strictly less than $2K$, say
$\tau_j\in\tau+(-K,K)$.

\begin{thm}
\label{thm:asymptotics_inverse_Kasteleyn}
Assume that $k\neq 0$.
As $\vert \xb-\yb\vert\to\infty$, one has
\begin{equation*}
     \Ks^{-1}_{\xs,\ys}=\frac{-\fs_\xs(u_0+2K+2iK')\fs_{\ys}(u_0+2iK')}{4\sqrt{2\pi \vert \xb-\yb\vert \chi''(u_0)}} e^{\vert \xb-\yb\vert \chi(u_0)}\cdot (1+o(1)),
\end{equation*}
where $u_0$ is the unique $u\in\tau+2K+(-K,K)$ such that $\chi'(u)=0$, and $\chi(u_0)< 0$. 
\end{thm}

\begin{proof}
It consists in applying the saddle-point method to the 
contour integral \eqref{equ:KF_inverse}. It is very similar to the proof of
Theorem~14 in \cite{BdTR1}, which is devoted to the derivation of the
asymptotics of the Green function \eqref{eq:def_Green_function}. Indeed, the
integrands of \eqref{eq:def_Green_function} and \eqref{equ:KF_inverse} only
differ by the prefactor function $\fs_\xs(\cdot+2K)\fs_{\ys}(\cdot)$ (as well as
a constant multiplicative term). This prefactor function will affect the
asymptotics by multiplying by its value at the saddle-point $u_0\pm2iK'$ the
asymptotics of \eqref{eq:def_Green_function}. Let us give some brief details. 
\begin{itemize}
  \item It follows from \cite[Lemma~15]{BdTR1} that the equation $\chi'(u)=0$
    has a unique solution $u_0$ in  the interval $\tau+2K+(-K,K)$, and moreover
    $\chi(u_0)< 0$ (cf.\ \cite[Lemma~16]{BdTR1}). The point $u_0\pm 2iK'$ will
    be interpreted as the saddle-point.
  \item We then move the contour $\Gamma_{\xs,\ys}$ of \eqref{equ:KF_inverse}
    into a new one, denoted by $\Gamma'_{\xs,\ys}$, going through $u_0\pm 2iK'$
    and satisfying some further properties. The validity of this change of
    contour is based on the fact that neither the exponential function nor the
    prefactor have poles in the sector $s_{\xs,\ys}$, see
    Figure~\ref{Fig:repr_sectors}.
  \item We adjust the new contour $\Gamma'_{\xs,\ys}$ so as to have,
    classically, a contribution exponentially negligible outside a neighborhood
    of $u_0\pm 2iK'$ (this can be done by introducing suitable steepest descent
    paths).
  \item In the neighborhood of $u_0\pm 2iK'$, we apply (a uniform version of)
    the saddle-point method, which eventually yields to the expansion written in
    Theorem~\ref{thm:asymptotics_inverse_Kasteleyn}.\qedhere
\end{itemize}
\end{proof}

\begin{rem}
\label{rem:sign_asymptotics_inverse_Kasteleyn}
Let us note that the constant $-\fs_\xs(u_0+2K+2iK')\fs_{\ys}(u_0+2iK')$ in
Theorem~\ref{thm:asymptotics_inverse_Kasteleyn} is positive. Indeed, by
\eqref{equ:rewriting_fs}, $\fs(u)$ is the sum of one or two terms
$\nc(\frac{u-{\alpha_j}}{2})$. Due to the location of the poles described in
Section~\ref{subsubsec:encoding_poles}, $\alpha_{j}\in\tau+(-K,K)$. Hence 
\begin{equation*}
     \frac{u_0\pm2iK'-\alpha_{j}}{2}\in K+iK'+(-K,K),
\end{equation*}
and by Table~\ref{table:identities_Jacobi_function}, $\nc(\frac{u_0\pm2iK'-\alpha_{j}}{2})=\nc(K+iK'+v_0)=i{k'}^{-1}k\cn(v_0)$, with some $v_0\in(-K,K)$. Consequently, the constant $-\fs_\xs(u_0+2K+2iK')\fs_{\ys}(u_0+2iK')$ equals minus the product of two (sums of) terms $i{k'}^{-1}k\cn(v_0)$. The positivity follows from $-i^2=1$.
\end{rem}

\subsection{The case where $\Gs$ is $\ZZ^2$-periodic}\label{sec:KFperio}

In this section, we suppose moreover that the graph $\Gs$ is $\ZZ^2$-periodic,
implying that the graph $\GF$ is also $\ZZ^2$-periodic. 
We consider the dimer model on the graph $\GF$ arising from a $Z$-invariant
Ising model on $\Gs$, and the $Z$-invariant rooted spanning forest model 
on $\Gs$. Note that the weight function corresponding to each of the models is
periodic. Consider the natural operators associated to the
two models, that is the Kasteleyn operator $\KF$ acting on $\CC^{\VF}$, arising
from a periodic admissible orientation of the edges of $\GF$\footnote{%
Such an orientation always exists by Theorem 3.1 of~\cite{CimaReshe1}, using 
that the number of vertices of the fundamental domain of $\GF$ is even.}; 
and the massive Laplacian operator $\Delta^m$ acting on $\CC^{\Vs}$.

Using Fourier techniques, see for example~\cite{CKP}, an important tool for understanding
each of the models is the 
characteristic polynomial of the respective operators. In this section, we prove
that the characteristic polynomials of the two models are equal 
up to an explicit constant. We state implications of this fact for the spectral
curve of the dimer model on $\GF$.

In the periodic case, we have two explicit expressions for the inverse operator
$\KF^{-1}$. The one given by Theorem~\ref{thm:KFmoins_un}
and the one obtained using Fourier techniques. In Corollary \ref{cor:uniqueness_Km1}, we prove that the two are equal.

These two facts are used in Section~\ref{subsec:dimer_model_GF} for obtaining
explicit expressions for the dimer model on $\GF$ and relating its free energy
to that of the rooted spanning
forest model.
\subsubsection{Quasi-periodic functions}
A natural toroidal exhaustion of $\Gs$ is $\{\Gs_n=\Gs/n\ZZ^2\}_{n\geq 1}$; in a
similar way $\{\GF_n=\GF/n\ZZ^2\}_{n\geq 1}$
is a toroidal exhaustion of $\GF$. The smallest graphs $\Gs_1$ and $\GF_1$ of
the exhaustions are known as the \emph{fundamental domains}.

We note with an addition sign the action of $\ZZ^2$ on vertices, edges, faces of
$\Gs$ and $\GF$.
Let $\gamma_x$, $\gamma_y$ be two simple paths in $\Gs^*$ connecting a given
face $f_0$ , with $f_0+(1,0)$ and $f_0+(0,1)$ respectively.

To simplify notation, we also write
$\gamma_x,\gamma_y$ for the images of these paths when quotienting by the action
of $\ZZ^2$, which are now non trivial cycles of $\Gs_1^*$ generating the first
homology group of the torus on which $\Gs_1$ is drawn.
For $(z,w)\in\CC^2$, 
denote by $\CC^{\Vs}_{(z,w)}$
the space of \emph{$(z,w)$-quasi-periodic functions} on vertices of $\Gs$:
\begin{equation*}
     \forall\,\xb\in\Vs,\,\forall\,(z,w)\in\CC^2,\quad f(\xb+(m,n))=z^{-m}w^{-n}f(\xb).
\end{equation*}
For every vertex $\xb$ of $\Gs_1$, define $\delta_{\xb}(z,w)$ to be the
$(z,w)$-quasi-periodic function equal to $0$ on vertices
which are not translates of $\xb$ and to $1$ at $\xb$. Then the collection $\{\delta_{\xb}(z,w)\}_{\xb\in\Vs_1}$ is a natural basis for 
$\CC^{\Vs}_{(z,w)}$.

Note that $\gamma_x,\gamma_y$ can be deformed into directed cycles of the dual graph
$(\GF_1)^*$ (or of the diamond graph $\GR_1$).
In a way similar to $\CC^{\Vs}_{(z,w)}$ we define 
$\CC^{\VF}_{(z,w)}$, the space of $(z,w)$-quasi-periodic functions on vertices of $\GF$, with basis $\{\delta_{\xs}(z,w)\}_{\xs\in\VF_1}$.

\subsubsection{Characteristic polynomials and spectral curves}
Let $\KF(z,w)$ be the matrix of the restriction of $\KF$ to the space 
$\CC^{\VF}_{(z,w)}$ in the basis  $\{\delta_{\xs}(z,w)\}_{\xs\in\VF_1}$.
The matrix $\KF(z,w)$ is obtained from the 
Kasteleyn matrix of the fundamental domain $\GF_1$ as follows:
multiply coefficients of 
edges crossing $\gamma_x$ by $z$ (resp.\ $z^{-1}$) if the edge goes from the
left of $\gamma_x$ to the right (resp.\ from the right of $\gamma_x$ to the
left); coefficients of edges crossing
$\gamma_y$ are multiplied by $w$ or $w^{-1}$. 

In a similar way, $\Delta^{m}(z,w)$ is the matrix of the restriction of
$\Delta^m$ to the 
space $\CC^{\Vs}_{(z,w)}$ in the basis  $\{\delta_{\xb}(z,w)\}_{\xb\in\Gs_1}$.

The \emph{dimer characteristic polynomial} of $\GF$ is the determinant of the
matrix $\KF(z,w)$,
\begin{equation*}
     \Pdimer(z,w)=\det \KF(z,w).
\end{equation*}
The \emph{massive Laplacian characteristic polynomial} of $\Gs$ is the
determinant of $\Delta^m(z,w)$,
\begin{equation*}
     \Plapm(z,w)=\det \Delta^{m} (z,w).
\end{equation*}

Consider a Laurent polynomial $P(z,w)$ in two complex variables $z,w$. Then, the \emph{spectral curve} of the polynomial, denoted by $C(P)$, 
is defined to be its zero locus:
\begin{equation*}
C(P)=\{(z,w)\in(\CC^*)^2:P(z,w)=0\}.
\end{equation*}

The following proves that the two characteristic polynomials are equal up to an
explicit multiplicative constant. The constant is determined in
Section~\ref{sec:prop_char_pol}.

\begin{prop}\label{prop:equality_charact_pol}
There exists a nonzero constant $c$ such that
\begin{equation*}
     \Pdimer(z,w)=c\Plapm(z,w).
\end{equation*}
\end{prop}
\begin{proof}
Let us first prove that $\Plapm(z,w)$ divides $\Pdimer(z,w)$, using an argument similar to that of \cite[Lemma 11]{BoutillierdeTiliere:iso_perio}.
By \cite[Proposition 21]{BdTR1}, any point of the spectral curve of the
Laplacian is of the form $(z(u),w(u))$, with  $u\in\TT(k)$,
and 
\begin{equation}
  z(u) = 
  \prod_{e^{i\overline\alpha} \in \gamma_x}
  \left(i\sqrt{k'}\sc(\textstyle \frac{u-\alpha}{2})\displaystyle\right),\quad 
  w(u) =
  \prod_{e^{i\overline\alpha} \in \gamma_y}
  \left(i\sqrt{k'}\sc(\textstyle \frac{u-\alpha}{2})\displaystyle\right).
  \label{equ:param_spec_curve}
\end{equation}

The function $\gs_{(\cdot,\ys)}(u)$ of Equation~\eqref{equ:function_g}
is $(z(u),w(u))$-quasi-periodic, as it involves
the function $\fs_{\cdot}(u)$, which
is invariant by translations, and the discrete massive exponential function,
which is $(z(u),w(u))$ quasi-periodic.
By Proposition~\ref{prop:function_f_kernel_Kasteleyn_Fisher} it is in the kernel
of the Kasteleyn operator $\KF$. 
Therefore,
$\Pdimer(z(u),w(u))=0$.
As a consequence,
$\Plapm(z,w)$ divides $\Pdimer(z,w)$.

Now by~\cite{Dubedat} we know that, up to a constant, $\Pdimer(z,w)$ is equal
to the dimer characteristic polynomial on $\GQ$. 
The graph $\GQ$ being bipartite, the corresponding
spectral curve is a Harnack curve \cite{KOS}. Hence, the characteristic
polynomial on $\GQ$, and thus $\Pdimer(z,w)$, are irreducible.

The fact that $\Plapm(z,w)$ divides $\Pdimer(z,w)$ and that $\Pdimer(z,w)$ is
irreducible implies that the two polynomials are
equal up to a constant, and concludes the proof.
\end{proof}

By Proposition~\ref{prop:equality_charact_pol}, properties of the spectral
curve of $\Plapm$ obtained in~\cite{BdTR1},
are automatically transferred to the spectral curve of the dimer characteristic
polynomial $\Pdimer$.

\begin{cor}\label{cor:charact_pol}
Let $k^2\neq 0$.
\begin{itemize}
\item The spectral curve of the dimer model on $\GF$ is a Harnack curve of genus
  $1$, with $(z,w)\leftrightarrow(z^{-1},w^{-1})$ symmetry.
\item Every Harnack curve of genus $1$ with 
$(z,w)\leftrightarrow(z^{-1},w^{-1})$ symmetry arises from such a dimer model. 
\item The characteristic polynomial $\Pdimer(z,w)$ has no zero on the unit torus $\{(z,w)\in\CC^2:\vert z\vert =1,\vert w\vert =1\}$.
\end{itemize}
\end{cor}
\begin{proof}
For $k^2>0$, this follows from our results of the paper~\cite{BdTR1}. For the last point we
use the fact that $(0,0)$ does not belong to the amoeba of the spectral curve.
For $k^2 <0$, we prove later in Section~\ref{sec:duality_and_phase_transition} that the spectral
curve is the same as for an elliptic parameter $(k^*)^2>0$, such that
$(k^*)' k' =1$. 
\end{proof}

\begin{rem}
  For $k=0$, the spectral curve of the dimer model on $\GF$ is still the
  spectral curve of the Laplacian with conductances $\tan\theta$. It is a
  Harnack curve of genus $0$ with the same symmetry, and a double point at
  $z=w=1$, see~\cite{KO2,BoutillierdeTiliere:iso_perio,BdTR1}.
\end{rem}

\subsubsection{Inverse Kasteleyn operator}

Because of the periodicity of the graph, it is possible to define an inverse for
the Kasteleyn matrix, by inverting in Fourier space the multiplication operator
$K(z,w)$:
define the operator 
$\widetilde{\KF}^{-1}$ by its coefficients
\begin{equation}\label{equ:Kinv2}
\widetilde{\KF}^{-1}_{\xs+(m,n),\xs'+(m',n')}=\frac{1}{(2\pi i)^2}
\iint\limits_{\{(z,w)\in\CC^2:\vert z\vert =1,\vert w\vert =1\}}\frac{[\Cof K(z,w)]_{\xs',\xs}}{\Pdimer(z,w)} z^{m'-m}w^{n'-n}\frac{\ud z}{z}\frac{\ud w}{w},
\end{equation}
for all $\xs,\xs'\in\GF_1$, and $m,m',n,n'\in\ZZ$. 

By Corollary~\ref{cor:charact_pol}, we know that the characteristic polynomial
$\Pdimer(z,w)$ has no zero on the unit torus.
This means that Proposition~5 of~\cite{BoutillierdeTiliere:iso_perio} holds, and
we have:

\begin{prop}\label{cor:uniqueness_Km1}
The inverse Kasteleyn operators $\widetilde{\KF}^{-1}$ and $\KF^{-1}$ given by
Equations~\eqref{equ:KF_inverse} and~\eqref{equ:Kinv2},
respectively, are equal.
\end{prop}

\begin{proof}
For $k\neq 0$, there is no root of $\Pdimer(z,w)$ on the unit torus, by
Corollary~\ref{cor:charact_pol}. As a consequence, the quantities
$\widetilde{\KF}^{-1}_{\xs+(m,n),\xs'}$ are the Fourier coefficients of an
analytic periodic function, and as such, decay exponentially fast when
$\vert m\vert +\vert n\vert $ goes to $\infty$. In particular, these coefficients are bounded.
By the uniqueness statement in Theorem~\ref{thm:KFmoins_un},
$\widetilde{\KF}^{-1}$ and $\KF^{-1}$ are equal.

For $k=0$,
adaptating the computation of the
asymptotics of the integral formula for $\KF^{-1}$
from~\cite[Corollary~7]{BoutillierdeTiliere:iso_gen} proves that these
coefficients go to $0$. 
The result then follows
from~\cite[Proposition~5]{BoutillierdeTiliere:iso_perio}, stating uniqueness of
the inverse of the Kasteleyn operator with coefficients tending to $0$ at
infinity on a periodic Fisher graph. 
\end{proof}

Note that similarly to what has been done for the Green function of the $Z$-invariant
Laplacian~\cite[Section~5.5.1]{BdTR1}, it is also possible to directly understand the
transformation from the double integral expression~\eqref{equ:Kinv2} to the
contour integral~\eqref{equ:KF_inverse}, by computing one integral
(\textit{e.g.}, with respect to $w$) by residues, and then by perfoming the change of
variable from $z$ on the spectral curve to $u\in\TT(k)$.

\subsection{Dimer model on the graph $\GF$}

\label{subsec:dimer_model_GF}

Consequences of the results of Sections~\ref{sec:localKFinv}--\ref{sec:KFperio}
on the dimer model on $\GF$ are investigated:
in Section~\ref{subsubsec:Gibbs_measure}, we prove an explicit local expression
for a Gibbs measure, in the case where the graph $\Gs$ is periodic or not. 
In
Section~\ref{subsubsec:free_energy}, we prove an explicit local expression for
the free energy of the dimer model.
Combining this with the high temperature expansion, we deduce an explicit local
formula for the free energy of the Ising model,
as a sum of contributions for each edge of the fundamental domain,
similar to the one given by
Baxter, see~\cite[(7.12.7)]{Baxter:exactly} and \cite{Baxter:8V}.

\subsubsection{Gibbs measure}

\label{subsubsec:Gibbs_measure}

When the graph is infinite, the notion of Boltzmann measure is replaced by that of Gibbs measure.
A \emph{Gibbs measure} on the set of dimer configurations of $\GF$ is a probability measure satisfying the DLR conditions: if one fixes
a perfect matching in an annular region of $\GF$, then perfect matchings inside and outside of this annulus are independent. Moreover,
the probability of an interior perfect matching is proportional to the product of its edge-weights. 
Let $\F$ be the $\sigma$-field generated by cylinders of $\GF$, a \emph{cylinder} being the set of dimer configurations containing a fixed, finite subset
of edges of $\GF$. Following the argument of~\cite{CKP,deTiliere:quadri}, we obtain the following.

\begin{thm}
\label{thm:Gibbs_measure}
There is a unique probability measure $\PPdimer$ on $(\M(\GF),\F)$, such that the probability
of occurrence of a subset of edges $\mathscr{E} = \{\es_1=\xs_1\ys_1,\dots,\es_n=\xs_n\ys_n\}$ of $\GF$ in a
dimer configuration of $\GF$ is:
\begin{equation}
\PPdimer(\es_1,\dots,\es_n)=
\Biggl(\prod_{j=1}^n \KF_{\xs_j,\ys_j}\Biggr)\Pf[^t(\KF^{-1})_\E],
\end{equation}
where $\KF^{-1}$ is the inverse Kasteleyn operator whose coefficients are given by~\eqref{equ:KF_inverseH},
and $(\KF^{-1})_\E$ is the sub-matrix of $\KF^{-1}$ whose rows and columns are
indexed by vertices of $\E$. The measure $\PPdimer$ is a Gibbs measure.

Moreover, when the graph $\GF$ is $\ZZ^2$-periodic, the measure $\PPdimer$ is
obtained as weak limit of the dimer Boltzmann measures 
on the toroidal exhaustion $\GF_n$, and coefficients of $\KF^{-1}$ are also
given by $\widetilde{\KF}^{-1}$ of Equation~\eqref{equ:Kinv2}. 
\end{thm}
\begin{proof}
Convergence of the Boltzmann measures in the periodic case follows the argument
of~\cite{CKP}. The delicate issue in the convergence comes from 
the possible zeros of the dimer characteristic polynomial on the torus
$\{(z,w):\vert z\vert =1,\vert w\vert =1\}$. By Corollary~\ref{cor:charact_pol}, we know that 
whenever $k\neq 0$, it has no zero, and the argument goes trough. 
When
$k=0$, it has a double zero at $(1,1)$, the argument is more delicate and has
been done 
in~\cite{BoutillierdeTiliere:iso_perio}.

The argument in the non-periodic case follows that of~\cite{deTiliere:quadri}. The key
requirements are the convergence of the Boltzmann measure in the 
periodic case, uniqueness of the inverse Kasteleyn operator decreasing at
infinity, and the locality property of the formula given by
Theorem~\ref{thm:KFmoins_un}.
\end{proof}

\paragraph{Example.}
Consider an edge $\es=\vs_j(\xb)\vs_\ell(\yb)=\vs_j\vs_\ell$ of $\GF$ corresponding to an edge $e=\xb\yb$ of $\Gs$ with rhombus half-angle $\theta_e$. Then,
the probability $\PPdimer(\es)$ that this edge occurs in a dimer configuration of $\GF$, or equivalently the probability that this edge occurs in a 
high temperature polygon configuration of $\Gs$, is equal to 
\begin{equation}\label{ex:proba_comput}
\PPdimer(\es)=\KF_{\vs_j\vs_\ell}\KF^{-1}_{\vs_\ell,\vs_j}=\frac{1}{2}-\frac{1-2\Hh(2\theta_e)}{2\cn\theta_e}.
\end{equation}
This is computed in Lemma~\ref{lem:inverse_Kasteleyn_edge} of Appendix~\ref{app:explicit_integrals}.

To compute probabilities of edges occurring in low temperature polygon
configurations one uses the duality relation of Section~\ref{sec:duality}.

\subsubsection{Free energy}

\label{subsubsec:free_energy}

Suppose that the isoradial graph $\Gs$ is $\ZZ^2$-periodic.
The \emph{free energy} $F$ of a model is defined to be minus the exponential growth
rate of its partition function, that is
\begin{align*}
\Fdimer^k&=-\lim_{n\rightarrow\infty}\frac{1}{n^2}\log \Zdimer(\GF_n,\nu),\\
\Fising^k&=-\lim_{n\rightarrow\infty}\frac{1}{n^2}\log \Zising(\Gs_n,\Js),\\
\Fforest^k&=-\lim_{n\rightarrow\infty}\frac{1}{n^2}\log \Zforest(\Gs_n,\rho,m),
\end{align*}
where $\nu$, $\Js$, $m^2$ and $\rho$ are given by \eqref{eq:dimer_weights_GQ},
\eqref{eq:def_weight}, \eqref{eq:mass} and 
\eqref{eq:conductances}, respectively. 

In Theorem~\ref{thm:free_energy_dimer} we prove an explicit \emph{local} formula
for the \emph{free energy $\Fdimer^k$ of the dimer model} on the graph $\GF$
arising from a $Z$-invariant Ising model on $\Gs$. From this we deduce
an explicit formula for the \emph{free energy $\Fising^k$
of the $Z$-invariant Ising model}, see Corollary~\ref{cor:free_energy_Ising}.
Then in Corollary~\ref{cor:link_free_energies}, we compare the latter to the 
\emph{free energy $\Fforest^k$ of $Z$-invariant spanning forests on $\Gs$}.

\begin{thm}\label{thm:free_energy_dimer}
The free energy of the dimer model on the Fisher graph $\GF$ arising from the $Z$-invariant Ising model on $\Gs$ is equal to,
\begin{multline*}
\Fdimer^k= 
-(\vert \Es_1\vert +\vert\Vs_{1}\vert)\frac{\log 2}{2} \\
+\sum_{e\in\Es_1}
\left(
-\frac{1}{2}\log\left(\sc\frac{\theta_e}{2}\dn\frac{\theta_e}{2}\right)+
\left(\frac{1-2\Hh(2\theta_e)}{2}\right)\log\sc\theta_e
+\int_{\theta_e^{\mathrm{flat}}}^{\theta_e} 2{\Hh}'(2\theta)\log\sc \theta \,\ud\theta
\right).
\end{multline*}
\end{thm}

\begin{proof}
By Corollary~\ref{cor:charact_pol} the dimer characteristic polynomial
$\Pdimer(z,w)$ has no zero on the unit torus. This implies \cite{CKP} that
the free energy $\Fdimer$ is equal to
\begin{equation}
\Fdimer=-\frac{1}{2}\iint_{\vert z\vert =\vert w\vert =1} \log \Pdimer(z,w)\frac{\ud z}{2i\pi z} \frac{\ud w}{2i\pi w}.
\label{equ:free_energy_fourier}
\end{equation}

Following an idea of Kenyon~\cite{Kenyon3}, the next step consists in studying its variation as the embedding
of the graph is modified by tilting the train-tracks. Note that the idea of tilting the train-tracks can already be found in 
the section \emph{free energy} of~\cite{Baxter:8V}. 

Let us consider a smooth deformation of the isoradial graph $\Gs$, \emph{i.e.}, a
continuous family of isoradial graphs $(\Gs(t))_{t\in[0,1]}$ obtained by varying
the directions $(\alpha_T(t))$ of the train-tracks smoothly with $t$, in such a way that
$\Gs(1)=\Gs$ and $\Gs(0)=\Gs_{\text{flat}}$, where $\Gs_{\text{flat}}$ is an isoradial graph whose edges
have half-angles $\overline{\theta}^{\mathrm{flat}}$ equal to $0$ or $\frac{\pi}{2}$. Let $\GF(t)$ be the Fisher graph corresponding to the isoradial graph $\Gs(t)$, and let 
$\Fdimer(t)$ be the corresponding free energy. We thus have $\Fdimer=\Fdimer(1)$.

As the angles of the train-tracks vary smoothly with $t$, recalling that $\Pdimer(z,w)=\det \KF(z,w)$, one has:
\begin{align*}
\frac{\ud \Fdimer(t)}{\ud t}&=-\frac{1}{2}
 \iint_{\vert z\vert =\vert w\vert =1}  \frac{\ud}{\ud t} \log\det \KF(z,w) \frac{\ud z}{2\pi i z}\frac{\ud w}{2\pi i w}\\
   &=-\frac{1}{2}\iint_{\vert z\vert =\vert w\vert =1} \sum_{\xs,\ys\in\VF_1} 
      \frac{\partial \log\det \KF(z,w)}
           {\partial \KF(z,w)_{\xs,\ys}} 
      \frac{\ud \KF(z,w)_{\xs,\ys}}{\ud t}
   \frac{\ud z}{2\pi i z} \frac{\ud w}{2\pi i w} \\ 
   & = -\frac{1}{2}\sum_{\xs,\ys\in\VF_1} \iint_{\vert z\vert =\vert w\vert =1} (\KF(z,w)^{-1})_{\ys,\xs}
     \frac{\ud \KF(z,w)_{\xs,\ys}}{\ud t} 
   \frac{\ud z}{2\pi i z} \frac{\ud w}{2\pi i w}\\
   &=-\sum_{\es=\xs\ys\in\EF_1} \widetilde{\KF}^{-1}_{\ys,\xs}  \frac{\ud \KF_{\xs,\ys}}{\ud t}=
    -\sum_{\es=\xs\ys\in\EF_1} \PPdimer(\es)\frac{\ud \log \nu_\es}{\ud t}.
\end{align*}
In the penultimate line we used that for an invertible matrix $M=(M_{i,j})$, one has
\[\frac{\partial \log\det M}{\partial M_{i,j}} = (M^{-1})_{j,i}.\] In the last line we
used: the explicit expression for $\widetilde{\KF}^{-1}$ given by Equation~\eqref{equ:Kinv2}, Corollary~\ref{cor:uniqueness_Km1} 
and Theorem~\ref{thm:Gibbs_measure} implying that $\KF_{\xs,\ys}\widetilde{\KF}^{-1}_{\ys,\xs}=\KF_{\xs,\ys}\KF^{-1}_{\ys,\xs}=\PPdimer(\es)$, and
the fact that $\nu_\es=\vert\KF_{\xs,\ys}\vert$.

The weight of an edge $\es$ of a decoration of $\GF_1$ is equal to $1$, \emph{i.e.}, is 
independent of the rhombus angle, so that it does not contribute to $\frac{\ud \Fdimer(t)}{\ud t}$.

If $\es$ corresponds to an edge of $\Gs_1$ having rhombus angle $\overline{\theta}_e$, we have by 
Equations~\eqref{ex:proba_comput} and~\eqref{eq:dimer_weights_GF}:
\begin{equation}\label{eq:heart}
\PPdimer(\es)=\frac{1}{2}-\frac{1-2\Hh(2\theta_e)}{2\cn\theta_e}:=\PP(\theta_e),\quad 
\nu_\es=\frac{\sn\theta_e}{1+\cn\theta_e}:=\nu(\theta_e).
\end{equation}
Integrating the identity
\begin{align*}
\frac{\ud \Fdimer(t)}{\ud t}=-\sum_{e\in\Es_1} \PP(\theta_e) \frac{\ud \log\nu(\theta_e)}{\ud \theta_e}\frac{\ud \theta_e}{\ud t}
\end{align*}
along the deformation, we have
\begin{align}\label{equ:Fdimer}
\Fdimer(1)&=\Fdimer(0)+\int_0^1 \sum_{e\in\Es_1}\frac{\ud \Fdimer(t)}{\ud t}\ud t 
=\Fdimer(0)-
\sum_{e\in\Es_1} \int_{\theta_e^{\mathrm{flat}}}^{\theta_e}\PP(\theta)\frac{\ud \log \nu(\theta)}{\ud \theta} \ud \theta. 
\end{align}

We first consider the integral part of the above equation, and then the term $\Fdimer(0)$. 
Replacing $\PP(\theta)$ using~\eqref{eq:heart}, we have:
\begin{align*}
-\int \PP(\theta)\frac{\ud \log \nu(\theta)}{\ud \theta} \ud\theta &=
 -\int \left(\frac{1}{2}-\frac{1-2\Hh(2\theta)}{2\cn\theta }\right)\frac{\ud \log
 \nu(\theta)}{\ud \theta}\ud\theta\\
&=-\frac{1}{2}\log\nu(\theta)
+\int \left(\frac{1-2\Hh(2\theta)}{2\cn\theta }\right)\frac{\ud \log
\nu(\theta)}{\ud \theta}\ud\theta.
\end{align*}
Using the explicit expression of $\nu(\theta)$ (see again~\eqref{eq:heart}) and the formulas $\sn'=\cn\dn,\,\cn'=-\sn\dn$, $\cn^2+\sn^2=1$,
gives
\begin{equation*}
\frac{\ud \nu(\theta)}{\ud \theta}=\frac{(\cn\theta +\cn^2\theta+\sn^2\theta)\dn\theta}{(1+\cn\theta )^2}
=\frac{\dn\theta }{1+\cn\theta },
\end{equation*}
which readily implies that
\begin{equation*}
     \frac{\ud \log \nu(\theta)}{\ud \theta}=\frac{\dn\theta }{\sn\theta}=\ds\theta.
\end{equation*}
As a consequence,
\begin{align*}
\left(\frac{1-2\Hh(2\theta)}{2\cn\theta }\right)\frac{\ud \log\nu(\theta)}{\ud\theta}
&=\left(\frac{1-2\Hh(2\theta)}{2\cn\theta }\right)\ds\theta=
\frac{1-2\Hh(2\theta)}{2}\frac{\ud \log \sc\theta }{\ud \theta},
\end{align*}
using that $\sc'=\frac{\dn}{\cn^2}\,\Rightarrow\, (\log\sc)'=\frac{\dn}{\sn\cn}$.
Integrating by parts, we obtain
\begin{align*}
-\int \PP(\theta)\frac{\ud \log \nu(\theta)}{\ud \theta} \ud \theta &=
-\frac{1}{2}\log\frac{\sn\theta }{1+\cn\theta }+
\left(\frac{1-2\Hh(2\theta)}{2}\right)\log\sc\theta 
+\int 2{\Hh}'(2\theta)\log\sc\theta\,\ud\theta.
\end{align*}
We now need to compute $-\frac{1}{2}\log\frac{\sn\theta }{1+\cn\theta }+(\frac{1-2\Hh(2\theta)}{2})\log\sc\theta $ in the limit 
$\theta\rightarrow\theta^{\mathrm{flat}}$ for possible values of $\theta^{\mathrm{flat}}$, \emph{i.e.}, for $\theta^{\mathrm{flat}}\in\{0,K\}$
since $\overline{\theta}^{\mathrm{flat}}\in\{0,\frac{\pi}{2}\}$.

\underline{When $\theta\rightarrow K$.} 
We have $\sn K =1$, $\cn K =0$ and $\Hh(2K)=\frac{1}{2}$, see \eqref{eq:expressions_Hh_Hv_Z} for $k^2>0$ and the duality relation \eqref{eq:definition_Hh_Hv_k2<0} for $k^2<0$; so that
\[
-\frac{1}{2}\log\frac{\sn K}{1+\cn K}+
\left(\frac{1-2\Hh(2K)}{2}\right)\log\sn K=0.
\] 
Moreover as $\theta\rightarrow K$,
$\frac{1-2\Hh(2K)}{2}=O(\theta-K)$ implying that $\lim_{\theta\rightarrow K}\frac{1-2\Hh(2K)}{2}\log\cn\theta =0$. 
Wrapping up, we have
\[
\lim_{\theta\rightarrow K}-\frac{1}{2}\log\frac{\sn\theta }{1+\cn\theta }+\left(\frac{1-2\Hh(2\theta)}{2}\right)\log\sc\theta =
0.
\]

\underline{When $\theta\rightarrow 0$.}
We have $\sn 0 =0$, $\cn 0 =1$, $\Hh(0)=\frac{K'}{\pi}Z(0)=0$; implying that
\[
-\frac{1}{2}\log\frac{1}{1+\cn 0 }-\left(\frac{1-2\Hh(0)}{2}\right)\log\cn 0 =\frac{1}{2}\log 2.
\]
We are left with handling,
$
(-\frac{1}{2}+\frac{1}{2}-\Hh(2\theta))\log\sn\theta=-
\Hh(2\theta)\log\sn\theta. 
$
As $\theta\rightarrow 0$, $\Hh(2\theta)=O(\theta)$, thus $\lim_{\theta\rightarrow 0}\Hh(2\theta)\log\sn\theta=0$.
Wrapping up, we have
\[
\lim_{\theta\rightarrow 0}-\frac{1}{2}\log\frac{\sn\theta }{1+\cn\theta }+\left(\frac{1-2\Hh(2\theta)}{2}\right)\log\sc\theta =
\frac{1}{2}\log 2.
\]

Plugging this into Equation~\eqref{equ:Fdimer}, we obtain
\begin{align*}
\Fdimer(1)=&\Fdimer(0) -\frac{\log 2}{2}\vert\{ e\in\Es_1:\overline{\theta}_e^{\mathrm{flat}}=0\}\vert+\\
&+\sum_{e\in\Es_1}
\left(
-\frac{1}{2}\log\frac{\sn\theta }{1+\cn\theta }+
\Bigl(\frac{1-2\Hh(2\theta)}{2}\Bigr)\log\sc\theta 
+\int_{\theta_e^{\mathrm{flat}}}^{\theta_e} 2{\Hh}'(2\theta)\log\sc\theta\,\ud\theta
\right).
\end{align*}

Let us now compute $\Fdimer(0)$, that is the free energy of the dimer model on the flat graph $\GF_{\mathrm{flat}}$.
By Fisher's correspondence, for every $n\geq 1$, the number of  dimer configurations of the toroidal graph
$\GF_{n,\mathrm{flat}}=\GF_{\mathrm{flat}}/n\ZZ^2$ is equal to $2^{\vert\Vs_{1}\vert n^2}$ times the weighted number of
polygon configurations of $\Gs_{n,\mathrm{flat}}$ with edge-weights $\frac{\sn\theta }{1+\cn\theta }$ arising from the high temperature expansion.
Let us describe $\Gs^{\mathrm{flat}}$, see also~\cite{Kenyon3,BoutillierdeTiliere:iso_gen}.
Since the sum of the rhombus-angles around vertices of $\Gs_{\mathrm{flat}}$ is $2\pi$, and since rhombus half-angles are equal to $0$ or $\frac{\pi}{2}$,
there is around every vertex exactly two rhombi with rhombus half-angle $\frac{\pi}{2}$, with corresponding high temperature expansion weight 
equal to $1$. The other rhombi have rhombus half-angle $0$; with corresponding weight equal to $0$. The graph $\Gs_{n,\mathrm{flat}}$ thus consists of 
$p$ disjoint cycles covering all vertices, for some $p=O(n)$. A polygon configuration has even degree at every vertex
so that for each such cycle there is exactly two polygon configurations. As a consequence, 
$\Zdimer(\Gs_{n,\mathrm{flat}})=2^{\vert\Vs_{1}\vert n^2}2^p$, implying that $\Fdimer(0)=-\vert\Vs_{1}\vert\log 2$. 

From the geometric description of the graph $\Gs^{\mathrm{flat}}$, we also know that 
\begin{equation*}
     \vert\{ e\in\Es_1:\overline{\theta}_e^{\mathrm{flat}}=0\}\vert=\vert\Es_1\vert-\vert\{ e\in\Es_1:\overline{\theta}_e^{\mathrm{flat}}=\pi/2\}\vert=
     \vert\Es_1\vert-\vert\Vs_1\vert.
\end{equation*}

As a consequence, the dimer free energy $\Fdimer=\Fdimer(1)$ is equal to
{\begin{multline*}
-(\vert \Es_1\vert +\vert\Vs_{1}\vert)\frac{\log 2}{2}
\\+\sum_{e\in\Es_1}
\left(
-\frac{1}{2}\log\frac{\sn\theta_e}{1+\cn\theta_e}+
\left(\frac{1-2\Hh(2\theta_e)}{2}\right)\log\sc\theta_e 
+\int_{\theta_e^{\mathrm{flat}}}^{\theta_e} 2{\Hh}'(2\theta)\log\sc\theta\,\ud\theta
\right),
\end{multline*}}
and the proof is concluded by recalling that $\frac{\sn\theta_e}{1+\cn\theta_e}=\sc \frac{\theta_e}{2}\dn\frac{\theta_e}{2}$.
\end{proof}

We obtain the free energy of the $Z$-invariant Ising model, similar to the one given by
Baxter, see~\cite[(7.12.7)]{Baxter:exactly} and \cite{Baxter:8V}.

\begin{cor}
\label{cor:free_energy_Ising}
The free energy of the $Z$-invariant Ising model is equal to 
\begin{multline*}
\Fising^k=-\vert\Vs_{1}\vert\frac{\log 2}{2}-\vert\Vs_{1}\vert\int_{0}^{K} 2{\Hh}'(2\theta)\log\sc\theta\,\ud\theta
\\+\sum_{e\in\Es_1}\left(
-\Hh(2\theta_e)\log\sc\theta_e +\int_{0}^{\theta_e} 2{\Hh}'(2\theta)\log\sc\theta\,\ud\theta
\right).
\end{multline*}
\end{cor}

\begin{proof}
From the high temperature expansion, see Equation~\eqref{equ:HTE}, we have for every $n\geq 1$,
\begin{equation*}
     \Zising(\Gs_n,\Js)=\left(\prod_{e\in\Es_n}\cosh (\Js_e)\right)\Zdimer(\GF_n,\nu).
\end{equation*}
Let us compute $\cosh(\Js_e)$ for the $Z$-invariant coupling constants.
By Equation~\eqref{equ:dheart}, we have $\cosh(2\Js_e)=\nc\theta_e$, so that
\begin{equation}\label{equ:utile_2}
     \cosh(\Js_e)=\sqrt{\frac{1+\cosh(2\Js_e)}{2}}=\sqrt{\frac{1+\cn\theta_e}{2\cn\theta_e}}
=\sqrt{\frac{1+\cn\theta_e}{2\sn\theta_e}\sc\theta_e }.
\end{equation}
As a consequence, 
\begin{align}
\Fising&=\Fdimer-\sum_{e\in\Es_1} \log \cosh(\Js_e)\label{equ:utile}\\
&=\Fdimer+\vert \Es_1\vert\frac{\log 2}{2} +\sum_{e\in\Es_1} \Bigl(
\frac{1}{2}\log \frac{\sn\theta_e}{1+\cn\theta_e} - \frac{1}{2}\log\sc\theta_e \Bigr)\nonumber\\
&=-\vert\Vs_{1}\vert\frac{\log 2}{2}+\sum_{e\in\Es_1}
\left(
-\Hh(2\theta_e)\log\sc\theta_e +\int_{\theta_e^{\mathrm{flat}}}^{\theta_e} 2{\Hh}'(2\theta)\log\sc\theta\,\ud\theta
\right).\nonumber
\end{align}
We can moreover write
\begin{align*}
\sum_{e\in\Es_1}\int_{\theta_e^{\mathrm{flat}}}^{\theta_e} 2{\Hh}'(2\theta)\log\sc\theta\,\ud\theta&=
\sum_{e\in\Es_1}\Biggl(\int_{0}^{\theta_e} 2{\Hh}'(2\theta)\log\sc\theta\,\ud\theta 
-\int_{0}^{\theta_e^{\mathrm{flat}}} 2{\Hh}'(2\theta)\log\sc\theta\,\ud\theta\Biggr)\\
&=\sum_{e\in\Es_1}\Biggl(\int_{0}^{\theta_e} 2{\Hh}'(2\theta)\log\sc\theta\,\ud\theta\Biggr)
-\vert\Vs_{1}\vert\int_{0}^{K} 2{\Hh}'(2\theta)\log\sc\theta\,\ud\theta,
\end{align*}
since there are $\vert\Vs_{1}\vert$ edges whose rhombus half-angle is $\frac{\pi}{2}$ in $\Gs^{\mathrm{flat}}$. This concludes the proof.
\end{proof}

\begin{cor}
\label{cor:link_free_energies}
The free energy of the $Z$-invariant Ising model and spanning forests are related by
\begin{equation}
\label{eq:link_free_energies}
\Fising^{k}=-\vert\Vs_{1}\vert\frac{\log 2}{2}+\frac{\Fforest^k}{2}.
\end{equation}
\end{cor}

\begin{proof}
This is obtained by comparing the expressions for the free energies proved in Corollary~\ref{cor:free_energy_Ising} and \cite[Theorem 2]{BdTR1}.
\end{proof}

\begin{rem}
In the case $k=0$, Corollary \ref{cor:free_energy_Ising} is obtained in \cite[Theorem 3]{BoutillierdeTiliere:iso_gen}; the free energy becomes (the key point being that as $k\to0$, $H(u)\to\frac{u}{2\pi}$, see Lemma \ref{lem:properties_Hh_Hv})
\begin{equation*}
     \Fising^0=-\vert\Vs_{1}\vert\frac{\log 2}{2}-\sum_{e\in\Es_1}\left[\frac{\theta_e}{\pi}\log\tan\theta_e+\frac{1}{\pi}\Bigl(L(\theta_e)+L\Bigl(\frac{\pi}{2}-\theta_e\Bigr)\Bigr)\right],
\end{equation*}
where $L$ is the Lobachevsky function: $L(\theta_e)=-\int_{0}^{\theta_e}\log(2\sin(t))\text{d}t$. Moreover, for $k=0$, Corollay \ref{cor:link_free_energies} is derived in \cite{BoutillierdeTiliere:iso_gen}, below Theorem 3.
\end{rem}

\begin{rem}
\label{rem:feb}
It is possible to recover Baxter's ``local'' formula for the free energy when
$k\neq0$ without deformation argument, but rather by computing more directly the
double integral~\eqref{equ:free_energy_fourier}, by first fixing $w$ and evaluating the
integral in $z$ with residues: since $\Pdimer$ is reciprocal, and has no root on
the unit torus, then up to a multiplicative constant, one can rewrite
$\Pdimer(z,w)$ for $\vert z\vert =1$ as
\begin{equation*}
  \prod_{j=1}^d z_j(w)^+\left(1-\frac{z}{z_j^+(w)}\right)
    \left(1-\frac{z_j^-(w)}{z}\right),
\end{equation*}
where $z_j^{\pm}(w)$, $j=1,\ldots,d$ are the roots of $\Pdimer(\cdot,w)$,
$\vert z_j^+(w)\vert>1$ and $z_j^-(w) = (z_j^+(w))^{-1}$.
The $\log$ of this product can be expanded in series in $z$, whose term of
degree 0 is the sum of logarithms of the roots $z_j^+(w)$: this is the
contribution we get when dividing by $2i\pi z$  and integrating over the unit
circle. Ending here the computation for the square lattice yields the expression
of the free energy of the Ising model on $\ZZ^2$ given
in~\cite[Equation~(7.9.16)]{Baxter:exactly}. We can go even further, for general
periodic isoradial graphs, if we perform the change of variable from $w$ to $u$, as
in~\cite[Section~5.5]{BdTR1} to pass from the Fourier expression to the local
expression of the massive Green function on periodic isoradial graphs.
The roots $(z,w)$ of $\Pdimer$ where $\vert z\vert>1$ and $\vert w\vert =1$ form a closed curve on
the spectral curve, which is mapped under $\log\vert \cdot\vert $ to a horizontal segment
joining the two connected components of the boundary of the amoeba, and lifts on
$\TT(k)$ as a contour $\Gamma$ winding once vertically. We can then write
\begin{equation*}
  \int_{\vert w\vert =1} \sum_{j=1}^d \log z^+_j(w) \frac{\ud w}{2i\pi w}=
  \oint_\Gamma \log z(u) \frac{w'(u)}{w(u)} \frac{\ud u}{2i\pi},
\end{equation*}
where $z(u)$ and $w(u)$ are given by~\eqref{equ:param_spec_curve}. As
in~\cite[Section~5.3.3]{BdTR1} when computing the area of the hole of the
amoeba, the product of $\log z(u)$ and $w'(u)/w(u)$ can be rewritten as a sum
over intersections of train-tracks on $\Gs_1$, \textit{i.e.}, over edges of the
fundamental domain, thus yielding a local formula.
\end{rem}

\subsubsection{Computing the constant in $\Pdimer(z,w)=c\Plapm(z,w)$}
\label{sec:prop_char_pol}

With the computations of free energies, we can now compute the constant in Proposition~\ref{prop:equality_charact_pol}.

\begin{cor}
The dimer characteristic polynomial of the graph $\GF$ arising from the $Z$-invariant Ising model and the $Z$-invariant spanning forests
characteristic polynomial are related by 
\begin{equation}
\Pdimer(z,w)=2^{\vert \Vs_1\vert +\vert \Es_1\vert } \prod_{e\in\Es_1} 
\left(\frac{\cn\theta_e}{1+\cn\theta_e}\right) \Plapm(z,w).
\end{equation}
\end{cor}

\begin{proof}
In Proposition~\ref{prop:equality_charact_pol}, we proved that $\Pdimer(z,w)=c\Plapm(z,w)$. Moreover, 
the dimer and spanning forests free energies can be computed using the characteristic polynomials:
\begin{align*}
\Fdimer&=-\frac{1}{2}\iint_{\{\vert z\vert =\vert w\vert =1\}} \log \Pdimer(z,w)\frac{\ud z}{2i\pi z} \frac{\ud w}{2i\pi w},\\ 
\Fforest&=-\iint_{\{\vert z\vert =\vert w\vert =1\}} \log \Plapm(z,w)\frac{\ud z}{2i\pi z} \frac{\ud w}{2i\pi w}.
\end{align*}
This implies that
\begin{align*}
\log c&=\Fforest-2\Fdimer\\
      &=2\Fising-2\Fdimer+\vert \Vs_1\vert \log 2,\quad \text{ by Corollary~\ref{cor:link_free_energies}},\\
      &=-2\sum_{e\in\Es_1} \log \cosh\Js_e+\vert \Vs_1\vert \log 2,\quad \text{ by Equation~\eqref{equ:utile},}\\
      &=\vert \Es_1\vert \log 2 +\sum_{e\in\Es_1}\log 
      \left(\frac{\cn\theta_e}{1+\cn\theta_e}\right)
      +\vert \Vs_1\vert \log 2,\quad\text{ by \eqref{equ:utile_2}}\qedhere.
\end{align*}
\end{proof}
This is coherent with \cite[Corollary 14]{BoutillierdeTiliere:iso_gen}, concerning the case $k=0$.

\section{Duality and phase transition in the $Z$-invariant Ising model}
\label{sec:duality_and_phase_transition}

This section is about the behavior of the $Z$-invariant Ising model as the elliptic parameter $k$ varies. 
Section~\ref{sec:duality} exhibits a duality relation in the sense of Kramers and Wannier, see also~\cite{CGNP,McCM}. 
In Section~\ref{sec:phase_transition} we derive
the phase diagram of the model and compare it to that of the $Z$-invariant spanning forests of~\cite{BdTR1}. 
In Section~\ref{subsec:self_duality_Ising} we extend to all isoradial graphs a self duality relation proved by Baxter in the case of the 
triangular lattice. Finally in Section~\ref{sec:modular_group} we relate duality transformations to the 
modular group.

\subsection{Dual elliptic modulus}

\label{subsec:appendix_negative/dual}

Let $k\in[0,1)$ be a fixed elliptic modulus, and recall the notation $k'=\sqrt{1-k^2}$ for the complementary parameter. 
By definition, the dual parameter of $k$ is:
\begin{equation}
\label{equ:duality_parameter}
     k^* = i\frac{k}{k'}, \text{ or equivalently } {k^*}' = \frac{1}{k'}.
\end{equation}
Notice that $k^2\mapsto {k^{*2}}$ is an involutive bijection between $[0,1)$ and $(-\infty,0]$. As we shall see, the dual parameter $k^*$ parametrizes the dual temperature. 

We need the following
duality identities relating elliptic integrals and Jacobi elliptic functions with parameters $k$ and $k^*$:
\begin{align}
&\sqrt{k'}K(k)=\sqrt{{k^*}'}K(k^*), \quad\label{id:K_imaginary_modulus}\text{\cite[17.4.17]{AS}}\\
&\sqrt{k'}\sc(u\vert k)=\sqrt{{k^*}'}\sc(k'u\vert k^*),\quad\label{id:sc_imaginary_modulus} 
\text{\cite[16.10.2 and 16.10.3]{AS}}.
\end{align}

\subsection{Duality in the $Z$-invariant Ising model}\label{sec:duality}

Kramers and Wannier's duality~\cite{KramersWannier1,KramersWannier2} says the following. Consider an Ising model on a planar graph $\Gs$ with coupling constants 
$\Js$ and an Ising model on $\Gs^*$ with coupling constants $\Js^*$. Perform the high temperature expansion of the first Ising model,
and low temperature expansion of the second. Both expansions yield polygon configurations on $\Gs$. 
The Ising models are said to be \emph{dual} if both induce the 
same measure on polygon configurations. This is true if the coupling constants satisfy the following \emph{duality relation}:
\begin{equation}
\label{equ:duality0}
\forall\,e\in\Es, \, \tanh(\Js_e)=e^{-2\Js^*_{e^*}} \quad \Longleftrightarrow \quad \sinh(2\Js_e)\sinh(2\Js^*_{e^*})=1,
\end{equation}
where $e^*$ is the dual edge of $e$. Duality maps a high temperature Ising model on a low temperature one and vice versa.

In this setting, the $Z$-invariant Ising model on $\Gs$ with 
parameter $k$ and the one on $\Gs^*$ with parameter $k^*$ are dual models, see also~\cite{CGNP} for the case of the triangular/hexagonal lattices. Indeed, 
for the first and second model we respectively have, for every edge $e$ of $\Gs$ and dual edge $e^*$ of $\Gs^*$,
\begin{equation*}
     \left\{\begin{array}{rl}
\sinh(2\Js_e) =\hspace{-2mm}& \displaystyle\sinh(2\Js(\overline{\theta}_e\,\vert \,k))=
\sc\Bigl(\overline{\theta}_e \frac{2K(k)}{\pi} \Big\vert k\Bigr),\smallskip\\
\sinh(2\Js^*_{e^*}) =\hspace{-2mm}& \displaystyle \sinh\Bigl(2\Js\Bigl(\frac{\pi}{2}-\overline{\theta}_e\Big\vert k^*\Bigr)\Bigr)=
\sc\Bigl(K(k^*)-\overline{\theta}_e \frac{2K(k^*)}{\pi} \Big\vert k^*\Bigr).\end{array}\right.
\end{equation*}
Moreover, 
\begin{align*}
\sc\Bigl(K(k^*)-\overline{\theta}_e \frac{2K(k^*)}{\pi} \Big\vert k^*\Bigr)&=
\frac{1}{{k^*}'}\cs\Bigl(\overline{\theta}_e \frac{2K(k^*)}{\pi} \Big\vert k^*\Bigr),\quad \text{ by Table~\ref{table:identities_Jacobi_function},}\\
&=\cs\Bigl(\overline{\theta}_e \frac{2K(k)}{\pi} \Big\vert k\Bigr),\quad  
\text{ by \eqref{equ:duality_parameter}, \eqref{id:K_imaginary_modulus} and \eqref{id:sc_imaginary_modulus},}
\end{align*}
from which we deduce the duality relation \eqref{equ:duality0} for $Z$-invariant Ising models:
\begin{equation}
\label{equ:duality}
     \sinh(2\Js(\overline{\theta}_e\vert k))\sinh\Bigl(2\Js\Bigl(\frac{\pi}{2}-\overline{\theta}_e\Big\vert k^*\Bigr)\Bigr)=1.
\end{equation}
The elliptic parameters $k$ and $k^*$ can be interpreted as parametrizing dual temperatures.

Note that this duality relation together with the computation of Equation~\eqref{ex:proba_comput}
can be used to obtain the probability of an edge occurring in a polygon configuration arising from the low temperature 
expansion of the Ising model.

\subsection{Range of the $Z$-invariant Ising model and phase transition}\label{sec:phase_transition}

The following lemma shows that the $Z$-invariant coupling constants are analytic in $k^2$ and that
the whole range of temperatures is covered as the parameter $k$ varies.

\begin{lem}
  \label{lem:poids_croissants}
  Let $\Gs$ be an isoradial graph and consider an edge $e$ of $\Gs$. Then the coupling constant
  \begin{equation*}
  \Js(\overline{\theta}_e\vert k)=\frac{1}{2} 
    \log\left(\frac{1+\sn(\theta_e\vert k)}{\cn(\theta_e\vert k)}\right)
  \end{equation*}
defined in \eqref{eq:def_weight}, seen as a
function of $k^2$, is analytic on $(-\infty,1)$ and increases from $0$ to
$\infty$ as $k^2$ goes from $-\infty$ to $1$. 
\end{lem}
\begin{proof}
Jacobi's amplitude $\am(\cdot\vert k)$, defined by $\am(u\vert k)=\int_{0}^{u}\dn(v\vert k)\ud v$, 
relates Jacobi and classical 
trigonometric functions. For the function $\sc$ we have
\begin{equation*}
     \sc(u\vert k)=\tan(\am(u\vert k)).
\end{equation*}
The fact that $\Js(\overline{\theta}_e\vert k)$ is an increasing function of $k^2$ on $[0,1)$ comes from the relation
\begin{equation}
\label{eq:Ising_weights_tan_am}
     \sinh(2\Js(\overline{\theta}_e\vert k))=\sc\Bigl(\overline{\theta}_e\frac{2K(k)}{\pi} \vert k\Bigr)
     =\tan(\am(\overline{\theta}_e \frac{2K(k)}{\pi}\vert k))
\end{equation}
together with the fact that $\sinh^{-1}$, $\tan$, $\am$ and $K$ are all
increasing functions. Moreover, as $k\to1$, $\am(u\vert k)\to
2\arctan(e^u)-\frac{\pi}{2}$, see \cite[16.15.4]{AS}, and thus
$\am(\overline{\theta}_e \frac{2K(k)}{\pi}\vert k)\to\frac{\pi}{2}$. Using
\eqref{eq:Ising_weights_tan_am} then leads to
\begin{equation*}
  \lim_{k\to1} \Js(\overline{\theta_e}\vert k)=\lim_{k\to1}
  \sinh(2\Js(\overline{\theta_e}\vert k))=\infty.
\end{equation*}
From the duality relation~\eqref{equ:duality} and from the case
$k^2\in[0,1)$, we deduce that $\Js(\overline{\theta_e}\vert k)$ is increasing on $k^2\in(-\infty,0]$
and goes to $0$ as $k^2\to-\infty$.

The analyticity of $\Js(\overline{\theta_e}\vert k)$ is clear on $(-\infty ,0)$ and
$(0,1)$. For the neighborhood of $0$ we use the series expansion of $\sc$ in
terms of the Nome $q=\exp(-\pi K'(k)/K(k))$, see \cite[16.23.9]{AS}:
\begin{equation}
\label{eq:expansion_sc_Nome}
     \sc\Bigl(\overline{\theta}_e\frac{2K(k)}{\pi}\vert k\Bigr)=\frac{\pi}{2\sqrt{1-k^2}K(k)}\tan(\overline{\theta}_e)+\frac{2\pi}{\sqrt{1-k^2}K(k)}\sum_{n=1}^{\infty}(-1)^{n}\frac{q^{2n}}{1+q^{2n}}\sin(2n\overline{\theta}_e).
\end{equation}
\end{proof}

The following proves a \emph{second order} phase transition, which together with Lemma~\ref{lem:poids_croissants} allows to derive the 
phase diagram of the $Z$-invariant Ising model (see Figure \ref{Fig:phase_diagram}). Comments on the result are given in Remark~\ref{rem:phase_transition}.

\begin{cor}[Phase transition]
\label{cor:phase_transition}
The free energy $\Fising^{k}$ of the $Z$-invariant Ising model on an isoradial graph $\Gs$
is analytic for $k^2\in\mathbb R\setminus \{0\}$. The model has continuous
(or \emph{second order}) phase transition at $k=0$. Namely, one has
\begin{equation*}
     \Fising^{k}=\Fising^{0}-\vert k\vert^2\log \vert k\vert^{-1}\frac{\vert \Vs_{1}\vert}{2}+O(k^2).
\end{equation*}
\end{cor}

\begin{proof}
We start from Corollary~\ref{cor:link_free_energies} relating the free energy of the $Z$-invariant spanning forests and of the Ising model, 
and from \cite[Theorems~36 and 38]{BdTR1}. The analyticity for $k^2>0$ and the phase transition when $k^2>0$ tends to $0$ immediately follow 
from these results.

In the case $k^2<0$, the key point is the expression of the free energy of rooted spanning forests, 
derived in \cite[Theorem~38]{BdTR1} whenever $k^2>0$. It is an integral expression in terms of the function $\Hh$. It turns out that a 
similar expression holds in the case $k^2<0$. 
Performing an asymptotic expansion (along the same lines as in the proof of \cite[Theorem~38]{BdTR1}) of the so-obtained expression of 
$\Fising^{k}$ when $k$ tends to $0$ then concludes the proof.

In the case $k^2<0$, we could also use the forthcoming Corollary~\ref{cor:free_energy_self_dual} relating $\Fising^{k}$ and $\Fising^{k^*}$. 
This allows us to express $\Fising^{k}$ in terms of the free energy associated with the elliptic modulus $k^*$, whose square is positive. 
In this way we can again use \eqref{eq:link_free_energies} and \cite[Theorem~38]{BdTR1}.
\end{proof}

\begin{rem}\label{rem:phase_transition}$\,$
\begin{itemize}
\item In the particular case of the square lattice, Corollary \ref{cor:phase_transition} is derived in \cite[(7.12.7)]{Baxter:exactly},
proving criticality at $k=0$ of the $Z$-invariant Ising model on $\ZZ^2$. Our result is thus a generalization of the latter to all 
isoradial graphs. 
\item Criticality for the $Z$-invariant Ising model has been proved in~\cite{Li:critical,CimasoniDuminil,Lis}, 
with a different parametrization of the temperature and with different techniques:
the authors multiply the $Z$-invariant weights at $k=0$ 
by an inverse temperature parameter $\beta$, and prove that when 
$\beta$ is equal to the particular value $\beta_c:=1$, the model
is critical.
In their setting, the difference $\beta-\beta_c$ is to be compared to the first
nonzero term in the expansion of $J(\overline{\theta}|k)-J(\overline{\theta}|0)$ as $k$ goes to zero.
This expansion, which takes into account the fact that
$\theta=\overline{\theta}\frac{2K(k)}{\pi}$ is itself a function of
$k$, reads:
\begin{equation*}
  J(\overline{\theta}|k)=J(\overline{\theta}|0)+\frac{k^2}{8}\sin(\bar{\theta})+O(k^4),
\end{equation*}
yielding that $k^2 \asymp (\beta-\beta_c)$, which allows to compare the free
energy as a function of $\beta$ close to $\beta_c$ (as the one given by
Onsager~\cite{Onsager}) and Corollary~\ref{cor:phase_transition}.

\item What is remarkable and not present in the physics literature is that the phase transition of the $Z$-invariant Ising model 
is (up to a multiplicative factor $\frac{1}{2}$) the same as the phase transition of the $Z$-invariant spanning forest model.
As explained in the proof, this follows from fact that the free energies of the two models are related by a simple formula proved in Corollary~\ref{cor:link_free_energies}.

\end{itemize}
\end{rem}
We deduce the following phase diagram for the $Z$-invariant Ising model:
\begin{itemize}
     \item $k^2=0$: critical Ising model,
     \item $k^2\in(0,1)$: low-temperature Ising model,
     \item $k^2\in(-\infty,0)$: high temperature Ising model.
\end{itemize}
Note that the phase diagram is nicer when expressed with the complementary elliptic modulus $(k')^2=1-k^2\in(0,\infty)$, 
see Figure~\ref{Fig:phase_diagram}. 
In the rest of the paper we have nevertheless chosen to use the elliptic parameter $k$ since it is the one classically 
used in the notation of elliptic functions.

\smallskip

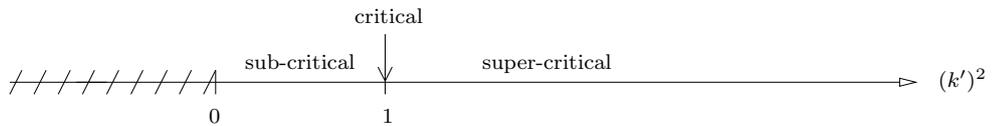
\begin{figure}[ht]
  \centering
\begin{overpic}[width=12cm]{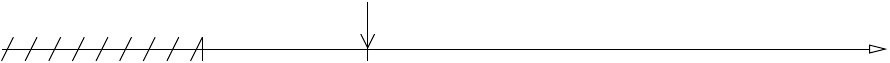}
\put(102,1){\scriptsize $(k')^2$}
\put(41,-3){\scriptsize $1$}
\put(22,-3){\scriptsize $0$}
\put(38,8){\scriptsize critical}
\put(26,3){\scriptsize sub-critical}
\put(52,3){\scriptsize super-critical}
\end{overpic}
\smallskip
\caption{Phase diagram of the $Z$-invariant Ising model on an isoradial graph $\Gs$ 
as a function of the complementary elliptic modulus $(k')^2$.\label{Fig:phase_diagram}}
\end{figure}

The domain $(k')^2<0$ (or equivalently $k^2>1$) on Figure \ref{Fig:phase_diagram} corresponds to the reciprocal parameter, also named Jacobi's real transformation. Jacobi elliptic functions are still defined for $k^2>1$, see \cite[16.11]{AS}. However, one main difference is that the period $K(k)$ is not real anymore. Accordingly, the angles $\theta_e=\overline{\theta}_e \frac{2K}{\pi}$ have a non-zero imaginary part.

We now examine the effect of this reciprocal transformation on the coupling constant $J(\overline{\theta}_e\vert k)$ defined in \eqref{eq:def_weight}.
As it does not seem natural to use complex angles, 
we extend the formula \eqref{eq:def_weight} in the regime $k^2>1$ with only the real parts of the angles. Using \cite[16.11]{AS} in \eqref{eq:def_weight}, one finds that for $k^2>1$
\begin{equation}
\label{eq:coupling_constant_reciprocal}
     J(\overline{\theta}_e\vert k)=\frac{1}{2}\log \left(\frac{1+\frac{1}{k}{\rm sn}(\overline{\theta}_e\frac{2K(1/k)}{\pi}\vert 1/k)}{{\rm dn}(\overline{\theta}_e\frac{2K(1/k)}{\pi}\vert 1/k)}\right).
\end{equation}
Using similar arguments as in Lemma \ref{lem:poids_croissants}, we see that when $k^2$ goes from $1$ to $+\infty$, the coupling constant \eqref{eq:coupling_constant_reciprocal} goes decreasingly from $+\infty$ to $0$. 
So this is the same as for the classical regime $k^2\in(-\infty,1)$; 
this range of parameter does not allow to reach a new regime, as for instance the anti-ferromagnetic regime.

\subsection{Self-duality for the $Z$-invariant Ising model}
\label{subsec:self_duality_Ising}

We now prove a self-duality relation for the $Z$-invariant massive Laplacian. 
From this and Corollary~\ref{cor:link_free_energies}, we deduce a self-duality relation for the 
$Z$-invariant Ising model,
see Remark~\ref{rem:self_duality}.

\begin{lem}
\label{lem:Laplacian_proportional}
The Laplacian operators associated to $k$ and $k^*$ satisfy the following self-duality relation:
\begin{equation*}
     \sqrt{{k^*}'}\Delta^{m(k^*)}=\sqrt{k'}\Delta^{m(k)},
\end{equation*}     
and hence the discrete massive harmonic functions are the same.
\end{lem}

\begin{proof}
Due to the particular form of the Laplacian operator \eqref{eq:Laplacian_operator}, Lemma \ref{lem:Laplacian_proportional} is equivalent to 
proving that $\sqrt{k'}\sc({\theta}_{e}\vert k)=\sqrt{k'}\sc(\frac{2K(k)}{\pi}\overline{\theta}_{e}\vert k)$ and 
$\sqrt{k'}\sum_{j=1}^{n}\Arm(\theta_j\vert k)$ are self-dual. For the first quantity, this directly follows from \eqref{id:sc_imaginary_modulus} and \eqref{id:K_imaginary_modulus}. For the second one, we write $\theta_j=\frac{\alpha_{j+1}-\alpha_j}{2}$ and use $n$ times the addition theorem \eqref{cor:Armbis:item3} and finally the periodicity relation~\eqref{cor:Armbis:item1}. We obtain
\begin{equation*}
     \sqrt{k'}\sum_{j=1}^{n}\Arm(\theta_j\vert k)=-\sum_{j=1}^{n}\{\sqrt{k'}\sc({\alpha}_{j+1}\vert k)\}\{\sqrt{k'}\sc({\alpha}_{j}\vert k)\}\{\sqrt{k'}\sc({\alpha}_{j+1}-{\alpha}_{j}\vert k)\},
\end{equation*}     
which is self-dual, for the same reasons as previously.
\end{proof}

\begin{cor}
\label{cor:free_energy_self_dual}
The free energy of the $Z$-invariant Ising model on the graph $\Gs$ satisfies the following self-duality relation.
\begin{equation*}
     \Fising^k+\frac{\vert\Vs_1\vert}{2}\log k'=\Fising^{k^*}+\frac{\vert\Vs_1\vert}{2}\log {k^{*}}'.
\end{equation*}
\end{cor}

\begin{rem}\label{rem:self_duality}$\,$
\begin{itemize}
\item In the case of the triangular lattice this is proved in~\cite[(6.5.1)]{Baxter:exactly}, our result thus extends the latter
to all isoradial graphs.
\item There is no simple self-duality relation 
between the coupling constants $\Js(\overline{\theta}_e\vert k)$ and $\Js(\overline{\theta}_e\vert k^*)$ so that
this result is not straightforward. Baxter's argument in the triangular case reads as 
follows: he transforms the $Z$-invariant Ising model with parameter $k$ on the triangular lattice into the 
one on the honeycomb lattice with the same parameter $k$ by using $Y$-$\,\Delta$ moves, from this he deduces that the partition functions differ by
an explicit constant; then he uses Kramers and Wannier duality to map the partition function of the $Z$-invariant Ising model with parameter $k$ 
on the honeycomb lattice into the one of the triangular lattice (dual graph) with parameter $k^*$. Making
the constants explicit allows to relate the free energies with parameters $k$ and $k^*$ on the triangular lattice. 
It is not obvious that this argument should extend to general isoradial graphs,
even though one can
generically go from a periodic isoradial graph
to its
dual using $Y$-$\,\Delta$ moves\footnote{%
  If on the torus, there is a unique train-track with a given homology class,
  making it move across the torus through every pairwise intersection once by
  $Y$-$\,\Delta$ moves yields the same rhombic graph, but with the role of
  primal and dual vertices exchanged. If there are several of them, then moving
  them all to put the first one instead of the second one, the second instead of
  the third, etc., yields the result. Note though that this is not possible in the case
  when there are just two different homology classes, like in the case of the square
  lattice.
}:
when working out the constants in Baxter's computation, there seems to be 
some cancellations that are specific to the triangular and honeycomb lattices. 
\item This self-duality relation and the assumption of uniqueness of the critical point is used in~\cite[(6.5.5)--(6.5.7)]{Baxter:exactly}
to compute the critical temperature of the Ising model on the triangular and honeycomb lattice. Corollary~\ref{cor:free_energy_self_dual} 
allows to extend this physics argument to all isoradial graphs.
\end{itemize}
\end{rem}

\subsection{Dualities of the Ising model and the modular group}
\label{sec:modular_group}

Various changes of the elliptic modulus $k$ are considered throughout this article. Besides the intrinsic complementary transformation $k\mapsto k'$, we have seen the importance of the dual transformation $k\mapsto k^*$, as many quantities are self-dual: 
\begin{itemize}
     \item $\sqrt{k'}K(k)$, see \eqref{id:K_imaginary_modulus}; 
     \item $\sqrt{k'}\sc(\frac{2K(k)}{\pi}u\vert k)$, see \eqref{id:sc_imaginary_modulus}; 
     \item the exponential function \eqref{eq:recursive_def_expo}; 
     \item the modified Laplacian operator, see Lemma \ref{lem:Laplacian_proportional}; 
     \item the modified free energy, see Corollary \ref{cor:free_energy_self_dual};
     \item the rescaled function $\Hh$, as $\Hh(K(k)u\vert k)=\Hh(K(k^*)u\vert k^*)$, see \eqref{eq:definition_Hh_Hv_k2<0}.
\end{itemize}     
Moreover, in the proof of Theorem~\ref{thm:asymptotics_inverse_Kasteleyn}, we make use of the ascending Landen transformation $k\mapsto \frac{2-k^2-2\sqrt{1-k^2}}{k^2}$ (note, this does not appear explicitly in the proof, as we refer to the companion paper \cite{BdTR1} for the details). 

Our aim in this paragraph is twofold: first we reformulate the self-duality as a parity property of expansions in terms of the Nome $q$, 
then we relate the various transformations of $k$ to the modular group. Links between the Ising model and the modular group already exist in the 
physics literature, see in particular \cite[Chapter 8]{Krieger} as well as \cite{MaBo-01,Bostan-11}.

\subsubsection{Self-duality and expansions in terms of the Nome}

Let us first mention that a function of the elliptic modulus $k^2$ is analytic at $0$ if and only if it is analytic at $0$ as a function of the Nome $q=e^{-\pi K'/K}$. Indeed, $q$ is analytic in terms of $k^2$ and 
\begin{equation*}
     q=\frac{k^2}{16}+8\left(\frac{k^2}{16}\right)^2+84\left(\frac{k^2}{16}\right)^3+992\left(\frac{k^2}{16}\right)^4+\ldots,\quad\text{(\cite[17.3.21]{AS})}.
\end{equation*}
Accordingly, any generic quantity of our article admits an analytic expansion in terms of the Nome at $0$. The following simple criterion translates the self-duality property as a matter of parity:
\begin{lem}
\label{lem:self-duality_parity}
Let $f(k)$ be a function analytic in $k^2$ around $0$. Then $f$ is self-dual (\textit{i.e.}, $f(k)=f(k^*)$) if and only if its expansion in terms of the Nome is even.
\end{lem}
\begin{proof}
When $k$ is replaced by $k^*$, the quarter-periods become $K(k^*) ={k'}K(k)$, see \eqref{id:K_imaginary_modulus}, and
\begin{equation}
\label{id:K*_imaginary_modulus}
     K'(k^*)=k'(K'(k)+iK(k))=\frac{K'(k)+iK(k)}{\sqrt{1-{k^*}^{2}}},\quad \text{(\cite[17.4.17]{AS})}.
\end{equation}
It becomes obvious that $q(k^*)=-q(k)$, Lemma \ref{lem:self-duality_parity} follows.
\end{proof}

With Lemma \ref{lem:self-duality_parity} the question of finding expansions in terms of the Nome comes up. In fact, such expansions typically appear rather indirectly, when writing Fourier expansions of Jacobi functions or elliptic integrals; cf.\ \eqref{eq:expansion_sc_Nome} for the Fourier expansion of the $\sc$ function, as well as \cite[16.23 and 16.38]{AS} and \cite[Section 8.7]{La89} for a more systematic treatment.

On the other hand, the Ising weights \eqref{eq:def_weight}, which are at the heart of our whole construction, are not self-dual. This default of duality is responsible for the non-analyticity (and in some sense of the phase transition, see Corollary \ref{cor:phase_transition}) of the free energy $\Fising^k$ at $0$.

\subsubsection{Dualities and modular group}

The modular group (see \cite[Chapter 9]{La89} for an introduction) is the group generated by the transformations $S(\tau)=-1/\tau$ and $T(\tau)=\tau+1$, acting on the upper half-plane. This group is the set of all transformations 
\begin{equation}
\label{eq:transformations_modular_group}
     \tau\mapsto \frac{c+d\tau}{a+b\tau},
\end{equation}
with $a,b,c,d\in\mathbb Z$ such that $ad-bc=1$. 

Two pairs of complex vectors $(1,\tau)$ and $(1,\tau')$ generate exactly the same lattice $\mathbb Z+\tau\mathbb Z=\mathbb Z+\tau'\mathbb Z$ if 
and only if $\tau'$ is obtained from $\tau$ by a modular transformation \eqref{eq:transformations_modular_group}. 

The quantity $\tau$ should be interpreted as ratios of quarter-periods; for instance $\tau=\frac{iK'}{K}$ (and then the Nome is $q=e^{i\pi\tau}$). 

It is interesting to notice that both generators of the modular group correspond to a duality: $S$ is the complementary duality and $T$ the self-duality, see Table~\ref{table:correspondance_modular}. 

\setlength{\doublerulesep}{\arrayrulewidth}
\begin{table}[ht]
\begin{center}
\begin{tabular}{|| l| l| l||}
  \hline\hline
   \begin{minipage}{0.1\textwidth}\smallskip
   Elliptic 
   modulus\smallskip
   \end{minipage} & \begin{minipage}{0.24\textwidth}\smallskip Name of the change of 
   modulus\smallskip\end{minipage}& \begin{minipage}{0.25\textwidth}\smallskip
   Transformations 
   of the modular group\smallskip\end{minipage}  \\ 
   \hline\hline
    \quad$k$ & No change &  \quad$\tau$ \\
    \hline
    \quad$k'$ & Complementary & \quad$S(\tau)=-1/\tau$  \\
    \hline
    \quad$k^*$ & (Self-)duality & \quad$T(\tau)=\tau+1$  \\
    \hline
    \begin{minipage}{0.14\textwidth}\smallskip   $\frac{2-k^2-2\sqrt{1-k^2}}{k^2}$
    
    \smallskip
    \end{minipage} & Landen transformation & \quad$2\tau$ \\
  \hline\hline
\end{tabular}
\end{center}
\caption{Correspondance between changes of the elliptic modulus and transformations of the modular group.}
\label{table:correspondance_modular}
\end{table}

The set of all transformations \eqref{eq:transformations_modular_group} with
$ad-bc\geq 1$ also forms a group, called the extended modular group. The
quantity $ad-bc$ is then named the order of the transformation
\eqref{eq:transformations_modular_group}. In our elliptic treatment of the Ising
model we have also encountered higher order transformations: namely, the Landen
ascending transformation (used in the proof of Theorem
\ref{thm:asymptotics_inverse_Kasteleyn}) has order $2$, see again Table
\ref{table:correspondance_modular}.

This short discussion suggests that combinatorial links could exist between any
two Ising models associated with elliptic modulus whose $\tau$'s are related by
a transformation \eqref{eq:transformations_modular_group}.

\section{The double $Z$-invariant Ising model via dimers on the graph~$\GQ$}
\label{sec:double_Ising}

In the whole of this section we consider the dimer model on the bipartite graph $\GQ$ arising from two independent 
$Z$-invariant Ising models defined on an infinite isoradial graph $\Gs$. Edges of $\GQ$ are assigned the weight function
$\overline{\nu}$ of~\eqref{eq:dimer_weights_GQ}.

In Section~\ref{sec:functions_kernel} we introduce a one parameter family of functions in the kernel of the Kasteleyn operator $\KQ$ of this dimer model. 
This is the key object used in Section~\ref{sec:invKQ} to prove a local expression for an inverse of the operator $\KQ$. 
In Section~\ref{sec:asymptKQ} we derive asymptotics of this operator, and in Section~\ref{sec:dimerKQ} we derive consequences for the 
dimer model Gibbs measure. We aslo give a few examples of computations.

\subsection{Kasteleyn matrix/operator}

Let us recall the construction of the bipartite graph $\GQ$. Every edge of the graph $\Gs$ is replaced by a ``rectangle'' 
and the latter are glued together in a circular way using \emph{external edges}. Each ``rectangle'' has two edges ``parallel''
to an edge of $\Gs$ and two edges ``parallel'' to the dual edge, see Figure~\ref{fig:GQ}. For instance, if $\Gs$ is the square lattice, $\GQ$ is the square-octagon lattice.

When the graph $\Gs$ is isoradial, so is the graph $\GQ$ with radii of circles being one half of those of $\Gs$. The isoradial embedding
of $\GQ$ is such that external edges have length 0 and ``rectangles'' are real rectangles;
vertices of the rectangles are in the middle of the edges of the diamond 
graph $\GR$, and each rectangle is included in a rhombus of $\GR$, see Figure~\ref{Fig:Kweights} or~\ref{Fig:example_f}. 
For the sequel it is useful to note that every vertex of $\GQ$ belongs to a unique rhombus of $\GR$.

The graph $\GQ$ being bipartite, its vertices can be split into white and black $\VQ=\WQ\bigcup\BQ$. 
In this case, the Kasteleyn matrix $\KQ$ has rows indexed by white vertices and columns by black ones. Following Kuperberg~\cite{Kuperberg}, 
instead of considering
an admissible orientation of $\GQ$ as we have done for $\GF$, one can assign phases $(e^{i\phi_{wb}})_{wb\in\EQ}$ to edges of $\GQ$, 
in such a way that, for every face of $\GQ$ whose boundary vertices are $w_1,b_1,\dots,w_n,b_n$ in counterclockwise order, we have
\begin{equation}\label{equ:flat_phasing}
(-1)^{n-1} \prod_{j=1}^{n} e^{i\phi_{w_j b_j}}e^{-i\phi_{w_{j+1}b_j}}=1.
\end{equation}
In our case, we define the phasing of the edges to be:
\begin{equation*}
e^{i\phi_{wb}}=
\begin{cases}
1&\text{if the edge $wb$ is parallel to an edge $e$ of $\Gs$,}\\
i&\text{if the edge $wb$ is parallel to the dual of an edge $e$ of $\Gs$,}\smallskip\\
-ie^{-i\overline{\theta}}& \begin{minipage}{11cm} if $wb$ is an external edge and $w$ belongs to a rhombus of $\GR$ having half-angle $\theta$.\end{minipage}
\end{cases}
\end{equation*}
The fact that Equation~\eqref{equ:flat_phasing} holds is proved in~\cite[Lemma 4.1]{detiliere_partition}, see also~\cite{Kenyon3}.

Coefficients of the Kasteleyn matrix $\KQ$ are then given by, for every white vertex $w$ and every black vertex $b$ of $\GQ$,
\begin{equation*}
\KQ_{w,b}=e^{i\phi_{wb}}\overline{\nu}_{wb},
\end{equation*}
where $\overline{\nu}$ is the dimer weight function~\eqref{eq:dimer_weights_GQ}, see also Figure~\ref{Fig:Kweights}.

\begin{figure}[ht]
  \centering
\begin{overpic}[width=4.5cm]{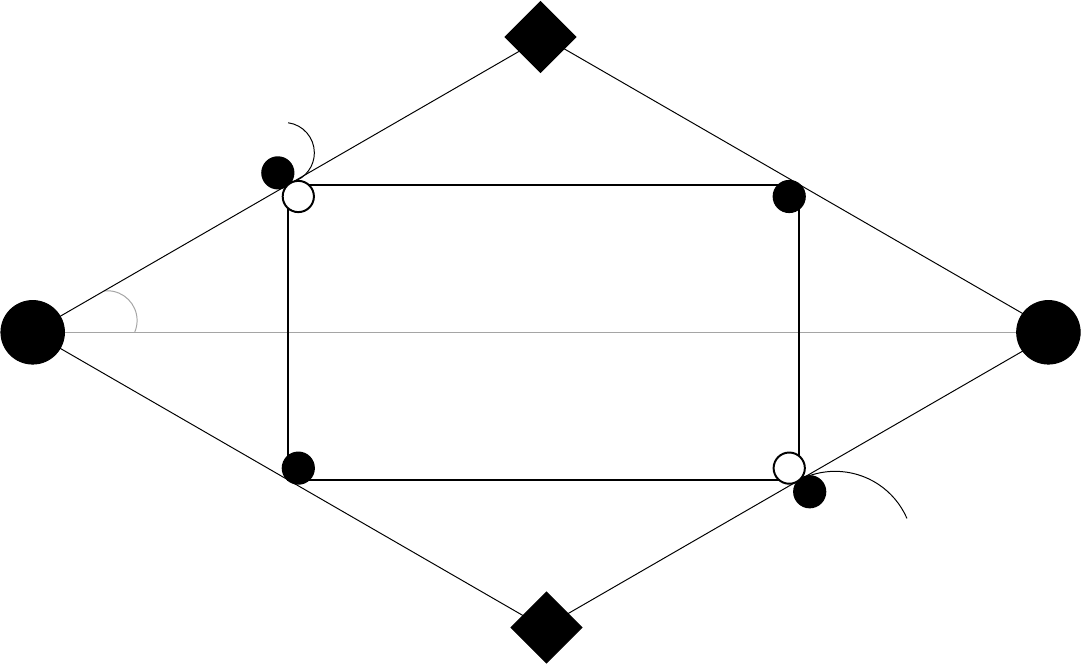}
\put(25,6){\scriptsize $b_1$}
\put(75,45){\scriptsize $b_2$}
\put(75,6){\scriptsize $b_3$}
\put(66,18){\scriptsize $w$}
\put(28,37){\scriptsize $w'$}
\put(15,30){\scriptsize $\theta$}
\put(50,30){\scriptsize $e$}
\put(28,25){\scriptsize $i\cn{\theta}$}
\put(76,25){\scriptsize $i\cn{\theta}$}
\put(46,18){\scriptsize $\sn{\theta}$}
\put(46,45){\scriptsize $\sn{\theta}$}
\put(20,51){\scriptsize $-ie^{-i\overline{\theta}}$}
\put(86,6){\scriptsize $-ie^{-i\overline{\theta}}$}
\end{overpic}
\caption{Coefficients of the Kasteleyn matrix $\KQ$ around a rectangle face of $\GQ$.\label{Fig:Kweights}}
\end{figure}
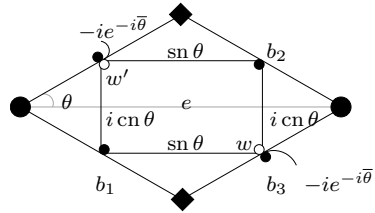

Note that $\KQ$ can also be seen as an operator
mapping $\CC^{\BQ}$ to $\CC^{\WQ}$:
\begin{equation*}
\forall\,f\in\CC^{\BQ},\,\forall\,w\in\WQ,\quad (\KQ f)_{w}=\sum_{b\in\BQ}\KQ_{w,b}\,f_b. 
\end{equation*}
More precisely, since every white vertex $w$ has degree $3$, denoting by $b_1,b_2,b_3$ its neighbors as in  
Figure~\ref{Fig:Kweights}, this relation can be rewritten as,
\begin{equation}\label{equ:MassiveDimer}
\forall\,f\in\CC^{\BQ},\,\forall\,w\in\WQ,\quad (\KQ f)_{w}=
\sn\theta f_{b_1}+i\cn\theta f_{b_2}-ie^{-i\overline{\theta}}f_{b_3}.
\end{equation}

\subsection{Functions in the kernel of the Kasteleyn operator $\KQ$}
\label{sec:functions_kernel}

We now define the function $f$ in the kernel of the Kasteleyn operator $\KQ$. 
Note that it generalizes to the elliptic case the function $f$ introduced by Kenyon in~\cite{Kenyon3} when the bipartite graph is 
$\GQ$.

\paragraph{Rhombus vectors.} In order to define the function $f$, we need to assign rhombus vectors to edges of the graph $\GQ$. 
Since the graph $\GQ$ is isoradial, it also has a diamond graph $(\GQ)^{\diamond}$; note that rhombi of $(\GQ)^{\diamond}$ are obtained by 
cutting those of $\GR$ in four identical rhombi, see Figure~\ref{Fig:example_f}.

Consider an edge $bw$ of $\GQ$.
Then we let
$\frac{1}{2}e^{i\overline{\alpha}}$ and $\frac{1}{2}e^{i\overline{\beta}}$ be 
the two rhombus vectors of $(\GQ)^{\diamond}$ of the edge $bw$, where
$\frac{1}{2}e^{i\overline{\alpha}}$ is on the right of the oriented edge
$(b,w)$. Some examples are given in Figure~\ref{Fig:example_f}.

\begin{rem}
\label{rem:definition_angles}
The angles $\overline{\alpha}$ and $\overline{\beta}$ above are defined so that 
$\overline{\beta}-\overline{\alpha}\in(0,2K)$.
\end{rem}

\begin{defi}
For every edge $bw$ of $\GQ$ and every $u\in\CC$, define
\begin{align}
&\label{eq:def_function_f}f_{(b,w)}(u)=\begin{cases}
\dc(\frac{u-\alpha}{2})\dc(\frac{u-\beta}{2})&\text{if $bw$ is parallel to an edge $e$ of $\Gs$,}\\
{-ik'}{\nc(\frac{u-\alpha}{2})\nc(\frac{u-\beta}{2})}&\text{if $bw$ is parallel to the dual of an edge $e$ of $\Gs$,}\smallskip\\
{ie^{i\overline{\theta}}}{\dc(\frac{u-\alpha}{2})\dc(\frac{u-\beta}{2})}&\begin{minipage}{9cm} if $bw$ is an external edge and $w$ belongs to a 
rhombus of $\GR$ having half-angle $\overline{\theta}$,\end{minipage}
\end{cases}\\
&f_{(w,b)}(u)=(f_{(b,w)}(u))^{-1}.\nonumber
\end{align}
The function $f:\BQ\times\WQ\times\CC\rightarrow \CC$ is then extended to all pairs $(b,w)$ inductively as follows. 
Let $b=b_1,w_1,b_2,w_2,\dots,b_n,w_n=w$ be a path from $b$ to $w$, then:
\begin{equation*}
\forall\,u\in\CC,\quad f_{(b,w)}(u)=\prod_{j=1}^n f_{(b_j,w_j)}(u) \prod_{j=1}^{n-1}f_{(w_j,b_{j+1})}(u).
\end{equation*}
\end{defi}
\begin{rem}
As the function $\gs$ of Definition~\ref{def:fonction_g}, the function $f$ is meromorphic and biperiodic:
\begin{equation*}
     f_{(b,w)}(u+4K) = f_{(b,w)}(u+4iK') = f_{(b,w)}(u).
\end{equation*}
This comes from \eqref{eq:def_function_f} and from the addition formulas of
$\cn$ and $\cd$ by $2K$, see Table~\ref{table:identities_Jacobi_function}.
We therefore also restrict the domain of definition to $\TT(k)=\CC /(4K\ZZ + 4iK'\ZZ).$
\end{rem}

Before proving that this function is well defined, \emph{i.e.}, independent of the choice of path from $b$ to $w$, we give
some examples of computation that are useful for the sequel.
\begin{exm}
\label{example:1}
We compute $f_{(b,w)}(u)$ for $b\in\{b_1,b_2,b_3\}$, where $b_1,b_2,b_3$
are the three black vertices incident to a white vertex $w$ of $\GQ$, see Figure \ref{Fig:Kweights}.
Let $e^{i\overline{\alpha}}$ and $e^{i\overline{\beta}}$ be the rhombus vectors of the rhombus of $\GQ$ containing the rectangle as in Figure \ref{Fig:example_f},
then the two rhombus vectors of $(\GQ)^{\diamond}$ of the edge:
\begin{itemize}
 \item $b_1 w$ are $\frac{1}{2}e^{i\overline{\alpha}}$ and $\frac{1}{2}e^{i\overline{\beta}}$, implying that 
 $f_{(b_1,w)}(u)={\dc(\frac{u-\alpha}{2})\dc(\frac{u-\beta}{2})}$,
 \item $b_2w$ are $\frac{1}{2}e^{i(\overline{\beta}-\pi)}$ and $\frac{1}{2}e^{i\overline{\alpha}}$, implying that 
 $f_{(b_2,w)}(u)={-ik'}{\nc(\frac{u-\beta+2K}{2})\nc(\frac{u-\alpha}{2})}$,
 \item $b_3w$ are $\frac{1}{2}e^{i\overline{\beta}}$ and $\frac{1}{2}e^{i(\overline{\beta}+\pi)}$, implying that 
 $f_{(b_3,w)}(u)={ie^{i\overline{\theta}}}{\dc(\frac{u-\beta}{2})\dc(\frac{u-\beta-2K}{2})}$.
\end{itemize}

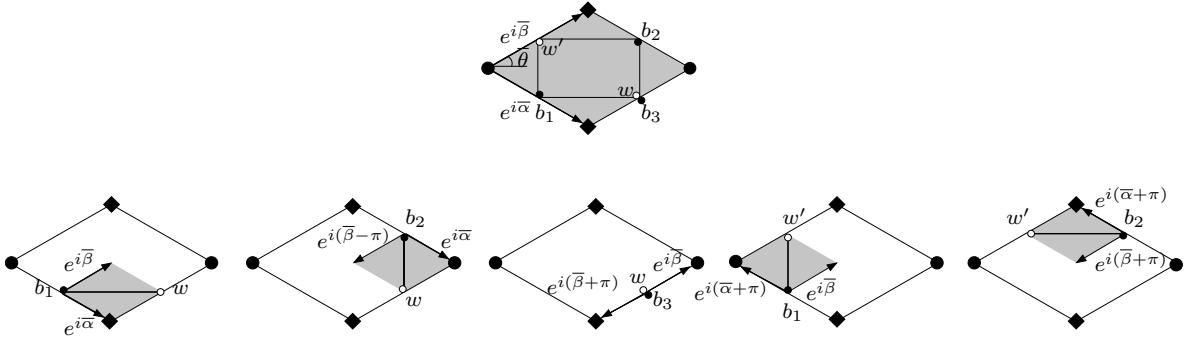
\begin{figure}[ht!]
  \centering
\begin{overpic}[width=\linewidth]{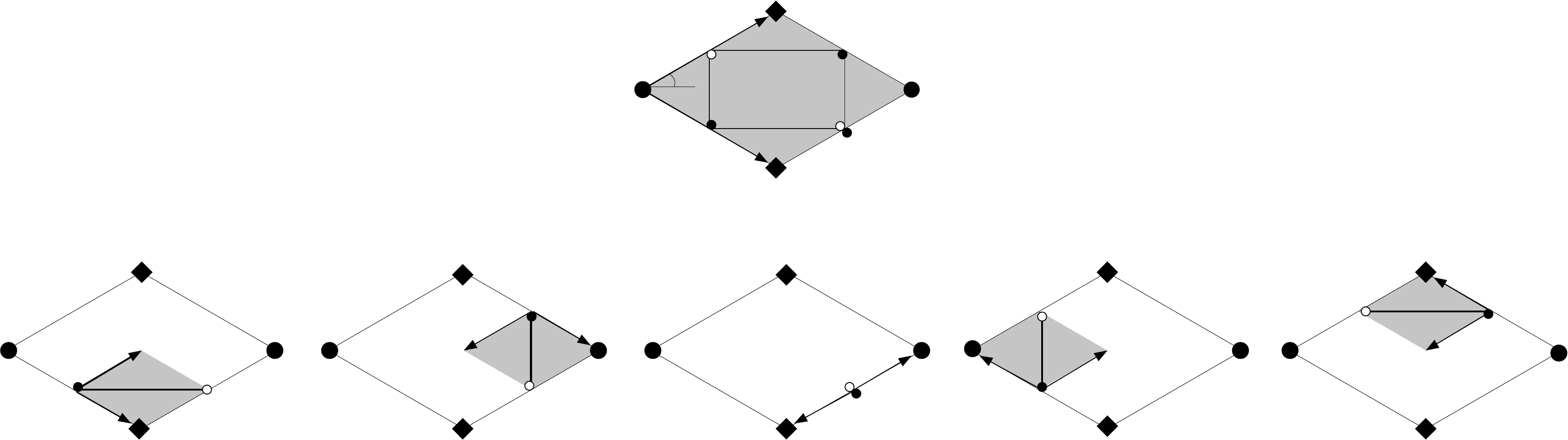}
\put(43.5,22.5){\scriptsize $\overline{\theta}$}

\put(45,18){\scriptsize $b_1$}
\put(54,25){\scriptsize $b_2$}
\put(54,18){\scriptsize $b_3$}
\put(52,20){\scriptsize $w$}
\put(45.5,23.5){\scriptsize $w'$}
\put(42,18){\scriptsize $e^{i\overline{\alpha}}$}
\put(42,24.5){\scriptsize $e^{i\overline{\beta}}$}

\put(2.5,3){\scriptsize $b_1$}
\put(14.2,3){\scriptsize $w$}
\put(5,-0.5){\scriptsize $e^{i\overline{\alpha}}$}
\put(5,5){\scriptsize $e^{i\overline{\beta}}$}

\put(34,9){\scriptsize $b_2$}
\put(34,2){\scriptsize $w$}
\put(37,7){\scriptsize $e^{i\overline{\alpha}}$}
\put(26.5,7){\scriptsize $e^{i(\overline{\beta}-\pi)}$}

\put(55,2){\scriptsize $b_3$}
\put(53,4){\scriptsize $w$}
\put(46,3){\scriptsize $e^{i(\overline{\beta}+\pi)}$}
\put(55,5){\scriptsize $e^{i\overline{\beta}}$}

\put(66,9){\scriptsize $w'$}
\put(66,1){\scriptsize $b_1$}
\put(58.5,2.5){\scriptsize  $e^{i(\overline{\alpha}+\pi)}$}
\put(68,2.5){\scriptsize $e^{i\overline{\beta}}$}

\put(85,9){\scriptsize $w'$}
\put(95,9){\scriptsize $b_2$}
\put(92.5,10.5){\scriptsize $e^{i(\overline{\alpha}+\pi)}$}
\put(92.5,5){\scriptsize $e^{i(\overline{\beta}+\pi)}$}
\end{overpic}
\caption{Computation of $f_{(b,w)}(u)$ for $b\in\{b_1,b_2,b_3\}$, and of $f_{(b,w')}(u)$ for $b\in\{b_1,b_2\}$.
To simplify the picture, the factor $\frac{1}{2}$ is omitted in the notation of the rhombus vectors in the bottom part of the picture.
\label{Fig:example_f}}
\end{figure}

We also compute $f_{(b,w')}(u)$ for $b\in\{b_1,b_2\}$, where $w'$ is the white vertex facing $w$ along the diagonal of the rectangle,
see Figure~\ref{Fig:example_f}. Then, the two rhombus vectors of $(\GQ)^{\diamond}$ of the edge:
\begin{itemize}
 \item $b_1 w'$ are $\frac{1}{2}e^{i \overline{\beta}}$ and $\frac{1}{2}e^{i(\overline{\alpha}+\pi)}$, implying that 
 $f_{(b_1,w')}(u)={-ik'}{\nc(\frac{u-\beta}{2})\nc(\frac{u-\alpha-2K}{2})}$,
 \item $b_2 w'$ are $\frac{1}{2}e^{i(\overline{\alpha}+\pi)}$ and $\frac{1}{2}e^{i (\overline{\beta}+\pi)}$, implying that 
 $f_{(b_2,w')}(u)={\dc(\frac{u-\alpha-2K}{2})\dc(\frac{u-\beta-2K}{2})}$.
\end{itemize}
\end{exm}

\begin{lem}
\label{lem:fs_well}
The function $f$ is well defined, that is independent of the choice of path from $b$ to $w$.
\end{lem}
\begin{proof}
It suffices to check that when traveling around each face of the graph, the product of the contributions of the edges is $1$.
There are three types of faces to consider: rectangles which correspond to edges of the graph $\Gs$ (or $\Gs^*$), faces
corresponding to those of the graph $\Gs$, and faces corresponding to those of the dual graph $\Gs^*$. 

Let us first check that this is true for rectangles, using the notation and computations of Example~\ref{example:1}.
Recalling that $f_{(w,b)}(u)=f_{(b,w)}(u)^{-1}$, we have
for a rectangle $b_1,w,b_2,w'$, 
\begin{align*}
&\textstyle f_{(b_1,w)}(u)f_{(w,b_2)}(u)f_{(b_2,w')}(u)f_{(w',b_1)}(u)=\\
&\textstyle =\dc(\frac{u-\alpha}{2})\dc(\frac{u-\beta}{2})
\dfrac{\cn(\frac{u-\beta+2K}{2})\cn(\frac{u-\alpha}{2})}{-ik'}
\dc(\frac{u-\alpha-2K}{2})\dc(\frac{u-\beta-2K}{2})
\dfrac{\cn(\frac{u-\beta}{2})\cn(\frac{u-\alpha-2K}{2})}{-ik'}\\
&\textstyle =-\cn(\frac{u-\beta+2K}{2})\nc(\frac{u-\beta-2K}{2})\dfrac{\dn(\frac{u-\alpha}{2})\dn(\frac{u-\beta}{2})\dn(\frac{u-\alpha-2K}{2})\dn(\frac{u-\beta-2K}{2})}{(k')^2}\\
&\textstyle =\dfrac{\dn(\frac{u-\alpha}{2})\dn(\frac{u-\beta}{2})\dn(\frac{u-\alpha-2K}{2})\dn(\frac{u-\beta-2K}{2})}{(k')^2},\
\text{ since $\cn(u+K)=-\cn(u-K)$,}\\
&\textstyle =1,\ \text{ because $\dn(u-K)\dn u=k'$ by Table~\ref{table:identities_Jacobi_function}}.
\end{align*}
We deduce that the function $f$ is well defined around rectangles. 
We now turn to faces which are not rectangles, and do some preliminary 
computations. Thanks to Example \ref{example:1} again, we have (see Table~\ref{table:identities_Jacobi_function} for the various simplifications involving Jacobi elliptic functions in \eqref{equ:biendef1} and \eqref{equ:biendef2})
\begin{align}
f_{(b_3,w)}(u)f_{(w,b_1)}(u)&
\textstyle=ie^{i\overline{\theta}}\dc(\frac{u-\beta}{2})\dc(\frac{u-\beta-2K}{2})
\textstyle\cd(\frac{u-\alpha}{2})\cd(\frac{u-\beta}{2})\nonumber \\
&\textstyle =ie^{i\overline{\theta}}\cd(\frac{u-\alpha}{2})\dc(\frac{u-\beta-2K}{2})
\textstyle{=}-ie^{i\overline{\theta}}\sn(\frac{u-(\alpha+2K)}{2})\ns(\frac{u-\beta}{2}).\label{equ:biendef1}
\end{align}
We also have
\begin{align}
f_{(b_2,w)}(u)f_{(w,b_3)}(u)&
\textstyle=-ik'\nc(\frac{u-\beta+2K}{2})\nc(\frac{u-\alpha}{2})
\dfrac{\cd(\frac{u-\beta}{2})\cd(\frac{u-\beta-2K}{2})}{ie^{i\overline{\theta}}} \nonumber \\ 
&\textstyle =\dfrac{\cn(\frac{u-\beta}{2})\nc(\frac{u-\alpha}{2})}{e^{i\overline{\theta}}}
=e^{-i\overline{\theta}}  \sd(\frac{u-\beta-2K}{2}) \ds(\frac{u-\alpha-2K}{2}).\label{equ:biendef2}
\end{align}
We have expressed the product $f_{(b_3,w)}f_{(w,b_1)}$ (resp.\ $f_{(b_2,w)}f_{(w,b_3)}$)
using the rhombus vectors rooted at the vertex of the dual graph $\Gs^*$ (resp.\ at the vertex of the primal graph $\Gs$), because 
this is what is needed to handle the product of local factors around faces of $\GQ$ corresponding to those of the graph $\Gs$ or $\Gs^*$.

Indeed, consider a face of $\GQ$ corresponding to a face of degree $n$ of the dual graph $\Gs^*$. Denote by $b_1,w_1,b_2,\dots, w_n,b_n,w_n$ 
its vertices in counterclockwise order. For every pair of black vertices $b_j,b_{j+1}$ denote by $\frac{1}{2}e^{i\alpha_j},\frac{1}{2}e^{i\alpha_{j+1}}$ the 
rhombus vectors rooted at the dual vertex corresponding to the face, and by
$\overline{\theta}_j$ the rhombus half-angle, see Figure~\ref{Fig:preuvef}
(left).

\begin{figure}[ht!]
  \centering
\begin{overpic}[width=13cm]{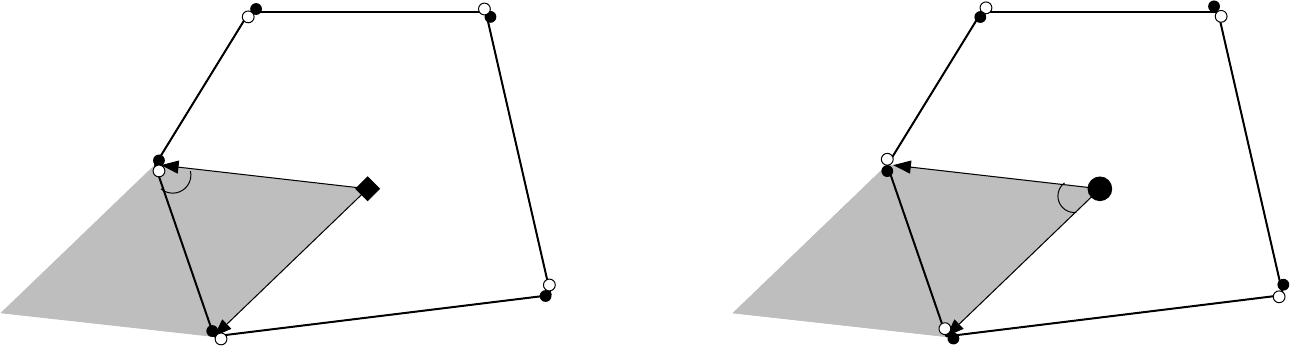}
\put(39,25){\scriptsize $b_1$}
\put(35,27){\scriptsize $w_1$}

\put(96,25){\scriptsize $w_1$}
\put(91,27){\scriptsize $b_1$}

\put(18,27){\scriptsize $b_2$}
\put(16,26){\scriptsize $w_2$}
\put(73,27){\scriptsize $w_2$}
\put(71,26){\scriptsize $b_2$}

\put(10,15){\scriptsize $b_j$}
\put(9,13){\scriptsize $w_j$}
\put(65,15){\scriptsize $w_j$}
\put(66,13){\scriptsize $b_j$}

\put(11,2){\scriptsize $b_{j+1}$}
\put(18,-1){\scriptsize $w_{j+1}$}
\put(67,2){\scriptsize $w_{j+1}$}
\put(75,-1){\scriptsize $b_{j+1}$}

\put(44,4){\scriptsize $w_n$}
\put(43,1.5){\scriptsize $b_n$}
\put(101,4){\scriptsize $w_n$}
\put(100,1.5){\scriptsize $b_n$}

\put(18,14){\scriptsize $e^{i\overline{\alpha}_j}$}
\put(24,6){\scriptsize $e^{i\overline{\alpha}_{j+1}}$}
\put(75,14){\scriptsize $e^{i\overline{\alpha}_j}$}
\put(81,6){\scriptsize $e^{i\overline{\alpha}_{j+1}}$}

\put(15,10){\scriptsize $\overline{\theta}_j$}
\put(78,10){\scriptsize $2\overline{\theta}_j$}
\end{overpic}
\caption{Faces around dual (left) and primal (right) vertices.
\label{Fig:preuvef}}
\end{figure}
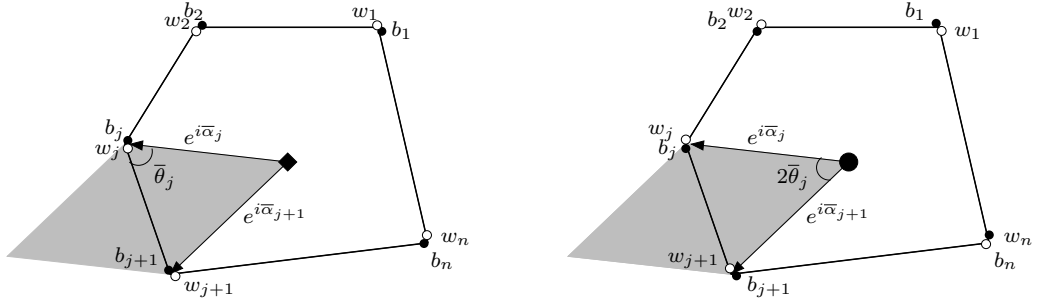

Then, by \eqref{equ:biendef1} we have $f_{(b_j,w_{j})}(u)f_{(w_{j},b_{j+1})}(u)=
(-i)e^{i\overline{\theta}_j}\sn(\frac{u-\alpha_{j+1}}{2})\ns(\frac{u-{\alpha_j}}{2})$. Moreover, 
\begin{equation*}
 (-i)e^{i\overline{\theta}_j}=e^{-i(\pi/2- \overline{\theta}_j)}=e^{-\frac{i}{2}(\overline{\alpha}_{j+1}-\overline{\alpha}_{j})},
\end{equation*}
see Figure \ref{Fig:preuvef} (left), implying that
\begin{equation*}
\textstyle f_{(b_j,w_{j})}(u)f_{(w_{j},b_{j+1})}(u)= 
e^{-\frac{i}{2}(\overline{\alpha}_{j+1}-\overline{\alpha}_{j})}\sn(\frac{u-\alpha_{j+1}}{2})\ns(\frac{u-{\alpha_j}}{2}).
\end{equation*}
As a consequence, for every $k<\ell$ (with cyclic notation for indices), we have
\begin{equation}
\label{eq:teles-1}
     \prod_{j=k}^{\ell-1} f_{(b_j,w_{j})}(u)f_{(w_{j},b_{j+1})}(u)=\textstyle
     e^{-\frac{i}{2}(\overline{\alpha}_{\ell}-\overline{\alpha}_j)}\sn(\frac{u-\alpha_\ell}{2})\ns(\frac{u-{\alpha_j}}{2}).
\end{equation}
It is important to notice that the right-hand side of \eqref{eq:teles-1} is independent of the determination of the angles $\overline{\alpha}_{\ell}$ and $\overline{\alpha}_j$.

In particular, the product around the face is (with $\alpha_{n+1}=\alpha_{1}+4K$)
\begin{equation*}
 \textstyle  e^{-i\pi}\sn(\frac{u-\alpha_{n+1}}{2})\ns(\frac{u-\alpha_{1}}{2})=
  -\sn(\frac{u-\alpha_{1}}{2}-2K)\ns(\frac{u-\alpha_{1}}{2})=1.
\end{equation*}

Consider now a face of degree $n$ of the graph $\GQ$ corresponding to a face of the graph $\Gs$. Using similar notation,
the picture differs in that the vertex at the center of the face belongs to $\Gs$, that black and white vertices are exchanged
and that the angle $2\overline{\theta}_j$ is
at the center of the face, see Figure~\ref{Fig:preuvef} (right). By~\eqref{equ:biendef2}, we have  
\begin{equation*}
f_{(b_j,w_{j})}(u)f_{(w_{j},b_{j+1})}(u)=\textstyle
e^{-i\overline{\theta}_j}\sd(\frac{u-\alpha_{j+1}}{2})\ds(\frac{u-{\alpha_j}}{2})=
e^{-\frac{i}{2}(\overline{\alpha}_{j+1}-\overline{\alpha}_{j})}
\sd(\frac{u-\alpha_{j+1}}{2})\ds(\frac{u-{\alpha_j}}{2}),
\end{equation*}
since we have $\overline{\theta}_j=\frac{1}{2}(\overline{\alpha}_{j+1}
-\overline{\alpha}_{j})$, see Figure \ref{Fig:preuvef} (right). As a consequence,
for every $k<\ell$ (with cyclic notation for indices), we have
\begin{equation}
\label{eq:teles-2}
\prod_{j=k}^{\ell-1} f_{(b_j,w_{j})}(u)f_{(w_{j},b_{j+1})}(u)=\textstyle
e^{-\frac{i}{2}(\overline{\alpha}_{\ell}-\overline{\alpha}_j)}
\sd(\frac{u-\alpha_\ell}{2})\ds(\frac{u-{\alpha_j}}{2}).
\end{equation}
In particular, the product around the face is (with $\alpha_{n+1}=\alpha_1+4K$)
\begin{equation*}
\textstyle
  e^{-i\pi}\sd(\frac{u-\alpha_{n+1}}{2})\ds(\frac{u-\alpha_{1}}{2})=
  -\sd(\frac{u-\alpha_{1}}{2}-2K)\ds(\frac{u-\alpha_{1}}{2})=1.\qedhere
\end{equation*}
\end{proof}

Note that Equations \eqref{eq:teles-1} and \eqref{eq:teles-2} are used again in
the proof of Lemma \ref{lem:simp_exp_1}, which proves an alternative expression
for the function $f$.

Next is the key proposition used in proving the local expression for an inverse of the Kasteleyn operator $\KQ$.

\begin{prop}
\label{prop:function_f_kernel_Kasteleyn}
Fixing a white base vertex $w_0$ of $\WQ$, for every $u\in\TT(k)$, the function
$f_{(\cdot,w_0)}(u)$, seen as a function on $\BQ$, is in the kernel of
the Kasteleyn operator~$\KQ$ of the bipartite graph $\GQ$.
\end{prop}

\begin{proof}
As we shall see, Proposition \ref{prop:function_f_kernel_Kasteleyn} follows
from the identity
\begin{equation}
\label{eq:(iii)32}
     \sn(u+v)\cn u-\cn(u+v)\dn v\sn u-\dn u\sn v=0, 
\end{equation}
which can be found in (iii) of Exercise 32 in~\cite[Chapter 2]{La89}.

By Equation~\eqref{equ:MassiveDimer}, we need to prove that, for every white
vertex $w$ with neighbors $b_1,b_2,b_3$ as in Figure~\ref{Fig:Kweights},
and every white base vertex $w_0$, we have:
\begin{equation*}
     \sn\theta\,f_{(b_1,w_0)}(u)
     +i \cn\theta\,f_{(b_2,w_0)}(u)
     -i e^{-i\overline{\theta}} f_{(b_3,w_0)}(u)=0.
\end{equation*}
Since the function $f$ is defined inductively on the edges of $\GQ$, it suffices to prove:
\begin{equation}
     \sn\theta\,f_{(b_1,w)}(u)+i\cn\theta\,f_{(b_2,w)}(u)-ie^{-i\overline{\theta}}f_{(b_3,w)}(u)=0.
     \label{eq:f_fund_rel}
\end{equation}
Using the computations of Example~\ref{example:1}, this reduces to showing:
\begin{multline}
\textstyle
     \sn\theta \dc(\frac{u-\alpha}{2})\dc(\frac{u-\beta}{2})+i\cn\theta (-ik')\nc(\frac{u-\beta+2K}{2})\nc(\frac{u-\alpha}{2})\\
     \textstyle-ie^{-i\overline{\theta}}{ie^{i\overline{\theta}}}{\dc(\frac{u-\beta}{2})\dc(\frac{u-\beta-2K}{2})}=0.
     \label{eq:f_fund_rel2}
\end{multline}
Using some identities from Table~\ref{table:identities_Jacobi_function}, this is
equivalent to proving
\begin{equation*}
\textstyle
\sn\theta \ns(\frac{u-\alpha-2K}{2})\ns(\frac{u-\beta-2K}{2})+\cn\theta \nc(\frac{u-\beta-2K}{2})\ds(\frac{u-\alpha-2K}{2})
\textstyle-\ns(\frac{u-\beta-2K}{2})\dc(\frac{u-\beta-2K}{2})=0.
\end{equation*}
Multiplying by $\sn(\frac{u-\alpha-2K}{2})\sn(\frac{u-\beta-2K}{2})\cn(\frac{u-\beta-2K}{2})$ and using that $\cn$ and $\dn$ are even functions and that $\sn$ is an odd function, this amounts to proving:
\begin{multline*}
\textstyle
\sn\theta \cn(-\frac{u-\beta-2K}{2})-\cn\theta \dn(\frac{u-\alpha-2K}{2})\sn(-\frac{u-\beta-2K}{2})
\textstyle-\dn(-\frac{u-\beta-2K}{2})\sn(\frac{u-\alpha-2K}{2})=0.
\end{multline*}
As announced, this is exactly \eqref{eq:(iii)32} with $u=-\frac{u-\beta-2K}{2}$, $v=\frac{u-\alpha-2K}{2}$ and $u+v=\theta$. 
\end{proof}

\subsection{Local expression for the inverse of the Kasteleyn operator $\KQ$}
\label{sec:invKQ}

We now state Theorem~\ref{thm:Kmoins_un}, proving an explicit, local formula for
an inverse ${\KQ}^{-1}$ of the Kasteleyn matrix $\KQ$, constructed from the function
$f$ defined in~\eqref{eq:def_function_f}.

\begin{thm}
\label{thm:Kmoins_un}
Define the infinite matrix ${\KQ}^{-1}$ whose coefficients are given,
for any $(b,w)\in\BQ\times\WQ$, by
\begin{equation}
  \label{eq:definition_C_b,w}
    {\KQ}^{-1}_{b,w}=\frac{1}{4i\pi} \int_{\Gamma_{b,w}} f_{(b,w)}(u) \ud u,
  \end{equation}
  where $\Gamma_{b,w}$ is a vertical contour directed upwards on $\mathbb T(k)$,
  crossing the real axis outside of the sector of size $2K$ containing all the
  poles of $f_{(b,w)}$.

  Then ${\KQ}^{-1}$ is an inverse operator of $\KQ$. For $k\neq 0$, it is the
  only inverse with bounded coefficients.

The quantity ${\KQ}^{-1}_{b,w}$ in \eqref{eq:definition_C_b,w} can alternatively be expressed as
\begin{equation}
\label{eq:fsbis}
    {\KQ}^{-1}_{b,w}=\frac{1}{4i\pi} \oint_{\C_{b,w}}  f_{(b,w)}(u) \Hh(u)\ud u,
\end{equation}
where $\Hh$ is related to Jacobi's zeta function and is defined in~\eqref{eq:definition_Hh_Hv_k2>0}--\eqref{eq:definition_Hh_Hv_k2<0},
$\C_{b,w}$ is a trivial contour on the torus, not crossing
$\Gamma_{b,w}$
and containing in its interior all the poles of $f_{(b,w)}$ and the pole of $\Hh$. 
\end{thm}

\begin{proof}
To show that the two expressions \eqref{eq:definition_C_b,w} and
\eqref{eq:fsbis} indeed coincide, we use the same argument as in the proof of Theorem \ref{thm:KFmoins_un}.

The structure of the proof of Theorem \ref{thm:Kmoins_un} is analogous to that of \cite[Theorem~1]{BdTR1}. 
Instead of using the form~\eqref{eq:definition_C_b,w} as in the proof of Theorem~\ref{thm:KFmoins_un},
we use the alternative expression \eqref{eq:fsbis} of $\KQ^{-1}_{b,w}$. Indeed,
from the computations done below to prove that
$(\KQ\KQ^{-1})_{w,w}=1$, one can extract as a by-product the explicit probability of a
given edge to be present in the random dimer configuration in the corresponding
dimer model.

Let $w$ be a white vertex of $\GQ$ and $b_1$, $b_2$, $b_3$ be its three black
neighbors, as in Figure~\ref{Fig:Kweights}. Let $w'$ be another white vertex,
different from $w$. The contours $\C_{b_1,w'}$, $\C_{b_2,w'}$ and $\C_{b_3,w'}$
entering into the definition of ${\KQ}^{-1}_{b_1,w'}$, ${\KQ}^{-1}_{b_2,w'}$ and
${\KQ}^{-1}_{b_3,w'}$ can be deformed into a common contour $\C$ without crossing
any pole. Therefore, the
entry $(\KQ {\KQ}^{-1})_{w,w'}$ can be written as:
\begin{equation*}
  (\KQ {\KQ}^{-1})_{w,w'} = \sum_{i=1}^{3}\KQ_{w,b_i} {\KQ}^{-1}_{b_i,w'}=
  \frac{1}{4i\pi}\oint_\C \Hh(u) \left(\sum_{i=1}^3 \KQ_{w,b_i}
  f_{(b_i,w')}(u)\right)\ud u =0,
\end{equation*}
by Proposition~\ref{prop:function_f_kernel_Kasteleyn}.
  
We now need to compute the entry $(\KQ \KQ^{-1})_{w,w}$. This is done explicitly,
via the residue theorem. In addition to the simple pole at $u=2iK'$ coming
from the function~$\Hh$ (with residue $2K'/\pi$, see Lemma \ref{lem:properties_Hh_Hv}), there
are other (simple) poles located at the zeros of the functions in the
denominator of $f$, \emph{i.e.}, when the argument of the functions $\cd$ and
$\cn$ is equal to $K$.

We shall successively compute $\KQ^{-1}_{b_1,w}$, $\KQ^{-1}_{b_2,w}$ and
$\KQ^{-1}_{b_3,w}$, using the different values of $f$ listed in
Example~\ref{example:1}. First, $\KQ^{-1}_{b_1,w}$ is obtained from (the
minus signs in the numerators in the right-hand side below come from the
expansion of $\cd$ around $K$, see \cite[Table 16.7]{AS})
\begin{multline*}
  2i\pi\{\res_{u=2K+\alpha}+\res_{u=2K+\beta}+\res_{u=2iK'}\}
  \left(\frac{1}{4i\pi}\Hh(u)f_{(b_1,w)}(u)\right) = \\
  \frac{-\Hh(2K+\alpha)}{\cd(K-\theta)}
  +
  \frac{-\Hh(2K+\beta)}{\cd(K+\theta)}
  +
  \frac{K'}{\pi}\frac{1}{\cd(iK'-\alpha/2)\cd(iK'-\beta/2)}.
\end{multline*}
Similarly, for the computation of $\KQ^{-1}_{b_2,w}$ we have:
\begin{multline*}
  2i\pi\{\res_{u=2K+\alpha}+\res_{u=\beta}+\res_{u=2iK'}\}
  \left(\frac{1}{4i\pi}\Hh(u)f_{(b_2,w)}(u) \right)= \\
  \frac{ \Hh(2K+\alpha)(-i{k'})}{-k'\cn(2K-\theta)}
  +
  \frac{ \Hh(\beta)(-i{k'})}{-k'\cn\theta }
  +
  \frac{K'}{\pi}\frac{(-i{k'})}{\cn(iK'-\alpha/2)\cn(K+iK'-\beta/2)}.
\end{multline*}
Finally, we obtain for $\KQ^{-1}_{b_3,w}$:
\begin{multline*}
  2i\pi\{\res_{u=2K+\beta}+\res_{u=4K+\beta}+\res_{u=2iK'}\}
  \left(\frac{1}{4i\pi}\Hh(u)f_{(b_3,w)}(u)\right) =\\
  \frac{{-} {ie^{i\overline{\theta}}}\Hh(2K+\beta)}{\cd0}
  +
  \frac{{-} {ie^{i\overline{\theta}}}\Hh(4K+\beta)}{\cd(2K)}
  +
  \frac{K'}{\pi}\frac{{ie^{i\overline{\theta}}}}{\cd(iK'-\beta/2)\cd(-K+iK'-\beta/2)}.
\end{multline*}
Multiplying these equations by the corresponding entries of the Kasteleyn matrix
(namely, $\sn\theta $, $i\cn\theta $ and $-ie^{-i\overline{\theta}}$, see
\eqref{equ:MassiveDimer}) and
summing them, one can group together terms having similar values of $\Hh$.

Terms with a $\Hh(2K+\alpha)$ give:
\begin{equation*}
  \Hh(2K+\alpha)\left(
  \frac{\sn\theta }{-\cd(K-\theta)} + \frac{\cn\theta }{\cn(2K-\theta)}\right)=
  \Hh(2K+\alpha)\left(
  \frac{\sn\theta }{-\sn\theta } + \frac{\cn\theta }{\cn\theta } \right)=
  0.
\end{equation*}
Similarly, those with a $\Hh(2K+\beta)$ give:
\begin{equation*}
  \Hh(2K+\beta)\left(
  \frac{\sn\theta }{-\cd(K+\theta)} + \frac{-ie^{-i\overline{\theta}}(-ie^{i\overline{\theta}})}{\cd(0)}\right) =
  \Hh(2K+\beta)\left(
  \frac{-\sn\theta }{-\sn\theta } -1\right)= 0.
\end{equation*}

We group the terms in $\Hh(\beta)$ and $\Hh(4K+\beta)$, and use the fact that
$\Hh(4K+\beta)=\Hh(\beta)+1$, stated in Lemma \ref{lem:properties_Hh_Hv}:
\begin{equation*}
  \Hh(\beta)\left(
  \frac{-ik'i\cn\theta }{-k'\cn\theta }\right)
  + \Hh(4K+\beta)\left(\frac{-ie^{-i\overline{\theta}}(-ie^{i\overline{\theta}})}{\cd(2K)}\right) =
  \Hh(\beta)\left( -1+1\right)
  + 1 =1.
\end{equation*}

We are left with computing the sum of residues at $2iK'$. It turns out that this 
boils down to~\eqref{eq:f_fund_rel2} for $u=2iK'$. Thus this sum equals $0$. 
Therefore
\begin{equation*}
  \sum_{i=1}^3 \KQ_{w,b_i}\KQ^{-1}_{b_i,w} = 1,
\end{equation*}
thereby completing the proof of Theorem \ref{thm:Kmoins_un}.
\end{proof}

\subsection{Asymptotics of the inverse Kasteleyn operator}
\label{sec:asymptKQ}

We first need to introduce some notation.
For any $b$ and $w$, there exists a path on the diamond graph
$\GR$ joining $b$ and $w$, see Figure \ref{fig:detailed_graphs}. The first and
the last edges of the path are half-edges of $\GR$, the other ones are plain
edges. We call $b=b_1$, and $b_n$ the black vertex adjacent to $w$. We further define
$b_2,\ldots ,b_{n-1}$ as the successive black vertices in the middle of the edges
of $\GR$ joining $b$ to $w$, see again Figure \ref{fig:detailed_graphs}. The $n$
edges are equal to 
\begin{equation*}
\textstyle
     \frac{1}{2}e^{i\overline{\alpha}_1},e^{i\overline{\alpha}_2},\ldots,e^{i\overline{\alpha}_{n-1}},
     \frac{1}{2}e^{i\overline{\alpha}_n}.
\end{equation*}
The $\overline{\alpha}_i$ are not well defined (in the sense that any multiple of
$2\pi$ could be added to $\overline{\alpha}_i$), but the $e^{\overline{\alpha}_i}$
are. The edges are orientated in such a way that 
\begin{equation}
\label{eq:orientation_edges}
     b_1+\frac{1}{2}e^{i\overline{\alpha}_1}+\sum_{j=2}^{n-1}e^{i\overline{\alpha}_j}+\frac{1}{2}e^{i\overline{\alpha}_n}=b_n.
\end{equation}
We also define the points $a_{j}$ ($j=1,\ldots,n-1$) as the vertices of the
diamond graph lying between $b_j$ and $b_{j+1}$. The notation of this paragraph is
illustrated on Figure~\ref{fig:detailed_graphs}.

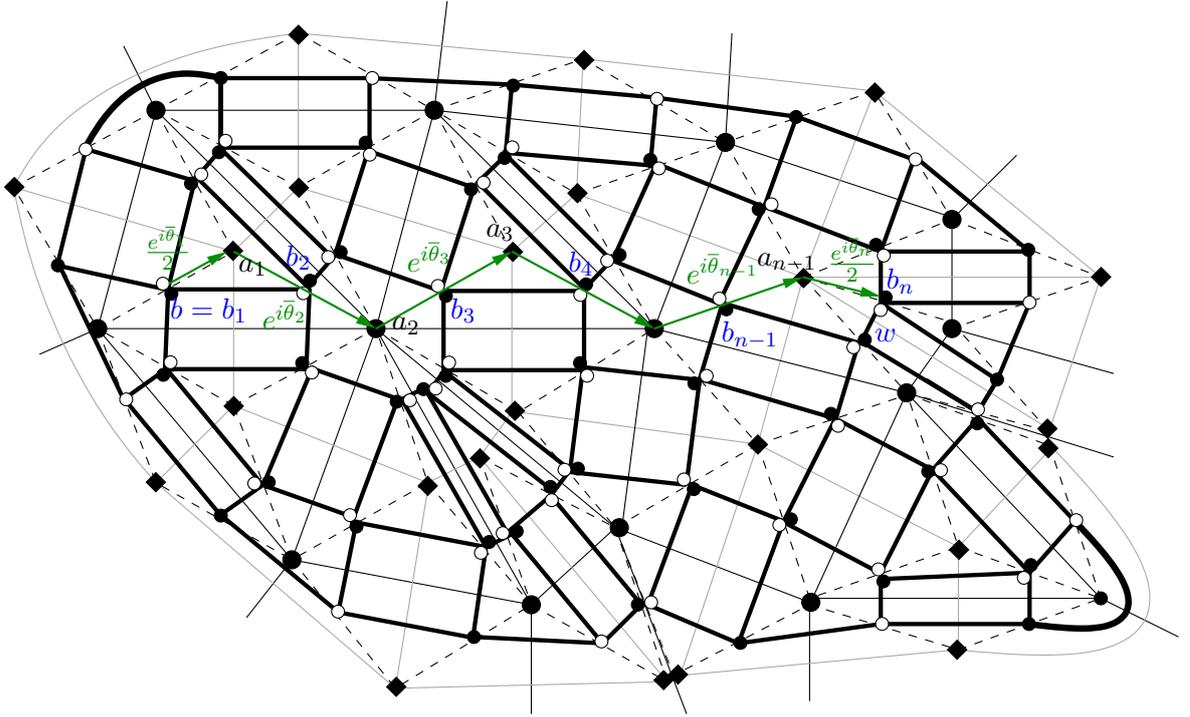
\begin{figure}[ht]
  \centering
\begin{overpic}[width=\textwidth]{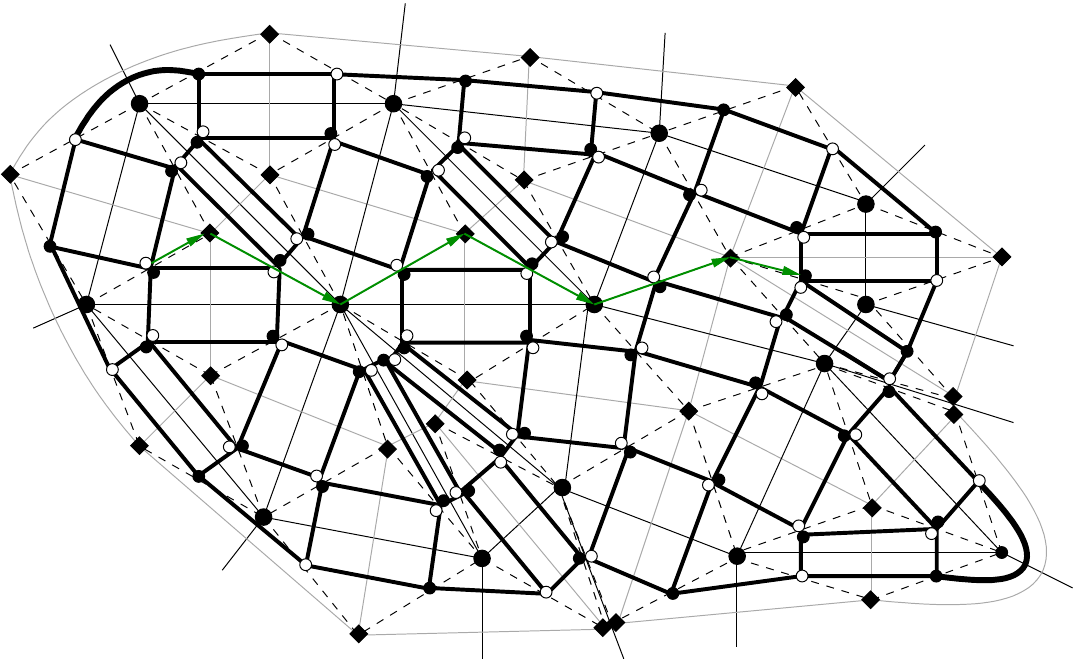}
  \put(14.2,34){\textcolor{blue}{$b=b_1$}}
  \put(24,38.5){\textcolor{blue}{$b_2$}}
  \put(38,34){\textcolor{blue}{$b_3$}}
  \put(48,38){\textcolor{blue}{$b_4$}}
  \put(61,32){\textcolor{blue}{$b_{n-1}$}}
  \put(75,36.5){\textcolor{blue}{$b_n$}}
  \put(74,32){\textcolor{blue}{$w$}}

  \put(20,38){$a_1$}
  \put(33,33){$a_2$}
  \put(41,41){$a_3$}
  \put(64,38.5){$a_{n-1}$}

  \put(12,39){\textcolor{darkgreen}{$\frac{e^{i\overline{\theta}_1}}{2}$}}
  \put(70,38){\textcolor{darkgreen}{$\frac{e^{i\overline{\theta}_n}}{2}$}}
  \put(22,33){\textcolor{darkgreen}{$e^{i\overline{\theta}_2}$}}
  \put(34.3,38){\textcolor{darkgreen}{$e^{i\overline{\theta}_3}$}}
  \put(58,37){\textcolor{darkgreen}{$e^{i\overline{\theta}_{n-1}}$}}
\end{overpic}
\caption{Notation for a path on the quad-graph
$\GR$ joining $b$ and $w$, see \eqref{eq:orientation_edges}.}
\label{fig:detailed_graphs}
\end{figure}

Finally, let $h$ as in \eqref{eq:def_function_h} and define
\begin{equation}
\label{eq:def_chi}
     \chi(u) = \frac{1}{\vert a_1-a_{n-1}\vert}\log\{ \expo_{(a_1,a_{n-1})}(u+2iK')\}.
\end{equation}
\begin{thm}
\label{thm:asymp_C}
Let $\Gs$ be a quasicrystalline isoradial graph. When the distance $\vert b-w\vert\rightarrow\infty$, we have 
\begin{equation*}
     \KQ^{-1}_{b,w}= \frac{e^{i\overline{\theta}}e^{-\frac{i}{2}(\overline{\alpha}_n-\overline{\alpha}_1)}h(u_0\pm2iK')}{2\sqrt{2\pi \vert a_1-a_{n-1}\vert \chi''(u_0)}} e^{\vert a_1-a_{n-1}\vert \chi(u_0)}\cdot (1+o(1)),
\end{equation*}
where $\overline{\theta}$ is the rhombus-angle of the rhombus to which $w$ belongs, $u_0$ is the unique $u\in(-K,K)$ such that $\chi'(u)=0$, and $\chi(u)< 0$. 
\end{thm}

Theorem \ref{thm:asymp_C} should be compared to its genus $0$ counterpart, \emph{i.e.}, Theorem~4.3 of \cite{Kenyon3} where polynomial decrease of 
coefficients of the inverse Kasteleyn operator is proved. 

\begin{rem}
Contrary to Theorem \ref{thm:asymptotics_inverse_Kasteleyn}, where we have shown that the constant in front of the exponential function is always positive (see Remark \ref{rem:sign_asymptotics_inverse_Kasteleyn}), we have less control on the constant in Theorem \ref{thm:asymp_C}. First, it can have a phase, due to the terms $e^{i\overline{\theta}}$ and $e^{-\frac{i}{2}(\overline{\alpha}_n-\overline{\alpha}_1)}$. The main point is that the quantity $h(u_0\pm2iK')$, which is real, can be positive, negative and even $0$. This follows from \eqref{eq:def_function_h}.
\end{rem}

Note that the proof of Theorem~\ref{thm:asymp_C} is similar to that of 
Theorem~\ref{thm:asymptotics_inverse_Kasteleyn}; therefore, we omit the details. We only need to give an expression for the function $h$ appearing in the statement of Theorem \ref{thm:asymp_C}. First of all, we prove that the function $f_{(b_1,b_n)}$ looks very
much like the exponential function $\expo_{(a_1,a_{n-1})}$ (we use the previous notation).

\begin{lem}
\label{lem:simp_exp_1}
The following formula holds:
\begin{equation}
\label{eq:simp_exp_f_1}
     f_{(b_1,b_n)}(u)=ie^{-\frac{i}{2}(\overline{\alpha}_n-\overline{\alpha}_1)}g(u)\expo_{(a_1,a_{n-1})}(u),
\end{equation}
where
\begin{equation*}
     g(u)=\left\{\begin{array}{ll}
     {\sn(\frac{u-{\alpha_n}}{2})}{\dc(\frac{u-{\alpha_1}}{2})} & \text{if $a_1$ and $a_{n-1}$ are dual},\\
     {\sd(\frac{u-{\alpha_n}}{2})}{\dc(\frac{u-{\alpha_1}}{2})} \cdot(-\sqrt{k'}) & \text{if $a_1$ is dual and $a_{n-1}$ primal},\\
     {\sd(\frac{u-{\alpha_n}}{2})}{\nc(\frac{u-{\alpha_1}}{2})}\cdot({k'})& \text{if $a_1$ and $a_{n-1}$ are primal},\\
     {\sn(\frac{u-{\alpha_n}}{2})}{\nc(\frac{u-{\alpha_1}}{2})}\cdot(-\sqrt{k'})& \text{if $a_1$ is primal and $a_{n-1}$ dual}.
     \end{array}\right.
\end{equation*}
\end{lem}

It is important to note that taken independently, the factors
$e^{-\frac{i}{2}(\overline{\alpha}_n-\overline{\alpha}_1)}$ and $g(u)$ are not
well defined: if ${\alpha}_n$ (or ${\alpha}_1$) is replaced by ${\alpha}_n+4K$, these
terms should be replaced by their opposite. However, the product
$e^{-\frac{i}{2}(\overline{\alpha}_n-\overline{\alpha}_1)}g(u)$ is well
defined, which suffices for our purpose.

\begin{proof}
With the previous notation we write
$f_{(b,b_{n})}(u)=\prod_{j=1}^{n-1}f_{(b_j,b_{j+1})}(u)$. There are two cases for
the computation of $f_{(b_j,b_{j+1})}$, according to whether $a_j$ is a primal
or a dual vertex: the identities \eqref{eq:teles-1} and
\eqref{eq:teles-2} yield
\begin{equation}
\label{eq:b_j,j+1-1}
     f_{(b_j,b_{j+1})}(u)=e^{-i (\overline{\alpha}_{j+1}-(\overline{\alpha_j\pm2K}))}\cdot \left\{\begin{array}{ll} 
     {\sd(\frac{u-\alpha_{j+1}}{2})}{\ds(\frac{u-\alpha_{j}\mp2K}{2})} & \text{if $a_j$ is primal},\\
    {\sn(\frac{u-\alpha_{j+1}}{2})}{\ns(\frac{u-\alpha_{j}\mp2K}{2})} & \text{if $a_j$ is dual}.
     \end{array}\right.     
\end{equation}
The term $\pm2K$ in \eqref{eq:b_j,j+1-1} comes from the fact that the
orientation of the rhombus vectors in \eqref{eq:teles-1}--\eqref{eq:teles-2} and in
\eqref{eq:orientation_edges} are reversed. The quantity
\eqref{eq:b_j,j+1-1} does not depend on the value of this sign (this is a
consequence of the addition formulas by $\pm K$ for the $\sn$ and $\sd$
functions, see Table~\ref{table:identities_Jacobi_function} in the appendix). 

Due to the fact that the value of $f_{(b_j,b_{j+1})}$ depends on the type
(primal or dual) of the vertex $a_j$, there are four cases for the computation
of $f_{(b,b_{n})}$, according to the types of $a_1$ and $a_{n-1}$. We write
down the computations in the particular case where both $a_1$ and $a_{n-1}$ are
dual, the other cases would follow in a very similar manner. We have (with all
signs $\pm = +$):
\begin{align*}
     f_{(b,b_{n})}(u)&\textstyle=e^{-\frac{i}{2} (\overline{\alpha}_{2}-(\overline{\alpha_1+2K}))}\times \cdots \times e^{-\frac{i}{2}  (\overline{\alpha}_{n}-(\overline{\alpha_{n-1}+2K}))}\dfrac{\sn(\frac{u-\alpha_{2}}{2})}{\sn(\frac{u-\alpha_{1}-2K}{2})}\times \cdots \times \dfrac{\sn(\frac{u-\alpha_{n}}{2})}{\sn(\frac{u-\alpha_{n-1}-2K}{2})}\\
     &\textstyle=e^{\frac{i}{2}(n-1)\pi}e^{-\frac{i}{2}(\overline{\alpha}_{n}-\overline{\alpha}_{1}) }\dfrac{\sn(\frac{u-\alpha_{n}}{2})}{\sn(\frac{u-{\alpha_1}}{2}-K)}\dfrac{\sn(\frac{u-\alpha_{2}}{2})}{\sd(\frac{u-\alpha_{2}}{2}-K)}\dfrac{\sd(\frac{u-\alpha_{3}}{2})}{\sn(\frac{u-\alpha_{3}}{2}-K)}\times \cdots \times \dfrac{\sd(\frac{u-\alpha_{n-1}}{2})}{\sn(\frac{u-\alpha_{n-1}}{2}-K)}\\
     &\textstyle=e^{\frac{i}{2}(n-1)\pi}e^{-\frac{i}{2}(\overline{\alpha}_{n}-\overline{\alpha}_{1}) }\dfrac{\sn(\frac{u-\alpha_{n}}{2})}{-\cd(\frac{u-{\alpha_1}}{2})}\dfrac{\sn(\frac{u-\alpha_{2}}{2})}{-k'^{-1}\cn(\frac{u-\alpha_{2}}{2})}\dfrac{\sd(\frac{u-\alpha_{3}}{2})}{-\cd(\frac{u-\alpha_{3}}{2})}\times \cdots \times \dfrac{\sd(\frac{u-\alpha_{n-1}}{2})}{-\cd(\frac{u-\alpha_{n-1}}{2})}\\
     &\textstyle=e^{\frac{i}{2}(n-1)\pi}e^{-\frac{i}{2}(\overline{\alpha}_{n}-\overline{\alpha}_{1}) }\dfrac{\sn(\frac{u-\alpha_{n}}{2})}{-\cd(\frac{u-{\alpha_1}}{2})}(-1)^{\frac{n-2}{2}}\prod_{j=2}^{n-1} i\sqrt{k'}\sc(\frac{u-{\alpha_j}}{2})\\
     &\textstyle=ie^{-\frac{i}{2}(\overline{\alpha}_{n}-\overline{\alpha}_{1}) }{\sn(\frac{u-\alpha_{n}}{2})}{\dc(\frac{u-{\alpha_1}}{2})}\expo_{(a_1,a_{n-1})}.
\end{align*}
The third equality above uses addition formulas of Table~\ref{table:identities_Jacobi_function}. In the last line, we have applied the definition
\eqref{eq:recursive_def_expo} of the exponential function. We also implicitly
used the fact that $n$ is even (because both $a_1$ and $a_{n-1}$ are dual
vertices). The first case of Equation \eqref{eq:simp_exp_f_1}, and thus of Lemma
\ref{lem:simp_exp_1}, is proved.
\end{proof}

We now give a formula for $f_{(b,w)}$ for general vertices $b$ and $w$.
\begin{lem}
\label{lem:simp_exp_2}
The following formula holds:
\begin{equation}
\label{eq:simp_exp_f_2}
     f_{(b,w)}(u)=e^{i\overline{\theta}}e^{-\frac{i}{2}(\overline{\alpha}_n-\overline{\alpha}_1)}h(u)\expo_{(a_1,a_{n-1})}(u),
\end{equation}
where $\overline{\theta}$ is the rhombus-angle of the rhombus to which $w$ belongs, and
\begin{equation}
\label{eq:def_function_h}
     h(u)=\left\{\begin{array}{ll}
     {\dc(\frac{u-{\alpha_1}}{2})\dc(\frac{u-{\alpha_n}}{2})}& \text{if $a_1$ and $a_{n-1}$ are dual},\\
     {\dc(\frac{u-{\alpha_1}}{2})\nc(\frac{u-{\alpha_n}}{2})} \cdot(\sqrt{k'}) & \text{if $a_1$ is dual and $a_{n-1}$ primal},\\
     {\nc(\frac{u-{\alpha_1}}{2})\nc(\frac{u-{\alpha_n}}{2})}\cdot(-{k'})& \text{if $a_1$ and $a_{n-1}$ are primal},\\
     {\nc(\frac{u-{\alpha_1}}{2})\dc(\frac{u-{\alpha_n}}{2})}\cdot(-\sqrt{k'})& \text{if $a_1$ is primal and $a_{n-1}$ dual}.
     \end{array}\right.
\end{equation}
\end{lem}

\begin{proof}
We have $f_{(b,w)}=f_{(b_1,b_n)}f_{(b_n,w)}$. The definition \eqref{eq:def_function_f} of the function $f$ provides
\begin{equation*}
     f_{(b_n,w)}(u)={ie^{i\overline{\theta}}}\cdot \left\{ \begin{array}{ll}
     {\dc(\frac{u-{\alpha_n}}{2})\dc(\frac{u-\alpha_{n}-2K}{2})}& \text{if $a_{n-1}$ is primal},\\
     {\dc(\frac{u-{\alpha_n}}{2})\dc(\frac{u-\alpha_{n}+2K}{2})}& \text{if $a_{n-1}$ is dual}.
     \end{array}\right.
\end{equation*}     
(The choice of the sign $\pm$ above comes from Remark \ref{rem:definition_angles}.) Using that 
$\cd(\frac{u-\alpha_{n}\mp2K}{2})=\mp \sn(\frac{u-{\alpha_n}}{2})$ together with Lemma \ref{lem:simp_exp_1} ends the proof of Lemma \ref{lem:simp_exp_2}.
\end{proof}

\subsection{Application to the dimer model on the graph $\GQ$}
\label{sec:dimerKQ}

In the same way as in Section~\ref{subsec:dimer_model_GF}, the inverse Kasteleyn operator $\KQ^{-1}$ can be used to obtain an explicit \emph{local} expression 
for a Gibbs measure $\PPdimer^{\Qs}$ on dimer configurations of the infinite graph $\GQ$ arising from two independent $Z$-invariant Ising models. It can also be used 
to obtain an explicit \emph{local} formula for the free energy of the model. By Dubédat~\cite{Dubedat} we know that this free energy is equal, up to an additive constant,
to that of the dimer model on $\GF$ (since the characteristic polynomials differ by a multiplicative constant), so that we feel it presents no real interest to derive the formula,
although it can be done using the approach of Theorem~\ref{thm:free_energy_dimer}.

For the Gibbs measure, everything works out in exactly the same way so that we do not write out the details. We obtain that the probability of occurrence of a subset of edges
$\E=\{w_1 b_1,\dots,w_k b_k\}$  in a dimer configuration of $\GQ$ is:
\begin{equation*}
\PPdimer^{\Qs}(w_1 b_1,\dots,w_k b_k)=\left(\prod_{j=1}^k \KQ_{w_j,b_j} \right)\det[(\KQ^{-1})_{\E}],
\end{equation*}
where $\KQ^{-1}$ is the inverse Kasteleyn operator whose coefficients are given by~\eqref{eq:definition_C_b,w} or~\eqref{eq:fsbis} and 
$(\KQ^{-1})_{\E}$ is the sub-matrix of $\KQ^{-1}$
whose rows are indexed by $b_1,\dots,b_k$ and columns by $w_1,\dots,w_k$.

\appendix
\section{Useful identities involving elliptic functions}

\label{app:elliptic}

In this section we list required identities satisfied by elliptic functions. 

\subsection{Identities for Jacobi's elliptic functions}

\paragraph{Change of argument.} 
Jacobi's elliptic functions satisfy various addition formulas by quarter-periods and half-periods, among which (cf.\ \cite[Table~16.8]{AS}):

\setlength{\doublerulesep}{\arrayrulewidth}
\begin{table}[ht]
\begin{center}
\begin{tabular}{|| c||c| c| c| c| c| c||}
  \hline\hline
    & $-u$ & $u\pm K$ & $u+2K$ & $u+iK'$ & $u+2iK'$  & $u+K+iK'$\\
  \hline 
  \hline
  $\sn$ &   $-\sn$ & $\pm\cd$ & $-\sn$ & $\frac{1}{k}\ns$ & $\sn$ & $\frac{1}{k}\dc$   \\
  \hline
  $\cn$ &   $\sn$ & $\mp k'\sd$ & $-\cn$ & $-\frac{i}{k}\ds$ & $-\cn$ & $-\frac{ik'}{k}\nc$   \\
  \hline
  $\dn$ &   $\dn$ & $k' \nd$ & $\dn$ & $-i\cs$ & $-\dn$ & $ik' \sc$   \\
  \hline
  $\cd$ &   $\cd$ & $\mp\sn$ & $-\cd$ & $\frac{1}{k}\dc$ & $\cd$ & $-\frac{1}{k}\ns$   \\
  \hline
  $\sc$ &   $-\sc$ & $-\frac{1}{k'}\cs$ & $\sc$ & $i\nd$ & $-\sc$ & $\frac{i}{k'}\dn$   \\
  \hline\hline
\end{tabular}
\end{center}
\caption{Addition formulas by quarter-periods and half-periods, taken from \cite[16.8]{AS}.}
\label{table:identities_Jacobi_function}
\end{table}

\subsection{Jacobi's epsilon, zeta and related functions}
\label{app:HhHv}

The explicit expressions of the inverse operators of
Theorems~\ref{thm:KFmoins_un} and~\ref{thm:Kmoins_un} involve the function~$\Hh$ defined as follows. In the remarks following Theorem \ref{thm:KFmoins_un}, we also use the function $\Hv$ mentioned below.

For $0<k^2<1$, let
\begin{equation}
\label{eq:definition_Hh_Hv_k2>0}
     \left\{\begin{array}{rl}
     \displaystyle\Hh(u\vert k)&\hspace{-2.5mm}=\displaystyle\frac{K'}{\pi}\Bigl\{\Erm\Bigl(\frac{u}{2}\Big\vert k\Bigr)+\frac{E'-K'}{K'}\frac{u}{2}\Bigr\},\smallskip\\ 
     \displaystyle\Hv(u\vert k)&\hspace{-2.5mm}=\displaystyle\frac{iK}{\pi}\Bigl\{\Erm\Bigl(\frac{u}{2}\Big\vert k\Bigr)-\frac{E}{K}\frac{u}{2}\Bigl\},
     \end{array}\right.,
\end{equation}
where $\Erm$ is \emph{Jacobi's epsilon function}: $\Erm(u)=\int_{0}^u \dn^2(v\vert k)\ud v$, see \cite[16.26.3]{AS}. For $k^2<0$, let 
\begin{equation}
\label{eq:definition_Hh_Hv_k2<0}
     \left\{\begin{array}{rr}
     \displaystyle\Hh(u\vert k)&\hspace{-2.5mm}=\Hh(k'u\vert k^*),\smallskip\\ 
     \displaystyle\Hv(u\vert k)&\hspace{-2.5mm}=\Hv(k'u\vert k^*).
     \end{array}\right.
\end{equation} 

Properties of these functions are presented below.
\begin{lem}
\label{lem:properties_Hh_Hv}   
The functions $\Hh$ and $\Hv$ admit jumps in the horizontal and vertical directions, respectively:
\begin{equation}
\label{eq:jumps_Hh_Hv}
     \left\{\begin{array}{l}
     \Hh(u+4K\vert k)-\Hh(u\vert k)=1,\smallskip\\
     \Hh(u+4i\Re K'\vert k)-\Hh(u\vert k)=0,
     \end{array}\right.\qquad
     \left\{\begin{array}{l}
     \Hv(u+4K\vert k)-\Hv(u\vert k)=0,\smallskip\\
     \Hv(u+4i\Re K'\vert k)-\Hv(u\vert k)=1.
     \end{array}\right.
     \end{equation}
In the fundamental rectangle $[0,4K]+[0,4i\Re K']$, the function $\Hh$ (resp.\ $\Hv$) has a simple pole, at $2i\Re K'$, with residue $\frac{2\Re K'}{\pi}$ (resp.\ $\frac{2iK}{\pi}$). Moreover,
\begin{equation*}
     \lim_{k\to 0}\Hh(u\vert k)=\frac{u}{2\pi},\qquad \lim_{k\to 0}\Hv(u\vert k)=0.
\end{equation*}   
The following addition formulas hold:
\begin{equation}
\label{eq:additionH}
     \Hh(v-u\vert k)=\Hh(v\vert k)-\Hh(u\vert k)+\frac{k^2 \Re K'}{\pi}
     \left\{\begin{array}{ll}
     \displaystyle\sn\Bigl(\frac{u}{2}\Big\vert k\Bigr)\sn\Bigl(\frac{v}{2}\Big\vert k\Bigr)\sn\Bigl(\frac{v-u}{2}\Big\vert k\Bigr) 
     &\text{if }0<k^2<1,\medskip\\
     \displaystyle(-{k'}^2) \sd\Bigl(\frac{u}{2}\Big\vert k\Bigr)\sd\Bigl(\frac{v}{2}\Big\vert k\Bigr)\sd\Bigl(\frac{v-u}{2}\Big\vert k\Bigr) &\text{if } k^2<0,
     \end{array}\right.
\end{equation}
\begin{equation}
\label{eq:additionV}
     \Hv(v-u\vert k)=\Hv(v\vert k)-\Hv(u\vert k)+\frac{ik^2 K}{\pi}
     \left\{\begin{array}{ll}
     \displaystyle\sn\Bigl(\frac{u}{2}\Big\vert k\Bigr)\sn\Bigl(\frac{v}{2}\Big\vert k\Bigr)\sn\Bigl(\frac{v-u}{2}\Big\vert k\Bigr) 
     &\text{if }0<k^2<1,\medskip\\
     \displaystyle(-{k'}^2) \sd\Bigl(\frac{u}{2}\Big\vert k\Bigr)\sd\Bigl(\frac{v}{2}\Big\vert k\Bigr)\sd\Bigl(\frac{v-u}{2}\Big\vert k\Bigr) &\text{if } k^2<0.
     \end{array}\right.
\end{equation}
\end{lem}

\begin{proof}
We first prove the lemma in the case $0<k^2<1$. All properties concerning $\Hh$ are proved in \cite[Lemmas~44 and 45]{BdTR1}. The statements for $\Hv$ follow similarly. 

A slightly different proof would consist in using a reformulation of $H$ and $V$ in terms of \emph{Jacobi's zeta function} $Z$:
\begin{equation}
\label{eq:expressions_Hh_Hv_Z}
     \Hh(u\vert k)=\frac{K'}{\pi}Z\Bigl(\frac{u}{2}\Big\vert k\Bigr)+\frac{u}{4K},\qquad \Hv(u\vert k)=\frac{iK}{\pi}Z\Bigl(\frac{u}{2}\Big\vert k\Bigr).
\end{equation}
(Equation \eqref{eq:expressions_Hh_Hv_Z} is a consequence of \eqref{eq:definition_Hh_Hv_k2>0} and \cite[17.4.28 and 17.3.13]{AS}.) Then we could use the numerous properties satisfied by $Z$ (see in particular the addition formula \cite[17.4.35]{AS}) to derive Lemma \ref{lem:properties_Hh_Hv}.

In the case $k^2<0$, we use the definition \eqref{eq:definition_Hh_Hv_k2<0} of $\Hh$ and $\Hv$, allowing to transfer all properties from 
the case $k^2>0$. In doing so we have to: consider $\Re K'$ in \eqref{eq:jumps_Hh_Hv} because in the case $k^2<0$ the 
quarter-period $K'$ is non-real, see \eqref{id:K*_imaginary_modulus}, and use the 
transformation of the $\sn$ function into the $\sd$ one under the dual transformation, see \cite[16.10]{AS}. 
%
\end{proof}

The Laplacian operator~\eqref{eq:Laplacian_operator} uses the function $\Arm$, which satisfies the following properties.
\begin{lem}[Lemma 44 in \cite{BdTR1}]
\label{cor:Armbis}
The function $\Arm(\cdot\vert k)$ is odd and satisfies the following identities:
\begin{align}
\bullet\ &\Arm(v-u\vert k)=\Arm(v\vert k)-\Arm(u\vert k)-k'\sc(u\vert k)\sc(v\vert k)\sc(v-u\vert k),\label{cor:Armbis:item3}\\
\bullet\ &\Arm(u+2K\vert k)=\Arm(u\vert k).\label{cor:Armbis:item1}
\end{align}
\end{lem}

\section{Some explicit integral computations}
\label{app:explicit_integrals}

We gather here computations of some contour integrals appearing in the
expression of the Kasteleyn operator, in the Fisher case of Section \ref{sec:Ising_dimers}.

\subsection{An important contour integral}
\label{app:second_proof_lemma_fisher}

The following result has been used when proving Theorem \ref{thm:KFmoins_un}.
\begin{lem}
\label{lem:integrale_fisher} 
One has
\begin{equation*}
     \frac{1}{2i\pi} \int_\Gamma \fs(u+2K)\fs(u)\ud u=-\frac{1}{k'},
\end{equation*}
where $\Gamma$ is a vertical contour on $\TT(k)$ winding once vertically on the
torus and crossing the horizontal axis in the interval
$[\alpha,\alpha+2K]=\{x:\alpha \leq x \leq \alpha+2K\}$
and $\fs(u)= \nc(\frac{u-\alpha}{2})$.
\end{lem}

On the rectangle, the contour $\Gamma$ is supposed to cross the horizontal axis
inside of the interval $[\alpha,\alpha+2K]$.
If the vertical contour crosses the horizontal axis in the other interval
$[\alpha+2K,\alpha(+4K)]$, the integral is equal to $+\frac{1}{k'}$, as it
corresponds to changing $\alpha$ into $\alpha+2K$ and
$\nc(\frac{u-(\alpha+4K)}{2})=-\nc(\frac{u-\alpha}{2})$.
Note further that Lemma \ref{lem:integrale_fisher} is independent of the choice of the angle $\alpha$ mod $4K$, and the integrand $\fs(u+2K)\fs(u)$ is.

\begin{proof}
First, it follows from the change of variable $u\to u+2K$ and the above-mentioned property of $\nc$ that
\begin{equation*}
     \frac{1}{2i\pi} \int_\Gamma \fs(u+2K)\fs(u)\ud u=-\frac{1}{2i\pi} \int_{\Gamma-2K} \fs(u+2K)\fs(u)\ud u=\frac{1}{2i\pi} \int_{\widetilde{\Gamma}} \fs(u+2K)\fs(u)\ud u,
\end{equation*}
where $\widetilde\Gamma$ is the contour $\Gamma-2K$ crossed in the opposite direction of $\Gamma$. Further, using the $4iK'$-periodicity of the integrand, we deduce that 
\begin{equation}
\label{eq:gamma_C_new_contour}
     \frac{1}{2i\pi} \int_\Gamma \fs(u+2K)\fs(u)\ud u=\frac{1}{2}\frac{1}{2i\pi} \int_{\C} \fs(u+2K)\fs(u)\ud u,
\end{equation}
where $\C$ is the closed contour $(\Gamma-2iK')\bigcup(\widetilde\Gamma-2iK')\bigcup S_1\bigcup S_2$, $S_1$ and $S_2$ being the horizontal segments joining $(\Gamma-2iK')$ and $(\widetilde\Gamma-2iK')$, see Figure \ref{fig:contour_C_gamma}.

\unitlength=0.5cm
\begin{figure}[ht]
\vspace{31mm}
\begin{center}
\begin{picture}(0,0)(0,0)
\put(-6,0){\line(1,0){12}}
\put(-6,0){\line(0,1){3}}
\put(6,0){\line(0,1){3}}
\put(-2,6){\line(1,0){4}}
\put(2,6){\line(1,0){4}}
\put(-6,6){\line(1,0){4}}
\put(-6,3){\line(0,1){3}}
\put(6,3){\line(0,1){3}}
\put(0.8,2.1){\textcolor{blue}{$\Gamma$}}
\put(-5.1,2.1){$\C$}
\put(-5,-0.18){{$\bullet$}}
\put(1,-0.18){{$\bullet$}}
\put(-6.18,-0.18){{$\bullet$}}
\put(-6.18,5.82){{$\bullet$}}
\put(5.82,-0.18){{$\bullet$}}
\put(-6.55,-0.8){$0$}
\put(5.74,-0.8){$4K$}
\put(-7.7,5.9){$4iK'$}
\put(-1.5,3.3){$S_1$}
\put(-1.5,-2.7){$S_2$}
\put(-5.05,-0.8){$\alpha$}
\put(0.95,-0.8){$\alpha+2K$}
\linethickness{0.2mm}
\put(-5.5,3){\vector(0,-1){3}}
\put(-5.5,0){\vector(0,-1){3}}
\put(-5.5,-3){\vector(1,0){3}}
\put(-2.5,-3){\vector(1,0){3}}
\put(0.5,-3){\vector(0,1){3}}
\put(0.5,0){\vector(0,1){3}}
\put(-5.5,-3){\line(1,0){6}}
\put(0.5,3){\vector(-1,0){3}}
\put(-2.5,3){\vector(-1,0){3}}
\linethickness{0.3mm}
\put(0.55,0){\textcolor{blue}{\vector(0,1){3}}}
\put(0.55,3){\textcolor{blue}{\vector(0,1){3}}}
\end{picture}
\end{center}
\bigskip\bigskip\bigskip
\caption{The contour $\C$ in \eqref{eq:gamma_C_new_contour} used in the proof of Lemma \ref{lem:integrale_fisher}.}
\label{fig:contour_C_gamma}
\end{figure}
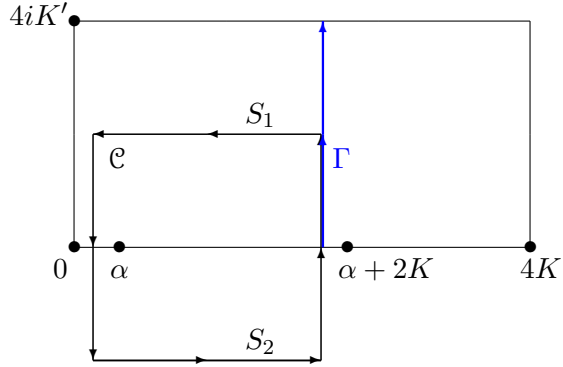
The main point is that the contour integral in the right-hand side of \eqref{eq:gamma_C_new_contour} can be computed with the residue theorem: the only pole (of order $1$) is at $\alpha$ and has residue $\frac{-2}{k'}$, see Table \ref{table:identities_Jacobi_function}. Lemma \ref{lem:integrale_fisher} follows.
\end{proof}

\subsection{Inverse Kasteleyn operator at an edge}

Here we compute the probability that a given edge $e=\xb\yb$ of the isoradial graph $\Gs$ with rhombus half-angle $\theta_e$
appears in the high temperature contour expansion of the Ising model. In terms of
dimers, it corresponds to the probability that the corresponding edge $\es=\vs_j(\xb)\vs_\ell(\yb)=\vs_j\vs_\ell$ belongs to 
a dimer configuration of $\GF$. By Theorem~\ref{thm:Gibbs_measure} this probability is given by 
$\PPdimer(\es)=\KF_{\vs_{j},\vs_{\ell}}\KF^{-1}_{\vs_{\ell},\vs_{j}}$.

\begin{lem}
\label{lem:inverse_Kasteleyn_edge}
One has
\begin{equation*}
     \KF_{\vs_{j},\vs_{\ell}}\KF^{-1}_{\vs_{\ell},\vs_{j}}=\frac{1}{2}-\frac{1-2\Hh(2\theta_e)}{2\cn\theta_e}.
\end{equation*}
\end{lem}

\begin{proof}
Instead of $\KF_{\vs_{j},\vs_{\ell}}\KF^{-1}_{\vs_{\ell},\vs_{j}}$ we compute $\KF_{\vs_{\ell},\vs_{j}}\KF^{-1}_{\vs_{j},\vs_{\ell}}$. Both quantities are obviously equal, but the second one happens to be more convenient when applying the results of Section \ref{sec:Ising_dimers}.

We start from the expression of $\KF^{-1}$ given in \eqref{equ:KF_inverseH} of Theorem \ref{thm:KFmoins_un} (note that by \eqref{eq:expression_C_x_y}, $C_{\vs_{j},\vs_{\ell}}=0$):
\begin{equation*}
     \KF^{-1}_{\vs_{j},\vs_{\ell}}=\frac{i k'}{8\pi}\oint_{\C_{\vs_{j},\vs_{\ell}}}\gs_{(\vs_{j},\vs_{\ell})}(u)\Hh(u)\ud u=
     \frac{i k'}{8\pi}\oint_{\C_{\vs_{j},\vs_{\ell}}}\fs_{\vs_{j}}(u+2K)\fs_{\vs_{\ell}}(u)\expo_{(\xb,\yb)}(u)\Hh(u)\ud u.
\end{equation*}
Using the harmonicity property \eqref{eq:harmonicity_g} of $\gs_{(\zs,\cdot)}(u)$ enables us to rewrite
\begin{align*}
\KF_{\vs_j,\vs_\ell}\fs_{\vs_{\ell}}(u)\expo_{(\xb,\yb)}(u)=
-[\KF_{\vs_j,{\ws_j}}\fs_{\ws_j}(u)+\KF_{\vs_j,{\ws_{j+1}}}\fs_{\ws_{j+1}}(u)].
\end{align*}
By definition of the function $\fs_{\vs_j}$, see~\eqref{equ:rewriting_fs}, we thus have
\begin{multline*}
\KF_{\vs_j,\vs_\ell}\fs_{\vs_{j}}(u+2K)\fs_{\vs_{\ell}}(u)\expo_{(\xb,\yb)}(u)=\\
-[\KF_{\vs_j,\ws_j}\fs_{\ws_j}(u+2K)+\KF_{\ws_{j+1},\vs_j}\fs_{\ws_{j+1}}(u+2K)]
[\KF_{\vs_j,{\ws_j}}\fs_{\ws_j}(u)+\KF_{\vs_j,{\ws_{j+1}}}\fs_{\ws_{j+1}}(u)].
\end{multline*}
Recalling that the orientation of the triangle $(\ws_j,\vs_j,\ws_{j+1})$ is admissible, we moreover have 
$\KF_{\vs_j,{\ws_j}}\KF_{\vs_j,\ws_{j+1}}=-\KF_{\ws_j,\ws_{j+1}}$, and since $\KF_{\vs_j,\vs_\ell}=-\KF_{\vs_\ell,\vs_j}$, we have
\begin{multline}
\label{eq:sum_four_terms}
\KF_{\vs_\ell,\vs_j}\fs_{\vs_{j}}(u+2K)\fs_{\vs_{\ell}}(u)\expo_{(\xb,\yb)}(u)=\\
\fs_{\ws_j}(u+2K)\fs_{\ws_j}(u)-\fs_{\ws_{j+1}}(u+2K)\fs_{\ws_{j+1}}(u)
+\KF_{\ws_j,\ws_{j+1}}(\fs_{\ws_{j+1}}(u+2K)\fs_{\ws_{j}}(u)-\fs_{\ws_j}(u+2K)\fs_{\ws_{j+1}}(u)).
\end{multline} 
With \eqref{eq:sum_four_terms} the quantity $\KF_{\vs_{\ell},\vs_{j}}\KF^{-1}_{\vs_{j},\vs_{\ell}}$ is a sum of four terms. The first two ones are computed thanks to Lemma \ref{lem:integrale_fisher}: with the choice of contour $\C_{\vs_j,\vs_\ell}$, we have
\begin{equation*}
\frac{i k'}{8\pi}\oint_{\C_{\vs_j,\vs_\ell}}[\fs_{\ws_j}(u+2K)\fs_{\ws_j}(u)-\fs_{\ws_{j+1}}(u+2K)\fs_{\ws_{j+1}}(u)]\Hh(u)\ud u= 2\frac{i k'}{8\pi}\frac{-2\pi i}{k'}=\frac{1}{2}.
\end{equation*}
To compute the last two terms in \eqref{eq:sum_four_terms}, namely
\begin{equation}
\label{eq:last_two_terms}
\frac{i k'}{8\pi}\oint_{\C_{\vs_j,\vs_\ell}}\KF_{\ws_j,\ws_{j+1}}[\fs_{\ws_{j+1}}(u+2K)\fs_{\ws_{j}}(u)-\fs_{\ws_{j}}(u+2K)\fs_{\ws_{j+1}}(u)]\Hh(u)\ud u,
\end{equation}
recall that by  \eqref{equ:def_angle}
\begin{align*}
\fs_{\ws_{j+1}}(u)=
\textstyle \nc(\frac{u-\alpha_{j+1}}{2})=
\KF_{\ws_j,\ws_{j+1}}\nc(\frac{u-\alpha_j-2\theta_e}{2} ),
\end{align*}
so that 
\begin{align*}
\textstyle
\KF_{\ws_j,\ws_{j+1}}\fs_{\ws_{j+1}}(u+2K)\fs_{\ws_{j}}(u)=
\nc(\frac{u+2K-\alpha_j-2\theta_e}{2} )\nc(\frac{u-\alpha_j}{2} ).
\end{align*}
We therefore focus on the term 
\begin{equation*}
     \int_{\C_{\vs_{j},\vs_{\ell}}}\nc\Bigl(\frac{u+2K-\alpha_j-2\theta_e}{2} \Bigr)\nc\Bigl(\frac{u-\alpha_j}{2} \Bigr)\Hh(u)\frac{\ud u}{2\pi i},
\end{equation*}
in which we set, without loss of generality, $\alpha_j=0$. This is equivalent to replace the function $\Hh(u)$ by $\Hh(u-\alpha_j)$, which is possible because both functions satisfy the same jump conditions stated in Lemma \ref{lem:properties_Hh_Hv}. There are three residues, at $2\theta_e$, $2K$ and $2i\Re K'$. We thus have
\begin{align}
     \int_{\C_{\vs_{j},\vs_{\ell}}}& \nc\Bigl(\frac{u+2K-2\theta_e}{2} \Bigr)\nc\Bigl(\frac{u}{2} \Bigr)\Hh(u)\frac{\ud u}{2\pi i}\nonumber\\&\qquad=\frac{-2}{k'}\frac{\Hh(2\theta_e)}{\cn\theta_e}+\frac{-2}{k'}\frac{\Hh(2K)}{\cn(2K-\theta_e)}+\frac{2\Re K'}{\pi}\nc(i\Re K'+K-\theta_e)\nc(i\Re K')\nonumber\\&\qquad=\frac{2}{k'}\frac{\Hh(2K)-\Hh(2\theta_e)}{\cn \theta_e},\label{eq:partial_conclusion_mixed_1}
\end{align}
since $\nc(i\Re K')=0$. The same reasoning as above gives
\begin{multline}
\label{eq:partial_conclusion_mixed_2}
     \int_{\C_{\vs_{j},\vs_{\ell}}}\nc\Bigl(\frac{u-2\theta_e}{2} \Bigr)\nc\Bigl(\frac{u+2K}{2} \Bigr)\Hh(u)\frac{\ud u}{2\pi i}\\=\frac{2}{k'}\frac{\Hh(2\theta_e-2K)-\Hh(0)}{\cn\theta_e}+\frac{2\Re K'}{\pi}\nc(i\Re K'+K)\nc(i\Re K'-\theta_e),
\end{multline}
which may be slightly simplified, using that $\Hh(0)=0$.
Thanks to \eqref{eq:partial_conclusion_mixed_1}, \eqref{eq:partial_conclusion_mixed_2} and Table~\ref{table:identities_Jacobi_function}, we obtain that \eqref{eq:last_two_terms} equals
\begin{equation*}
     \KF_{\vs_{\ell},\vs_{j}}\KF^{-1}_{\vs_{j},\vs_{\ell}}=\frac{1}{2}+\frac{\Hh(2\theta_e)-\Hh(2K)+\Hh(2\theta_e-2K)}{2\cn\theta_e}+\frac{\Re K' k^2}{2\pi}\sd\theta_e.
\end{equation*}
The addition formula \eqref{eq:additionH} results in
\begin{align*}
     \Hh(2\theta_e-2K)-&\Hh(2K)\\&=\Hh(2\theta_e-4K)-\frac{\Re K'k^2}{\pi}\left\{\begin{array}{ll}
\sn K \sn(\theta_e-K)\sn(\theta_e-2K) & \text{if } k^2>0\medskip\\
(-k'^2)\sd K \sd(\theta_e-K)\sd(\theta_e-2K)& \text{if } k^2<0\end{array}\right.\\
&=\Hh(2\theta_e)-1-\frac{\Re K'k^2}{\pi}\cn\theta_e\sd\theta_e.
\end{align*} 
The proof is complete.
\end{proof}

Using that $\lim_{k\to 0}\Hh(u)=\frac{u}{2\pi}$, see Lemma~\ref{lem:properties_Hh_Hv}, we find that
\begin{equation*}
     \lim_{k\to 0}\KF_{\vs_{j},\vs_{\ell}}\KF^{-1}_{\vs_{\ell},\vs_{j}}=\frac{1}{2}-\frac{\pi-2\theta_e}{2\pi\cos\theta_e}.
\end{equation*}
 This is in accordance with the computation of~\cite[Appendix A]{BoutillierdeTiliere:iso_gen} 
for the case $k=0$. Indeed, the dimer model
considered in the latter paper corresponds to complementary polygon configurations meaning that the probability \eqref{ex:proba_comput}
is $1$ minus the probability of the paper~\cite{BoutillierdeTiliere:iso_gen}.

\section{Proof of Lemma~\ref{lem:def_angles}}\label{app:proof_angles_4pi}

This section consists in the proof of Lemma~\ref{lem:def_angles} stating that the angles 
$(\overline{\alpha}_j(\xb))_{\xb\in\Vs,\,j\in\{1,\dots,d(\xb)\}}$ defined in Equations~\eqref{equ:def_angle} and~\eqref{eq:relative_definition_a_a'}
are indeed well defined mod $4\pi$.

\begin{proof}
We first need to check that angles around a rhombus corresponding to two neighboring decorations are well defined, 
see Figure~\ref{Fig:fig_GF_3}.

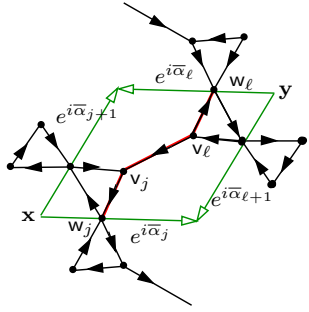
\begin{figure}[ht]
  \centering
\begin{overpic}[width=4cm]{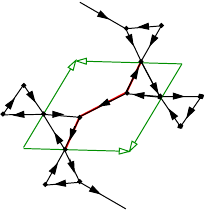}
\put(40,40){\scriptsize $\vs_j$}
\put(20,25){\scriptsize $\ws_j$}
\put(60,50){\scriptsize $\vs_\ell$}
\put(72,72){\scriptsize $\ws_\ell$}
\put(5,28){\scriptsize $\xb$}
\put(88,69){\scriptsize $\yb$}
\put(40,21){\scriptsize $e^{i\overline{\alpha}_j}$}
\put(48,73){\scriptsize $e^{i\overline{\alpha}_{\ell}}$}
\put(16,60){\scriptsize $e^{i\overline{\alpha}_{j+1}}$}
\put(66,33){\scriptsize $e^{i\overline{\alpha}_{\ell+1}}$}
\end{overpic}
\caption{Compatibility around a rhombus.}
\label{Fig:fig_GF_3}
\end{figure}

We want to prove that the following is equal to $0$ mod $4\pi$:
\begin{equation*}
(\overline{\alpha}_{\ell+1}-\overline{\alpha}_{\ell})-
(\overline{\alpha}_{j+1}-\overline{\alpha}_{j})+(\overline{\alpha}_{\ell}-\overline{\alpha}_{j})-
(\overline{\alpha}_{\ell+1}-\overline{\alpha}_{j+1}).
\end{equation*}
By definition of the angles within a decoration we have, mod $4\pi$,
\begin{equation*}
(\overline{\alpha}_{\ell+1}-\overline{\alpha}_{\ell})-(\overline{\alpha}_{j+1}-\overline{\alpha}_{j})=
\begin{cases}
0 & \text{if $\co(\ws_\ell,\ws_{\ell+1})+\co(\ws_j,\ws_{j+1})$ is even,}\\
2\pi& \text{if $\co(\ws_\ell,\ws_{\ell+1})+\co(\ws_j,\ws_{j+1})$ is odd}.
\end{cases}
\end{equation*}
By definition of the angles in neighboring decorations we have, mod $4\pi$,
\begin{equation*}
(\overline{\alpha}_{\ell}-\overline{\alpha}_{j})-(\overline{\alpha}_{\ell+1}-\overline{\alpha}_{j+1})=
\begin{cases}
0&\text{if $\co(\ws_{j+1},\vs_j,\ws_j)+\co(\ws_{\ell+1},\vs_\ell,\ws_\ell)$ is even,}\\
2\pi&\text{if $\co(\ws_{j+1},\vs_j,\ws_j)+\co(\ws_{\ell+1},\vs_\ell,\ws_\ell)$ is odd}.
\end{cases}
\end{equation*}
This implies that, mod $4\pi$, we have:
\begin{multline*}
(\overline{\alpha}_{\ell+1}-\overline{\alpha}_{\ell})-
(\overline{\alpha}_{j+1}-\overline{\alpha}_{j})+
(\overline{\alpha}_{\ell}-\overline{\alpha}_{j})-
(\overline{\alpha}_{\ell+1}-\overline{\alpha}_{j+1})\\
=
\begin{cases}
0&\text{if $\co(\ws_{j+1},\vs_j,\ws_j,\ws_{j+1})+\co(\ws_{\ell+1},\vs_\ell,\ws_\ell,\ws_{\ell+1})$ is even,} \\
2\pi& \text{if $\co(\ws_{j+1},\vs_j,\ws_j,\ws_{j+1})+\co(\ws_{\ell+1},\vs_\ell,\ws_\ell,\ws_{\ell+1})$ is odd}.
\end{cases}
\end{multline*}
But since the orientation of the graph is admissible, we have that $\co(\ws_{j+1},\vs_j,\ws_j,\ws_{j+1})$ and 
$\co(\ws_{\ell+1},\vs_\ell,\ws_\ell,\ws_{\ell+1})$ are odd, implying that the sum is even, thus concluding the proof for angles around a rhombus.

We now need to prove that when doing the inductive procedure around a cycle of $\GF$, we recover the same angle mod $4\pi$.
There are two types of cycles to consider: inner cycles of decorations and cycles arising from boundary of faces of $\Gs$. 

Consider a cycle of a decoration corresponding to a 
vertex $\xb$ of $\Gs$ of degree $d$, and let $n$ be the number of
edges of the inner cycle oriented clockwise. By definition of the angles, we have:
\begin{equation*}
\overline{\alpha}_{d+1}-
\overline{\alpha}_{1}=\sum_{j=1}^d (\overline{\alpha}_{j+1}-\overline{\alpha}_{j})\\
=\sum_{j=1}^d 2\overline{\theta}_{j}+2\pi n=2\pi(n+1).
\end{equation*}
Since the orientation of the edges is admissible, $n$ is odd, thus proving that $\overline{\alpha}_{d+1}-\overline{\alpha}_{1}=0\ [4\pi]$.

Now consider a cycle $C$ arising from the boundary 
$\xb_1,\dots,\xb_m$ of a face of $\Gs$, with vertices labeled in clockwise order, 
see also Figure~\ref{Fig:fig_GF_2}. Up to a relabeling 
of vertices of the decorations, this cycle can be written as
\begin{equation*}
     C=(\ws_{2}(\xb_1),\vs_2(\xb_1),\vs_{1}(\xb_2),\ws_{2}(\xb_2),\dots,\vs_{1}(\xb_1)).
\end{equation*}

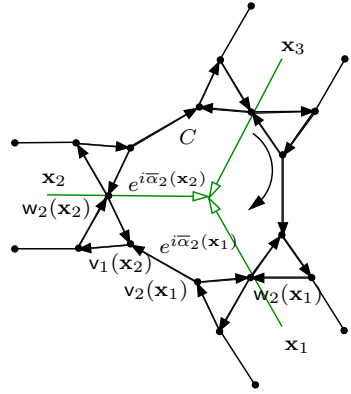
\begin{figure}[ht]
  \centering
\begin{overpic}[width=4.5cm]{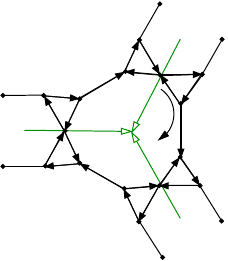}

 \put(62,23){\scriptsize $\ws_2(\xb_1)$}
 \put(29,24){\scriptsize $\vs_2(\xb_1)$}
 \put(20,31){\scriptsize $\vs_1(\xb_2)$}
 \put(3,45){\scriptsize $\ws_2(\xb_2)$}
 \put(70,10){\scriptsize $\xb_1$}
 \put(8,53){\scriptsize $\xb_2$}
 \put(69,87){\scriptsize $\xb_3$}
 \put(38,35){\scriptsize $e^{i\overline{\alpha}_2(\xb_1)}$}
 \put(30,51){\scriptsize $e^{i\overline{\alpha}_{2}(\xb_2)}$}
 \put(43,63){\scriptsize $C$}
\end{overpic}
\caption{Notation for a cycle $C$ arising from the boundary of a face of $\Gs$.\label{Fig:fig_GF_2}}
\end{figure}

Using the definition of the angles within a decoration and in neighboring ones we deduce that, mod $4\pi$, we have:
\begin{equation*}
\overline{\alpha}_{2}(\xb_j)-\overline{\alpha}_{2}(\xb_{j+1})=
\begin{cases}
-\pi-2\theta_1(\xb_j) &\text{if $\co(\ws_2(\xb_j),\vs_2(\xb_j),\vs_1(\xb_{j+1}),\ws_2(\xb_{j+1}))$ is odd,}\\
\pi-2\theta_1(\xb_j) &\text{if $\co(\ws_2(\xb_j),\vs_2(\xb_j),\vs_1(\xb_{j+1}),\ws_2(\xb_{j+1}))$ is even}.
\end{cases}
\end{equation*}
Let $n(C)$ denote the number of portions of the cycle $C$ where 
\begin{equation*}
     \co(\ws_2(\xb_j),\vs_2(\xb_j),\vs_1(\xb_{j+1}),\ws_2(\xb_{j+1}))
\end{equation*}
is odd.
Then, writing $\xb_{m+1}=\xb_1$, we have:
\begin{equation*}
     \alpha_2(\xb_{1})-\alpha_2(\xb_{m+1})=\sum_{j=1}^m \alpha_2(\xb_j)-\alpha_2(\xb_{j+1})\\
=\sum_{j=1}^m(\pi-2\theta_1(\xb_j))-2\pi n(C).
\end{equation*}
Since $\sum_{j=1}^m(\pi-2\theta_1(\xb_j))$ is the sum of angles at the center of the cycle, it is equal to $2\pi$. 
The orientation of the cycle being admissible, $n(C)$ is odd, thus concluding the proof.
\qedhere
\end{proof}

\bibliographystyle{abbrv}

\def\cprime{$'$}

\end{document}